\documentclass[10pt]{amsart}
\usepackage{eurosym}
\usepackage{amsfonts}
\usepackage{amsmath}
\usepackage{amssymb}
\usepackage{amsthm}
\usepackage[all]{xy}
\usepackage{cite}
\usepackage{bbm}
\usepackage{amscd}
\usepackage[left=2cm,right=2cm,bottom=3cm,top=3cm]{geometry}
\usepackage[hidelinks]{hyperref}

\newtheorem*{thmintro}{Theorem}

\newtheorem{theorem}{Theorem}[section]
\newtheorem{proposition}[theorem]{Proposition}
\theoremstyle{definition}
\newtheorem{lemma}[theorem]{Lemma}
\newtheorem{definition}[theorem]{Definition}
\newtheorem{corollary}[theorem]{Corollary}

\newtheorem{remark}[theorem]{Remark}

\newtheorem{example}[theorem]{Example}

\AtBeginDocument{   \def\MR#1{}
}
\input{tcilatex}

\begin{document}
\title[Applications of model theory to C*-dynamics]{Applications of model theory to
C*-dynamics}
\author[Eusebio Gardella]{Eusebio Gardella}
\address{Eusebio Gardella\\
Westf\"{a}lische Wilhelms-Universit\"{a}t M\"{u}nster, Fachbereich
Mathematik, Einsteinstrasse 62, 48149 M\"{u}nster, Germany}
\email{gardella@uni-muenster.de}
\urladdr{http://wwwmath.uni-muenster.de/u/gardella/}
\author[Martino Lupini]{Martino Lupini}
\address{Martino Lupini, Mathematics Department, California Institute of
Technology, 1200 East California Boulevard, Mail Code 253-37, Pasadena, CA
91125}
\email{lupini@caltech.edu}
\urladdr{http:n//www.lupini.org/}
\thanks{The first-named author was partially funded by SFB 878 \emph{Groups,
Geometry and Actions}, and by a postdoctoral fellowship from the Humboldt
Foundation. The second-named author was partially supported by the NSF Grant
DMS-1600186. Part of this work was carried out by the authors while visiting
the Institut Mittag-Leffler in occasion of the program ``Classification of
operator algebras: complexity, rigidity, and dynamics''. The authors
gratefully acknowledge the hospitality and the financial support of the
Institute.}
\dedicatory{}

\begin{abstract}
We initiate the study of compact group actions on C*-algebras from the
perspective of model theory, and present several applications to C*-dynamics. 
Firstly,
we prove that the continuous part of the central sequence algebra of a
strongly self-absorbing action is indistinguishable from the continuous part
of the sequence algebra, and in fact equivariantly isomorphic under the
Continuum Hypothesis. As another application, we present a unified approach
to several dimensional inequalities in C*-algebras, which is done through the
notion of order zero dimension for an (equivariant) *-homomorphism.
Finiteness of the order zero dimension implies that the dimension of the
target algebra can be bounded by the dimension of the domain. The dimension
can be, among others, decomposition rank, nuclear dimension, or Rokhlin
dimension. As a consequence, we obtain new inequalities for these quantities.

As a third application we obtain the following result: if a C*-algebra $A$
absorbs a strongly self-absorbing C*-algebra $D$, and $\alpha $ is an action
of a compact group $G$ on $A$ with finite Rokhlin dimension with commuting
towers, then $\alpha $ absorbs any strongly self-absorbing action of $G$ on $%
D$. This has a number of interesting consequences, already in the case of
the trivial action on $D$. For example, we deduce that $D$-stability passes 
from $A$ to the crossed product. Additionally, in many cases of interest, 
our result restricts the possible values of the
Rokhlin dimension to $0, 1$ and $\infty$, showing a striking parallel to the 
behavior of the nuclear dimension for simple C*-algebras. 
We also show that an action of a finite group with finite Rokhlin
dimension with commuting towers automatically has the Rokhlin property if
the algebra is UHF-absorbing.
\end{abstract}

\maketitle




\section{Introduction}

The use of (central) sequence algebras in the theory of operator algebras
has a long history, dating back to McDuff's characterization of factors
which absorb the hyperfinite II$_{1}$-factor with separable predual $R$, as
those whose central (W*-)sequence algebra contains a unital copy of $R$ \cite%
{mcduff_central_1970}. Applications in the context of C*-algebras are both
abundant and far-reaching, and they often appear in connection with
classification of C*-algebras. For instance, the fact that the central
(C*-)sequence algebra (with respect to a nonprincipal filter) of a Kirchberg
algebra is purely infinite and simple is a major cornerstone in the work of
Kirchberg and Phillips \cite{kirchberg_embedding_2000}, which is the
starting point of the classification of Kirchberg algebras; see \cite%
{kirchberg_classification_2000} and \cite{phillips_classification_2000}.

Another major application of central sequence algebras has been to the
theory of strongly self-absorbing C*-algebras \cite{toms_strongly_2007},
which have become a fundamental part of Elliott's classification programme
of nuclear C*-algebras. Indeed, the tight connections that strongly self
absorbing C*-algebras have with classification, have prompted a deeper study
of ultrapowers and (central) sequence algebras. In this context, the use of
model-theoretic methods has become predominant \cite%
{farah_model_2014,farah_model_2013,farah_model_2014-1,farah_model_2016,farah_countable_2013,eagle_saturation_2015}%
. The most prominent features of ultrapowers are model-theoretic in nature,
and include what model theorists usually refer to as \L os' theorem and
countable saturation. Even though relative commutants do not have a
satisfactory model-theoretic analog, it is shown in \cite%
{farah_relative_2015} that for a strongly self-absorbing C*-algebra, its
ultrapower and its relative commutant are indistinguishable, and in fact
isomorphic assuming the Continuum Hypothesis (CH).

Ultrapowers (and relative commutants) have also been a crucial tool in the
study of group actions on operator algebras. They have been used in the
classification of amenable group actions on the hyperfinite II$_{1}$-factor
by Connes \cite{connes_outer_1975}, Jones \cite{jones_actions_1980} and
Ocneanu \cite{ocneanu_actions_1985}. In their proofs, a crucial step is to
show that any outer action admits equivariant embeddings of matrix algebras
into its relative commutant, a condition that is now known as the Rokhlin
property. In the context of C*-algebras, relative commutants were used in
connection with the Rokhlin property for group actions in the work of 
Herman-Jones \cite{herman_period_1982}, Kishimoto \cite{kishimoto_rohlin_1995},
Izumi \cite{izumi_finite_2004}, Hirshberg-Winter \cite%
{hirshberg_rokhlin_2007}, and the first-named author \cite%
{gardella_compact_2015}. The study of Rokhlin dimension has also made
extensive use of these tools, for example in \cite{gardella_compact_2015}
and \cite{hirshberg_rokhlin_2016}, as well as the more recent work on
strongly self-absorbing actions \cite{szabo_strongly_2015}. As is clear
from these works, the use of sequence algebras in the equivariant setting
becomes even more delicate when the acting group is not discrete, since a
continuous action on an operator algebra may induce a discontinuous action
on its relative commutant. As such, equivariant (central) sequence algebras
are interesting objects whose systematic study is justified by their wide
application in the literature.

The present work takes up this task. For a given compact second countable
group $G$, we consider actions of $G$ on C*-algebras ($G$-C*-algebras) as
structures in the framework of continuous model theory. When the group $G$
is finite, one can regard a $G$-action as a usual metric structure by adding
a function symbol for every element of the group. This does not work for a
general compact group, since the canonical action on the ultrapower that one
obtains in this way is not always continuous; see Example~\ref{eg:NotCts}.
On the other hand, adding a sort for the group and enforcing uniform bounds
on the continuity moduli of an action would not capture the notion of
ultrapower of $G$-actions. The solution adopted in \cite{gardella_model_2017}%
, suggested by the theory of compact quantum groups and their actions on
C*-algebras, consists in replacing in the language for C*-algebras the sort
for the whole C*-algebras with several sorts for the \emph{isotropy
components }of the action, indexed by representations of $G$ on
finite-dimensional Hilbert spaces. This gives a language $\mathcal{L}_{G}^{%
\text{C*}}$, which has function and relations symbols corresponding to the
C*-algebra operations as well as function symbols for the restriction of the
*-homomorphism coding the action to the isotropy components. It is shown in 
\cite{gardella_model_2017} that $G$-C*-algebras form an axiomatizable class
in such a language, and explicit axioms are provided.

In this paper, we will consider $G$-C*-algebras as structures with respect
to several other languages. Such languages are very natural, as they
correspond to notions of morphisms other than *-homomorphisms---such as
completely positive contractive maps, or order zero completely positive
contractive maps---that are of crucial importance for the recent theory of
C*-algebras. There are important model-theoretic reasons to consider such
languages. Indeed, as the recent work on the model theory of C*-algebras has
shown \cite{farah_model_2016}, most of the properties of C*-algebras
considered in the C*-algebra literature can be captured model-theoretically.
As most maps that arise naturally in the applications are not elementary,
and often are not even *-homomorphisms, it is important to keep track of the
exact \emph{complexity }of formulas needed to describe C*-algebraic
properties, including which \emph{operations} are needed to describe them.
As it turns out, the multiplication symbol in many cases can be
dispensed of, and replaced with other predicates that capture the ordered
operator space structure, or the \textquotedblleft order
zero\textquotedblright\ structure of a given C*-algebra. This careful
analysis will make it apparent how various regularity property are
automatically preserved by several C*-algebraic and model-theoretic
constructions.

We present a number of applications to C*-dynamics in Section~4 and
Section~5. We focus mainly on strongly self-absorbing actions (in the sense
of \cite{szabo_strongly_2015}), actions with finite Rokhlin dimension (in
the sense of \cite{hirshberg_rokhlin_2015} and \cite{gardella_rokhlin_2014}%
), and general dimensional inequalities in C*-algebras. The main novelty in
this part is that we shift the attention from the actions themselves to the
study of equivariant maps between them; this is in the spirit of $KK$-theory
and other related theories. In particular, we consider equivariant order
zero maps between C*-dynamical systems; this is inspired in the notion of
weak containment for representations and measure-preserving actions of
countable groups, which are fundamental in modern ergodic theory and
representation theory.

The notion of strongly self-absorbing action has been recently introduced
and studied by Szab\'{o} in \cite%
{szabo_strongly_2015,szabo_strongly_2016}, where it is
shown that many familiar properties of strongly self-absorbing C*-algebras
have natural analogues for strongly self-absorbing actions. Building on this
work, in Section~4 we investigate the model-theoretic properties of strongly
self-absorbing actions. In particular, we show that the continuous part of
the central sequence algebra of a strongly self-absorbing action is
indistinguishable from the continuous part of the sequence algebra, and in
fact equivariantly isomorphic assuming CH, thus
generalizing results from \cite{farah_relative_2015}. We take the occasion
to remove an unnecessary assumption present in \cite{farah_relative_2015},
and observe that all the results hold for reduced products with respect to
an arbitrary countably incomplete filter, even without the assumption that
the corresponding reduced product be countably saturated. We also show that
the classification problem for strongly self-absorbing actions of a fixed
compact second countable group on C*-algebras is smooth in the sense of
Borel complexity theory. This is no longer the case for actions with
approximately inner half-flip, even if one restricts to actions on the Cuntz
algebra $\mathcal{O}_{2}$. Indeed, we observe that the relations of
conjugacy and cocycle conjugacy for $\mathbb{Z}/2\mathbb{Z}$-actions on $%
\mathcal{O}_{2}$ with approximately $\mathbb{Z}/2\mathbb{Z}$-inner half-flip
are complete analytic sets. Most of the results of this section admit
natural generalizations to the case of a locally compact (not necessarily
compact) second countable group $G$. This presents additional technical
difficulties, which can be overcome by considering a more general framework
than the usual framework for first order logic for metric structures.\ For
the sake of simplicitly, we will only consider the case when $G$ is compact.

Section~5 contains applications to dimensional inequalities in C*-algebras.
This is done through the notion of \emph{(equivariant) order zero dimension
(with and without commuting towers)} for an (equivariant) homomorphism. The
case of dimension zero corresponds to the notion of positive existential
embedding, which has been studied in \cite{goldbring_kirchbergs_2015} and,
under the name of sequentially split *-homomorphism, in \cite%
{barlak_sequentially_2016}. As an example, if $\alpha \colon G\rightarrow 
\mathrm{Aut}(A)$ is an action of a compact group $G$ on a C*-algebra $A$,
then the Rokhlin dimension of $\alpha $ is equal to the $G$-equivariant
order zero dimension of the factor embedding $\theta \colon A\rightarrow
C(G,A)$. As an application of the syntactic characterization of $G$%
-equivariant order zero dimension together with results from \cite%
{farah_model_2016}, we obtain the following result, which is new in the
non-unital case.

\begin{thmintro}
Let $G$ be a compact group, and $A$ be a $G$-C*-algebra $A$. Then%
\begin{equation*}
\dim _{\mathrm{nuc}}(A^{G})\leq \dim _{\mathrm{nuc}}(A\rtimes G)\leq
(\dim _{\mathrm{Rok}}(A)+1)(\dim _{\mathrm{nuc}}(A)+1)-1
\end{equation*}%
and%
\begin{equation*}
\mathrm{dr}(A^{G})\leq \mathrm{dr}(A\rtimes G)\leq (\dim _{\mathrm{Rok}%
}(A)+1)(\mathrm{dr}(A)+1)-1.
\end{equation*}
\end{thmintro}

We use results from the literature to give many other examples of
*-homomorphisms with finite order zero dimension which do not come from
group actions. Notable examples are the unital inclusions $\mathcal{O}%
_{\infty }\rightarrow \mathcal{O}_{2}$ and $\mathcal{Z}\rightarrow U$, where 
$U$ is any UHF-algebra of infinite type. As a consequence of our general
results, we recover and extend useful inequalities relating the nuclear
dimension, decomposition rank, and Rokhlin dimension of the $\mathcal{Z}$-
and $U$-stabilization of an arbitrary ($G$-)C*-algebra. Similar statements
hold for the $\mathcal{O}_{\infty }$- and $\mathcal{O}_{2}$-stabilizations,
and this allows us to recover a result from \cite{matui_decomposition_2014}:
the nuclear dimension of a Kirchberg algebra is at most 3. (The actual
dimension of Kirchberg algebras has been computed in \cite{bosa_covering_?}:
it is 1.) Some of the nuclear dimensional estimates that we derive here have
also been observed in \cite{barlak_rokhlin_2015}, while the estimates
involving the decomposition rank are new.

One of our main results requires that we first prove the following equivariant
generalization of the main result from \cite{dadarlat_trivialization_2008}, 
which is interesting in its own right.

\begin{thmintro}
Let $X$ be a compact metrizable space of finite covering dimension, let $G$
be a compact metrizable group, let $(D,\delta)$ 
be a strongly self-absorbing, unitarily regular $G$-C*-algebra, and let 
$(A,\alpha)$ be a separable, unital $G$-$C(X)$-algebra. If $A_x$ is $G$-equivariantly
$D$-absorbing, then there is a $C(X)$-linear $G$-isomorphism
\[(A,\alpha)\cong (D\otimes C(X), \delta\otimes \iota_{C(X)}).\]
\end{thmintro}

Combining the theorem above with our results related to order zero dimension, 
we prove the following. The
second assertion is a significant generalization of previous results from 
\cite{hirshberg_rokhlin_2015} and \cite{gardella_regularity_?}, which only
considered the case $D=\mathcal{Z}$.

\begin{thmintro}
Let $G$ be a compact group, let $D$ be a strongly self-absorbing C*-algebra,
let $A$ be a separable $D$-absorbing C*-algebra, let $\alpha \colon
G\rightarrow \mathrm{Aut}(A)$ be an action of $G$ on $A$ with finite Rokhlin
dimension with commuting towers, and let $\delta :G\rightarrow \mathrm{Aut}%
\left( D\right) $ be any strongly self-absorbing action (such as the trivial
action). Then $(A,\alpha )$ is $G$-equivariantly isomorphic to $\left(
D\otimes A,\alpha \otimes \delta \right) $. Furthermore, the fixed point
algebra $A^{G}$ and the crossed product $A\ltimes G$ are $D$-absorbing.
\end{thmintro}

Absorption of the trivial action on the Jiang-Su algebra is particularly
useful, since it opens the doors of a possible classification of actions
with finite Rokhlin dimension with commuting towers. Indeed, showing
absorption of well-behaved objects is a common feature in most of the
classification results for group actions. These aspects will be explored in
subsequent work.

We also deduce new Rokhlin dimension estimates for actions with finite
Rokhlin dimension with commuting towers of finite groups on $\mathcal{Z}$%
-absorbing C*-algebras, which imply that the possible values of the Rokhlin
dimension in this case are $0$ and $1$.

\begin{thmintro}
Let $G$ be a finite group, let $A$ be a C*-algebra, and let $\alpha \colon
G\rightarrow \mathrm{Aut}(A)$ with finite Rokhlin dimension with commuting
towers. Then $\alpha \otimes \mathrm{id}_{\mathcal{Z}}$ has Rokhlin
dimension at most $1$. If $A$ is $\mathcal{Z}$-absorbing, then $\alpha $ has
Rokhlin dimension at most $1$.
\end{thmintro}

This represents a satisfactory parallel with the $\{0,1,\infty \}$-type
behaviour that nuclear dimension and decomposition rank tend to have in the
noncommutative setting. It is also particularly satisfactory, since proving
finiteness of the Rokhlin dimension with commuting towers is a far easier
task than proving that the Rokhlin dimension is (at most) 1. In the
particular case when $A$ is a commutative unital C*-algebra $C\left(
X\right) $ for some compact Hausdorff space $X$, such a result can be seen
as a dynamical version of the main result of~\cite%
{tikuisis_decomposition_2014}, which states that \textrm{dr}$(C(X)\otimes 
\mathcal{Z})\leq 2$.

Finally, we also prove that finite Rokhlin dimension with commuting towers
implies the Rokhlin property for finite group actions on UHF-absorbing
C*-algebras.

\begin{thmintro}
Let $G$ be a finite group, let $A$ be an $M_{|G|^{\infty }}$-absorbing
C*-algebra, and let $\alpha \colon G\rightarrow \mathrm{Aut}(A)$ be an
action with finite Rokhlin dimension with commuting towers. Then $\alpha $
has the Rokhlin property. This in particular applies to Cuntz algebras of
the form $\mathcal{O}_{n|G|}$.
\end{thmintro}

Again, this is very relevant from a computational point of view: proving
directly that an action has the Rokhlin property is often challenging, and
there are not many tools available. On the other hand, Rokhlin dimensional
estimates are much easier to come by, particularly for finite groups. Having
access to the Rokhlin property is highly valuable, since it entails
classifiability of the action, and the structure of the crossed product is
extremely well-understood (see, for example, \cite{gardella_crossed_2014}).

We include an appendix, containing the relevant notions and results from
model theory that are used in this paper; see also the appendix of \cite%
{gardella_model_2017}. A quick introduction to logic for metric structures
can be found in \cite{ben_yaacov_model_2008}, and as it pertains to
C*-algebras in \cite{lupini_invitation_2017}, while \cite{farah_model_2016}
is a more complete reference for the model-theoretic study of C*-algebras.

The model-theoretic perspective is crucial to our approach, as it allows us
to isolate the semantic content of properties of $G$-C*-algebras and
equivariant embeddings, such as the Rokhlin property, Rokhlin dimension, $G$%
-equivariant sequentially split *-homomorphism (in the terminology of \cite%
{barlak_sequentially_2016,barlak_spatial_2017}), $G$-equivariant order zero
dimension (introduced here). In turn, this is the fundamental ingredient to
effortlessly deduce preservation results from the semantic characterizations
of regularity properties of C*-algebras obtained in \cite%
{goldbring_kirchbergs_2015,farah_model_2016}, subsuming, simplifying, and
generalizing many results from the literature. The realization that the
\textquotedblleft continuous part of the ultrapower\textquotedblright\ of a $%
G$-C*-algebra is just its ultrapower as an $\mathcal{L}_{G}^{\text{C*}}$%
-structure allows us to clarify its properties, including saturation and \L %
os' theorem, which are here deduced from general model-theoretic facts. This
subsumes and technically simplifies the proof of many particular instances
that had previously appeared in the literature. We then crucially use the
full strength and semantic content of saturation and \L os' theorem in the
proof of our main results. The model-theoretic study of $G$-C*-algebras,
including the notion of first-order theory, is also fundamental in our study
of strongly self-absorbing $G$-C*-algebras. Particularly, we show that the
first-order theory provides a complete invariant (up to isomorphism) for
such $G$-C*-algebras. This is the crucial step in our computation of the
Borel complexity of the classification problem for strongly self-absorbing $G
$-C*-algebras.

For the rest of the paper, $G$ will be a second countable \emph{compact}
group. We denote by $C(G)$ the unital C*-algebra of continuous,
complex-valued function on $G$. The multiplication operation on $G$ induces
a unital *-homomorphism $\Delta \colon C(G)\rightarrow C(G\times G)\cong
C(G)\otimes C(G)$ given by $\Delta ( f) ( s,t)=
f(st)$ for $f\in C(G)$ and $s,t\in G$. A unitary representation $\pi \in \mathrm{%
Rep}(G)$ on a finite-dimensional Hilbert space $\mathcal{H}$ defines the subspace 
of \emph{matrix units} for $\pi $.
\begin{equation*}
C(G)_{\pi }=\left\{ \langle \xi ,\pi (f)\eta \rangle \colon \xi ,\eta
\in \mathcal{H}\text{, }f\in C(G)\right\}.
\end{equation*}%

\subsubsection*{Acknowledgements}

We would like to thank Mauro Di Nasso for referring us to the notion of good
ultrafilter, and Bradd Hart for drawing our attention to the framework of
real-valued logic, and Yasuhiko Sato for electronic correspondence
concerning the proof of Theorem \ref{thm:InclZUHF} as well as for helpful
comments and remarks on our work. We are also grateful to G\'{a}bor Szab\'{o}
for his comments and for pointing out connections with the preprint \cite%
{hirshberg_rokhlin_2016}. Finally, we are in debt to Ilijas Farah, Isaac
Goldbring, Alexander Kechris, and Robin Tucker-Drob for several helpful
conversations concerning the present work.

\section{Languages for C*-algebras\label{Section:languages}}

\subsection{The ordered selfadjoint operator space language\label%
{Subsection:osos-language}}

An \emph{ordered selfadjoint operator space}, as defined in \cite%
{blecher_ordered_2007,russell_characterizations_2015}, is a matricially
normed and matricially ordered $\ast $-vector space that admits a
selfadjoint completely isometric complete order embedding into a C*-algebra.
Concretely, one can defined an ordered selfadjoint operator space as a
selfadjoint closed subspace of $B(H)$ with the inherited matricial norms,
matricial positive cones, and involution. Ordered selfadjoint operator
spaces have been abstractly characterized in \cite%
{werner_subspaces_2002,russell_characterizations_2015,russell_characterizations_2016}%
, and further studied in \cite%
{werner_multipliers_2004,blecher_ordered_2006,blecher_ordered_2007,ng_operator_2011,karn_adjoining_2005,karn_order_2011}%
. For ordered operator spaces $X$ and $Y$, we denote by $\mathrm{CPC}(X,Y)$
the set of all selfadjoint completely positive completely contractive linear
maps $X\rightarrow Y$. (Observe that in an ordered selfadjoint operator
space the matrix positive cones are not necessarily spanning. Therefore a
completely positive linear map on an ordered operator space is not
necessarily selfadjoint.)

An ordered selfadjoint operator space $X$ can be naturally seen as a
structure in the language $\mathcal{L}^{\mathrm{osos}}$ that contains

\begin{itemize}
\item sorts $M_{n}(X)$, with $n\in \mathbb{N}$, for the matrix
amplifications of the space $X$, with balls centered at the origin as
domains of quantification;

\item a sort for each finite-dimensional C*-algebra $F$, with balls centered
at the origin as domains of quantification;

\item function symbols for the vector space operations and the involution in 
$X$ and $F$;

\item predicate symbols for the norms in $M_{n}(X)$ and in $F$;

\item predicate symbols for the distance function from the cone of positive
elements in $M_{n}(X)$ and in $F$;

\item predicate symbols for the function $F^{k}\times X^{k}\rightarrow 
\mathbb{R}$ given by 
\begin{equation*}
(\overline{y},\overline{z})\mapsto \inf\limits_{t\in \mathrm{CPC}%
(F,X)}\max\limits_{j=1,\ldots ,k}\left\Vert t(y_{j})-z_{j}\right\Vert .
\end{equation*}
\end{itemize}

We call such a language the \emph{ordered selfadjoint operator space language%
} $\mathcal{L}^{\mathrm{osos}}$. Observe that the $\mathcal{L}^{\mathrm{osos}%
}$-terms can be seen as degree $1$ matrix *-polynomials without constant
terms. These are expressions of the form 
\begin{equation*}
\alpha _{1}^{\ast }x_{1}\beta _{1}+\cdots +\alpha _{n}^{\ast }x_{n}\beta
_{n}+\gamma _{1}^{\ast }x_{1}^{\ast }\delta _{1}+\cdots +\gamma _{n}^{\ast
}x_{n}^{\ast }\delta _{n}
\end{equation*}%
where $n$ is a positive integer, $\alpha _{j},\beta _{j},\gamma _{j},\delta
_{j}$, for $1\leq j\leq n$, and are scalar matrices It is clear that a
function between ordered selfadjoint operator spaces is an $\mathcal{L}^{%
\mathrm{osos}}$-morphism if and only if it is \emph{selfadjoint}, \emph{%
completely positive }and \emph{completely contractive}, and it is an $%
\mathcal{L}^{\mathrm{osos}}$-embedding if and only if it is a \emph{%
selfadjoint completely isometric complete order embedding}. In particular,
any C*-algebra can be seen as an $\mathcal{L}^{\mathrm{osos}}$-structure in
the obvious way, by considering its canonical matrix norms and matrix
positive cones. It is observed in \cite[Appendix C]%
{goldbring_kirchbergs_2015}, \cite[Section 3 and Section 5]{farah_model_2016}
that all the predicates above are definable in the usual language of
C*-algebras as considered in \cite{farah_model_2014,farah_model_2016}.

An operator system is a closed, selfadjoint subspace $X\subset A$ of a \emph{%
unital} C*-algebra that contains its unit. The \emph{operator system
language }$\mathcal{L}^{\mathrm{osy}}$ is obtained from the ordered operator
space language by adding a constant symbol for the unit in $X$. The $%
\mathcal{L}^{\mathrm{osy}}$-terms can be seen as degree $1$ matrix
*-polynomials \emph{with a constant term}.

\subsection{The order zero language\label{Subsection:order-zero}}

If $A,B$ are C*-algebras, we denote by $\mathrm{OZ}(A,B)$ the space of
completely positive contractive order zero maps $A\rightarrow B$. The \emph{%
order zero language }$\mathcal{L}^{\mathrm{oz}}$ for C*-algebras is obtained
from $\mathcal{L}^{\mathrm{osos}}$ by adding, for any finite-dimensional
C*-algebra $F$ and any $k\in \mathbb{N}$, a predicate symbol to be
interpreted as the function $F^{k}\times A^{k}\rightarrow \mathbb{R}$, given
by 
\begin{equation*}
(\overline{y},\overline{z})\mapsto \inf\limits_{t\in \mathrm{OZ}%
(F,X)}\max\limits_{j=1,\ldots ,k}\left\Vert t(y_{j})-z_{j}\right\Vert .
\end{equation*}%
It is proved in \cite[Section 5.2]{farah_model_2016} that such functions are
definable in the usual language of C*-algebras as considered in \cite%
{farah_model_2014,farah_model_2016}. This follows from the structure theorem
for completely positive contractive order zero maps \cite[Corollary 4.1]%
{winter_completely_2009} and stability of the relations defining cones of
finite-dimensional C*-algebras \cite[Section 3.3]{loring_lifting_1997}.

A C*-algebras can be seen as an\emph{\ }$\mathcal{L}^{\mathrm{oz}}$%
-structure in the obvious way. Let $A$ and $B$ be C*-algebras and let $%
f\colon A\rightarrow B$ be a function. Then $f$ is an $\mathcal{L}^{\mathrm{%
oz}}$-morphism if and only if $f$ is a completely positive contractive order
zero map.

\begin{remark}
\label{Remark:orthogonality}In the order-zero language one can express the
fact that a pair $( a_{1},a_{2}) $ of elements of a C*-algebra $A$ are
(almost) orthogonal. Indeed one can consider the canonical basis elements $(
e_{1},e_{2}) $ of $\mathbb{C}\oplus \mathbb{C}$ and the formula $\varphi (
e_{1},e_{2},x_{1},x_{2}) $ defined by%
\begin{equation*}
\inf_{t\in \mathrm{OZ}(\mathbb{C}\oplus \mathbb{C},A)}\max \left\{
\left\Vert t( e_{1}) -x_{1}\right\Vert ,\left\Vert t( e_{2})
-x_{2}\right\Vert \right\} \text{.}
\end{equation*}%
We have that if $\varphi ( e_{1},e_{2},a_{1},a_{2}) <\varepsilon $, then $%
\left\Vert a_{1}a_{2}\right\Vert <2\varepsilon $. Conversely, for every $%
\varepsilon >0$ there exists $\delta >0$ such that if $a_{1},a_{2}$ are
positive contractions such that $\left\Vert a_{1}a_{2}\right\Vert <\delta $,
then $\varphi ( e_{1},e_{2},a_{1},a_{2}) <\varepsilon $.
\end{remark}

\subsection{The C*-algebra language\label{Subsection:C*-language}}

The\emph{\ C*-algebra language} $\mathcal{L}^{\text{C*}}$ is obtained from $%
\mathcal{L}^{\mathrm{\mathrm{osos}}}$ by adding a function symbol for the
multiplication operation in $M_{n}(A)$, for every $n\in \mathbb{N}$.
Similarly, the \emph{unital C*-algebra language }$\mathcal{L}^{1\text{,C*}}$
is obtained from $\mathcal{L}^{\text{C*}}$ by adding a constant symbol for
the unit. Observe that the terms in $\mathcal{L}^{1\text{,C*}}$
(respectively, $\mathcal{L}^{\text{C*}}$) can be canonically identified with
matrix *-polynomials with (respectively, without) constant term. A\emph{\
matrix *-polynomial }is a linear combination of expressions of the form $%
X_{1}\cdots X_{n}$ where $X_{j}$, for $j=1,\ldots ,n$, is either a scalar
matrix, or $x$ or $y^{\ast }$ for some variable $x,y$. A function between
C*-algebras is an $\mathcal{L}^{\text{C*}}$-morphism ($\mathcal{L}^{1\text{%
,C*}}$-morphism) if and only if it is a (unital) *-homomorphism, and an $%
\mathcal{L}^{\text{C*}}$-embedding ($\mathcal{L}^{1\text{,C*}}$-embedding)
if and only if it is a (unital) injective *-homomorphism.

\begin{remark}
\label{Remark:existentially-definable}The following properties of
C*-algebras have been proved to be definable by a uniform family of
existential positive $\mathcal{L}^{\text{C*}}$-formulas---see Definition \ref%
{Definition:uniform-family}---in \cite[Theorem 2.5.1 and Theorem 5.7.3]%
{farah_model_2016}: real rank zero, stable rank at most $n$,
quasidiagonality, simplicity, being simple and purely infinite, being simple
and TAF, being abelian of real rank at most $n$. Considering unital
C*-algebras and positive existential $\mathcal{L}^{1\text{,C*}}$-formulas
gives approximate divisibility.
\end{remark}

\begin{remark}
\label{Remark:existentially-definable-separable}The following properties of
C*-algebras have been proved to be definable by a uniform family of
existential positive $\mathcal{L}^{\text{C*}}$-formulas among \emph{%
separable }C*-algebras in \cite[Theorem 2.5.1 and Theorem 5.7.3]%
{farah_model_2016}: being UHF; being AF; being $D$-absorbing for a given
strongly self-absorbing C*-algebra $D$; and being $\mathcal{K}$%
-absorbing---also called \emph{stable}---where $\mathcal{K}$ is the algebra
of compact operators.
\end{remark}

\subsection{The nuclear languages\label{Subsection:nuclear-language}}

The \emph{nuclear ordered selfadjoint operator space language} $\mathcal{L}^{%
\mathrm{osos}\text{-}\mathrm{\mathrm{nuc}}}$ is obtained from $\mathcal{L}^{%
\mathrm{osos}}$ by adding, for every $k\in \mathbb{N}$ and every
finite-dimensional C*-algebra $F$, a predicate symbol for the function $%
X^{k}\times F^{k}\rightarrow \mathbb{R}$ 
\begin{equation*}
(\overline{x},\overline{y})\mapsto \inf\limits_{s\in \mathrm{CPC}%
(X,F)}\max\limits_{j=1,\ldots ,k}\left\Vert s(x_{j})-y_{j}\right\Vert .
\end{equation*}%
It is proved in \cite[Section 5]{farah_model_2016} that such a function is
definable in the language of C*-algebras considered in \cite%
{farah_model_2014,farah_model_2016}.

In the proof of Lemma \ref{Lemma:completely positive contractive-morphism},
we will need the following version of the Choi-Effros lifting theorem: if $%
A,B$ are C*-algebras, $A$ is separable, $f\colon A\rightarrow B$ is a
nuclear completely positive contractive map, $E\subset A$ is a
finite-dimensional subspace, and $\varepsilon >0$, then there exists a
completely positive contractive map $\eta \colon \overline{f(A)}\rightarrow
A $ such that $f\circ \eta $ is the identity map on $\overline{f(A)}$, and $%
\left\Vert (\eta \circ f)(x)-x\right\Vert <\varepsilon \Vert x\Vert $ for
all $x\in E$. When $A$, $B$, and $f$ are unital, this is a consequence of
the Choi-Effros lifting theorem for operator systems; see \cite[Lemma 3.8
and Section 4.3]{choi_completely_1976}. The general case can be reduced to
the unital one by taking unitizations; see \cite[Proposition 2.2.1 and
Proposition 2.2.4]{brown_c*-algebras_2008}.

\begin{lemma}
\label{Lemma:completely positive contractive-morphism} Let $A$ and $B$ be
C*-algebras, and let $f\colon A\rightarrow B$ be a function. Consider the
following assertions:

\begin{enumerate}
\item $f$ is a nuclear completely positive contractive map;

\item $f$ is an $\mathcal{L}^{\mathrm{osos}\text{-}\mathrm{nuc}}$-morphism;

\item $f$ is a completely positive contractive map.
\end{enumerate}

Then (1)$\Rightarrow $(2)$\Rightarrow $(3), and they are all equivalent if
either $A$ or $B$ is nuclear.
\end{lemma}

\begin{proof}
The implication (1)$\Rightarrow $(2) uses the Choi-Effros lifting theorem as
stated above. The proof is the same as the proof that nuclearity passes to
quotients and that decomposition rank and nuclear dimension are
nonincreasing under quotients; see \cite[\S 2.9]{winter_covering_2003}, \cite%
[Section 3]{kirchberg_covering_2004}, and \cite[Proposition 2.3]%
{winter_nuclear_2010}. The implication (2)$\Rightarrow $(3) is obvious.
Finally, if either $A$ or $B$ are nuclear, then any completely positive
contractive map $f\colon A\rightarrow B$ is nuclear, which gives (3)$%
\Rightarrow $(1).
\end{proof}

The \emph{nuclear order zero language} $\mathcal{L}^{\mathrm{oz}\text{-}%
\mathrm{nuc}}$ and the \emph{nuclear C*-algebra language} $\mathcal{L}^{%
\text{C*-}\mathrm{nuc}}$, are defined as above starting from $\mathcal{L}^{%
\mathrm{oz}}$ and $\mathcal{L}^{\text{C*}}$, respectively. It follows from
Lemma \ref{Lemma:completely positive contractive-morphism} that any nuclear
completely positive contractive order zero map (nuclear *-homomorphism)
between C*-algebras is an $\mathcal{L}^{\mathrm{oz}\text{-}\mathrm{nuc}}$%
-morphism ($\mathcal{L}^{\text{C*-}\mathrm{nuc}}$-morphism), and the
converse holds for nuclear C*-algebras.

It is proved in \cite[Section 5]{farah_model_2016} that any predicate that
is definable in $\mathcal{L}^{\text{C*-}\mathrm{nuc}}$ is also definable in $%
\mathcal{L}^{\text{C*}}$. However, considering the larger language $\mathcal{%
L}^{\text{C*-}\mathrm{nuc}}$ gives a more generous notion of (positive)
existential formula.

\begin{remark}
\label{Remark:existentially-definable-nuclear}The following properties of
C*-algebras are definable by a uniform family of existential positive $%
\mathcal{L}^{\text{C*-}\mathrm{nuc}}$-formulas (see \cite[Section 5]%
{farah_model_2016}): nuclearity, having nuclear dimension at most $n$, and
having decomposition rank at most $n$.
\end{remark}

\subsection{Actions of groups on C*-algebras}
Denote by $\mathrm{Aut}(A)$ denote the group
of automorphisms of $A$, endowed with the topology of pointwise convergence.
An action of $G$ on a C*-algebra $A$ is a (strongly) continuous
group homomorphism $\alpha\colon G\rightarrow \mathrm{Aut}(A)$. 
An action of $G$ on $A$ can be regarded as an injective
nondegenerate *-homomorphism 
\begin{equation*}
\alpha\colon A\rightarrow C( G,A) \text{,}
\end{equation*}%
defined by $\alpha(a)(g)=\alpha_{g^{-1}}(a)$ for all $g\in G$ and all $a\in A$. 
With respect to the identification $C( G,A) \cong C(G)\otimes A$, such
a map satisfies the identity%
\begin{equation*}
( \Delta \otimes \mathrm{id}) \circ \alpha =( \mathrm{id}%
\otimes \alpha ) \circ \alpha \text{.}
\end{equation*}%
This identity characterizes the injective nondegenerate *-homomorphisms that
arise from actions as above. 

\begin{definition}
A \emph{$G$-C*-algebra} is a C*-algebra endowed with a distinguisued action of $G$.
\end{definition}

An irreducible representation $\pi $ of $G$ on
a finite-dimensional Hilbert space defines a subspace $A_{\pi }$ of $A$,
called $\pi $-\emph{isotypical component }or $\pi $-\emph{spectral subspace}%
, given by 
\begin{equation*}
\left\{ a\in A:\alpha \left( a\right) \in C(G)_{\pi }\otimes A\right\} \text{%
.}
\end{equation*}%
In the particular case when $\pi $ is the trivial representation, one
obtains the \emph{fixed point algebra }$A^{G}$.

It is explained in \cite[Section 3.4]{gardella_model_2017} how $G$%
-C*-algebras can be seen as structures in the logic for metric structures
with respect to the language $\mathcal{L}_{G}^{\text{C*}}$, which is the
language obtained from the language of C*-algebras $\mathcal{L}^{\text{C*}}$
by replacing the sort for the C*-algebra with sorts indexed by $\mathrm{Rep}%
(G)$, to be interpreted as the isotypical components. Furthermore, one adds
function symbols to be interpreted as the restriction of the action to the
isotypical components, regarded as maps $A_{\pi }\rightarrow C(G)_{\pi
}\otimes A_{\pi }$. (Observe that $C(G)_{\pi}$ is finite-dimensional.) Explicit 
\emph{axioms }for $G$-C*-algebras are provided in \cite[Section 3.4]%
{gardella_model_2017}, thus showing that $G$-C*-algebras form an $\mathcal{L}%
_{G}^{\text{C*}}$-axiomatizable class. For each language for C*-algebras $%
\mathcal{L}$ that we considered above, one can consider the corresponding $G$%
-\emph{equivariant version }$\mathcal{L}_{G}$, which can be obtained from $%
\mathcal{L}$ exactly as $\mathcal{L}_{G}^{\text{C*}}$ is obtained from $%
\mathcal{L}^{\text{C*}}$.

In the following, if $A$ is a C*-algebra, and $\mathcal{F}$ is a filter over
a set $I$, then we let $\prod_{\mathcal{F}}A$ be the corresponding reduced
product. When $A$ is a $G$-C*-algebra, $\prod_{\mathcal{F}}A$ is endowed
with the canonical coordinate-wise action of $G$. We let $\prod_{\mathcal{F}}^{G}A$ be the subalgebra of 
$a\in \prod_{\mathcal{F}}A$ such that the map $g\mapsto g^{\prod_{\mathcal{F}%
}A}a$ is continuous. This is a $G$-C*-algebra, which can be identified 
with the reduced product of $A$ with
respect to $\mathcal{F}$ when regarded as a structure in the langusge of $G$%
-C*-algebras $\mathcal{L}_{G}^{\text{C*}}$; see \cite[%
Proposition 3.12]{gardella_model_2017}.

It is worth noticing that $\prod_{\mathcal{F}}^{G}A$ is in general different
from $\prod_{\mathcal{F}}A$ for a $G$-C*-algebra $A$, even in the case when $%
\mathcal{F}$ is an ultrafilter over $\mathbb{N}$, as the next example shows.

\begin{example}
\label{eg:NotCts}The canonical inclusion of $C\left( G\right) $ inside $%
\prod_{\mathcal{U}}^{G}C\left( G\right) $ is surjective for any ultrafilter $%
\mathcal{U}$. However, the inclusion of $C(G)$ in the (nonequivariant)
C*-algebra ultrapower $\prod_{\mathcal{U}}C(G)$ is in general strict. For
example, if $G=\mathbb{T}$ and $u\in C(\mathbb{T})$ is the canonical unitary
generator, then the element $\left[ u^{n}\right] $ of $\prod_{\mathcal{U}}C(%
\mathbb{T})$ with representative sequence $(u^{n})_{n\in \mathbb{N}}$ does
not belong to $C(\mathbb{T})$, since the canonical action of $\mathbb{T}$
on $\prod_{\mathcal{U}}C(\mathbb{T})$ is not continuous at $[u^{n}]_{n\in 
\mathbb{N}}$. It follows that $C(\mathbb{T})=\prod_{\mathcal{U}}^{\mathbb{T}%
}C(\mathbb{T})$ is properly contained in $\prod_{\mathcal{U}}C(\mathbb{T})$.
\end{example}

\subsection{Languages for \texorpdfstring{$A$}{A}-bimodules\label%
{Subsection:bimodules}}

Let $A$ and $B$ be C*-algebras. Then $B$ is an $A$-bimodule if it is endowed
with linear maps $b\mapsto a\cdot b$ and $b\mapsto b\cdot a$ for $a\in A$,
satisfying $\max \left\{ \left\Vert a\cdot b\right\Vert ,\left\Vert b\cdot
a\right\Vert \right\} \leq \max \left\{ \left\Vert a\right\Vert ,\left\Vert
b\right\Vert \right\} $ as well as the natural associativity requirements.
When $A,B$ are $G$-C*-algebras, then we say that $B$ is a \emph{$G$%
-equivariant $A$-bimodule} if it is an $A$-bimodule satisfying $%
(g^{A}a)\cdot (g^{B}b)=g^{B}(a\cdot b)$, and $(g^{B}b)\cdot
(g^{A}a)=g^{B}(b\cdot a)$ for all $a\in A$, $b\in B$ and $g\in G$. If $%
f\colon A\rightarrow B$ is a $G$-equivariant *-homomorphism, then it induces
a canonical $G$-equivariant $A$-bimodule structure on $B$, defined by $%
a\cdot b:=f(a)b$ and $b\cdot a=bf(a)$ for $a\in A$ and $b\in B$.

We let $\mathcal{L}^{\text{C*,}A\text{-}A}$ be the language obtained from $%
\mathcal{L}^{\text{C*}}$ by adding symbols for the $A$-bimodule structure.
Similar definitions apply to the other languages for C*-algebras considered
above. The interpretation of an $\mathcal{L}^{\text{C*,}A\text{-}A}$-formula
in a $A$-bimodule is defined in the obvious way.

\subsection{The Kirchberg language\label{Subsection:Kirchberg}}

Fix a C*-algebra $A$. In this subsection, we define the Kirchberg language $\mathcal{L}%
^{K}(A)$, which fits into the
more flexible setting described in Subsection \ref{Subsection:general}. This
language is obtained from $\mathcal{L}^{\text{C*}}$ by replacing the symbols
for the matrix norms with pseudometric symbols $d_{F}$ for every finite set $%
F$ in the unit ball of $A$. The distinguished collection $t_{A}^{c}(x)$ of
positive quantifier-free conditions that is part of the language $\mathcal{L}%
^{K}(A)$ consists of the conditions $\max_{a\in F}\left\Vert
ax-xa\right\Vert =0$ for every finite subset $F$ of the unit ball of $A$.

One can regard $A$ as an $\mathcal{L}^{K}(A)$-structure by interpreting $%
d_{F}$ on $M_{n}( A) $ as the pseudometric 
\begin{equation*}
( x,y) \mapsto \max_{a\in M_{n}(F)}\left\Vert a( x-y) \right\Vert \text{.}
\end{equation*}
Suppose that $\mathcal{U}$ is an ultrafilter. Then the reduced power of $A$
as an $\mathcal{L}^{K}(A)$-structure is equal to the Kirchberg invariant $F_{%
\mathcal{U}}(A)$ as introduced by Kirchberg in \cite{kirchberg_central_2006}%
; see also \cite{ando_non-commutativity_2015}. Considering reduced powers
instead of ultrapowers yields the generalization of the Kirchberg invariant
to arbitrary filters considered in \cite%
{szabo_strongly_2015,barlak_sequentially_2016}. In the following, we denote
by $t_{A}^{c}( x_{1},\ldots ,x_{n}) $ the type $t_{A}^{c}( x_{1}) \cup
\cdots \cup t_{A}^{c}( x_{n}) $. If $A$ is \emph{unital}, then $F_{\mathcal{F%
}}( A) $ is equal to $A^{\prime }\cap \prod_{\mathcal{F}}A$.

Let $\kappa $ be an infinite cardinal that is larger than the density
character of $A$, and let $\mathcal{F}$ be a countably incomplete $\kappa $%
-good filter. (When $A$ is separable, one can take any countably incomplete
ultrafilter.) Considering an approximate unit for $A$ shows that $F_{%
\mathcal{F}}(A)$ is unital. Let $t(x_{1},\ldots ,x_{n})$ be a positive
primitive quantifier free $\mathcal{L}^{1\text{,C*}}$-type. The
corresponding \emph{multiplier }$\mathcal{L}^{K}(A)$-\emph{type }$%
t_{A}^{m}(x_{1},\ldots ,x_{n})$ is defined as follows. Any condition in $t(%
\overline{x})$ should be replaced with all the conditions obtained by
substituting every occurrence of a basic formula of the form $\left\Vert 
\mathfrak{p}(\overline{x})\right\Vert $, for some *-polynomial $\mathfrak{p}$
with constant term, with the basic formula $\left\Vert b\ \mathfrak{p}(%
\overline{x})\right\Vert $, where $b$ is some element of the unit
ball of $A$.

\begin{remark}
\label{Remark:F}It follows from Remark~\ref{rmk:charactPosQtfFreeTyp} that
the following statements are equivalent:

\begin{enumerate}
\item $t(\overline{x})$ is realized in $F_{\mathcal{F}}(A)$,

\item $t^{m}(\overline{x})$ is realized in $F_{\mathcal{F}}(A)$,

\item $t^{m}(\overline{x})$ is approximately realized in $F_{\mathcal{F}}(A)$%
,

\item $t^{m}(\overline{x})\cup t_{A}^{c}(\overline{x})$ is approximately
realized in $A$.
\end{enumerate}

Furthermore $F_{\mathcal{F}}(A)$ is positively quantifier-free $\mathcal{L}%
^{1\text{,\textrm{C*}}}$-$\kappa $-saturated. When $\mathcal{U}$ is an
ultrafilter, $F_{\mathcal{U}}(A)$ is quantifier-free $\mathcal{L}^{1\text{,%
\textrm{C*}}}$-$\kappa $-saturated.
\end{remark}

Various results from \cite{kirchberg_central_2006} can be seen as
consequences of Remark \ref{Remark:F}.

Suppose now that $A$ is a $G$-C*-algebra. Then one can consider $A$ as an $%
\mathcal{L}_{G}^{K}(A)$-structure. In this case, the reduced power of $A$ as
an $\mathcal{L}_{G}^{K}(A)$-structure with respect to a filter $\mathcal{F}$%
---which we denote by $F_{\mathcal{F}}^{G}(A)$---recovers the equivariant
version of the Kirchberg invariant considered in \cite%
{szabo_strongly_2015,barlak_sequentially_2016}. Again, the following
proposition follows from the general remarks of Subsection \ref%
{Subsection:general}. If $t(\overline{x})$ is a positive primitive
quantifier free $\mathcal{L}_{G}^{1\text{,C*}}$-type, then the corresponding
multiplier $\mathcal{L}_{G}^{K}(A)$-type $t_{A}^{m}(\overline{x})$ can be
defined as above.

\begin{proposition}
\label{Proposition:F}Suppose that $A$ is a $G$-C*-algebra, $\kappa $ is a
cardinal larger than the density character of $A$, $\mathcal{F}$ is a
countably incomplete $\kappa $-good filter, and $t(\overline{x})$ is a
positive primitive quantifier-free $\mathcal{L}_{G}^{1\text{,C*}}$-type.
Then $F_{\mathcal{F}}^{G}(A)$ is a unital $G$-C*-algebra, and the following
statements are equivalent:

\begin{enumerate}
\item $t(\overline{x})$ is realized in $F_{\mathcal{F}}^{G}(A)$,

\item $t^{m}(\overline{x})$ is realized in $F_{\mathcal{F}}^{G}(A)$,

\item $t^{m}(\overline{x})$ is approximately realized in $F_{\mathcal{F}%
}^{G}(A)$,

\item $t(\overline{x})\cup t_{A}^{c}(\overline{x})$ is approximately
realized in $A$.
\end{enumerate}

Furthermore $F_{\mathcal{F}}(A)$ is positively quantifier-free $\mathcal{L}%
_{G}^{1\text{,\textrm{C*}}}$-$\kappa $-saturated.
\end{proposition}

Similar conclusions hold if one replaces filters with ultrafilters, and
positive primitive quantifier free types with arbitrary quantifier free
types.


\section{Strongly self-absorbing \texorpdfstring{$G$}{G}-C*-algebras\label%
{Section:ssa}}

In this section, we exhibit some applications of model theory to strongly
self-absorbing actions on C*-algebras, as introduced and studied in \cite%
{szabo_strongly_2015,szabo_strongly_2016}. We regard $G$-C*-algebras as
structures in the language of $G$-C*-algebras $\mathcal{L}_{G}^{\text{C*}}$.
An $\mathcal{L}_{G}^{\text{C*}}$-morphism between $G$-C*-algebras is a $G$%
-equivariant *-homomorphism, and an $\mathcal{L}_{G}^{\text{C*}}$-embedding
is an injective $G$-equivariant *-homomorphism. If $A$ and $B$ are $G$%
-C*-algebras, then we denote by $A\otimes B$ the \emph{minimal} tensor
product of $A$ and $B$ endowed with the continuous $G$-action defined by $%
g^{A\otimes B}(a\otimes b)=(g^{A}a)\otimes (g^{B}b)$.

\subsection{Positively \texorpdfstring{$\mathcal{L}_{G}^{\text{C*}}$-}{}%
existential injective *-homomorphisms\label{Subsection:seq-split}}

An injective *-homomorphism $\theta \colon A\rightarrow M$ between \emph{%
separable} $G$-C*-algebras is $G$-equivariantly sequentially split, in the
sense of \cite[Definition 3.3]{barlak_sequentially_2016}, if and only if it
is positively $\mathcal{L}_{G}^{\text{C*}}$-existential, as defined in
Subsection \ref{Subsection:existential-embedding}. For arbitrary $G$%
-C*-algebras, the notion of positively $\mathcal{L}_{G}^{\text{C*}}$%
-existential injective *-homomorphism is more generous than being $G$%
-equivariantly sequentially split.

In the case of a compact group $G$, Theorem \ref%
{Theorem:positive-existential-embedding} applied to $G$-C*-algebras recovers 
\cite[Lemma 3.6 and Corollary 3.7]{barlak_sequentially_2016}. In the case of
compact $G$, Lemma 2.3, Corollary 2.4, Proposition 2.5, Proposition 2.9,
Proposition 3.8 and Corollary 3.17 of \cite{barlak_sequentially_2016} are
then an immediate consequence of the definition of positively $\mathcal{L}%
_{G}^{\text{C*}}$-existential injective *-homomorphism; see Proposition \ref%
{Proposition:limit}. Proposition 3.11 of \cite{barlak_sequentially_2016} is
a particular instance of \cite[Proposition A.33]{gardella_model_2017}, since
the fixed point algebra $A^{G}$ of a $G$-C*-algebra is positively
existentially definable. By appropriately choosing the functor, one can also
see that Proposition \ref{Proposition:functor} has as particular instances
the following results from \cite{barlak_sequentially_2016}: (I), (II), (IV)
of Theorem 2.10, Proposition 3.9, Proposition 3.12, Corollary 3.14,
Corollary 3.15, Proposition 3.16.

It follows from Proposition \ref{Proposition:definable-class} that if $A,B$
are C*-algebras and $f\colon A\rightarrow B$ is a positively $\mathcal{L}^{%
\text{C*}}$-existential injective *-homomorphism, then $A$ has any of the
properties listed in Remark \ref{Remark:existentially-definable} or\ Remark %
\ref{Remark:existentially-definable-separable}, whenever $B$ does. The same
assertion holds for any of the properties listed in Remark \ref%
{Remark:existentially-definable-nuclear} when $B$ is nuclear. In particular,
this observation recovers (1), (2), (3), (4), (5), (7), (11), the first part
of (12), the first half of (14), and (16) of \cite[Theorem 2.11]%
{barlak_sequentially_2016}. Other preservation results have been obtained in 
\cite%
{gardella_classification_2014,gardella_classification_2014-1}%
. We present here an additional preservation result does not seem to follow
from the results mentioned above. Recall the definition of real rank from 
\cite[Definition V.3.2.1]{blackadar_operator_2006}.

\begin{proposition}
Suppose that $A,B$ are C*-algebras and $f\colon A\rightarrow B$ is a
positively $\mathcal{L}^{\text{C*}}$-existential injective *-homomorphism.
If $B$ has real rank at most $n$, then $A$ has real rank at most $n$.
\end{proposition}

\begin{proof}
Without loss of generality, we can assume that $A,B$ are separable. After
unitizing, we can assume that $A,B$ are unital, $f$ is unital, and $f$ is a
positively $\mathcal{L}^{\text{C*,}1}$-existential injective *-homomorphism.
We identify $A$ with its image under $f$. Fix selfadjoint elements $%
a_{0},\ldots ,a_{n}\in A$ and $\varepsilon >0$. Since $B$ has real rank at
most $n$, there exist selfadjoint elements $b_{0},\ldots ,b_{n}\in B$ such
that $b_{0}^{2}+\cdots +b_{n}^{2}$ is invertible, and $\left\Vert
a_{i}-b_{i}\right\Vert <\varepsilon $ for every $i=0,1,\ldots ,n$. Since the
inclusion $A\subset B$ is a positively $\mathcal{L}^{\text{C*,}1}$%
-existential *-homomorphism, we can conclude that there exist selfadjoint
elements $c_{0},\ldots ,c_{n}\in A$ such that $c_{0}^{2}+\cdots +c_{n}^{2}$
is invertible, and $\left\Vert c_{i}-a_{i}\right\Vert <\varepsilon $ for
every $i=0,1,\ldots ,n$.
\end{proof}

\subsection{Commutant existential theories\label{Subsection:commutant-theory}%
}

The notion of weak containment and weak equivalence are introduced in the
general setting of logic for metric structures in\ Subsection \ref%
{Subsection:existential}. In this section we consider, in the case of $G$%
-C*-algebras, the natural \emph{commutant }analogs of such notions. Suppose
that $A,B$ are $G$-C*-algebras. We say that $A$ is \emph{commutant
positively weakly contained }in $B$ if for some (equivalently, any) cardinal 
$\kappa $ larger than the density character of $A$ and $B$ and for some
(equivalently, any) countably incomplete $\kappa $-good filter $\mathcal{F}$
on has that every $\mathcal{L}_{G}^{1\text{,C*}}$-type $t$ that is realized
in $F_{\mathcal{F}}^{G}(A)$ is also realized in $F_{\mathcal{F}}^{G}(B)$.
Equivalently, for any unital $G$-C*-subalgebra $C$ of $F_{\mathcal{F}}(A)$
of density character less than $\kappa $ there exists a $G$-equivariant
injective unital *-homomorphism from $C$ to $F_{\mathcal{F}}^{G}(B)$. If $A$
is unital, then $A$ is commutant positively weakly contained in $B$ if and
only if there exists a unital *-homomorphism from $A$ to $F_{\mathcal{F}%
}^{G}(B)$ for any filter $\mathcal{F}$ as above. A syntactic
characterization of commutant positively weak containment can be obtained
using Proposition \ref{Proposition:F}.

Suppose that $A$ is a $G$-C*-algebra, and let $t_{A}^{c}(\overline{x})$ be
the collection of conditions $\max\nolimits_{j=1,\ldots ,n}\left\Vert
x_{j}a-ax_{j}\right\Vert \leq 0$, for $a\in A$. Recall that if $t(\overline{x%
})$ is a positive (primitive) quantifier-free $\mathcal{L}_{G}^{1\text{,C*}}$%
-type, then $t_{A}^{m}(\overline{x})$ denotes the positive (primitive)
quantifier-free $\mathcal{L}_{G}^{K}(A)$-type obtained from $t(\overline{x})$
by replacing every occurrence of $\left\Vert \mathfrak{p}(\overline{x}%
)\right\Vert $ for some $G$-*-polynomial $\mathfrak{p}$ with $\left\Vert b{}%
\mathfrak{p}(\overline{x})\right\Vert $ where $b$ is an arbitrary element of 
$A$ of norm at most $1$. We then have that $A$ is commutant positively
weakly contained in $B$ if and only if, for any positive primitive
quantifier-free $\mathcal{L}_{G}^{1\text{,C*}}$-type $t(\overline{x})$, $%
t_{A}^{m}(\overline{x})\cup t_{A}^{c}(\overline{x})$ is approximately
satisfied in $A$ if and only if $t_{B}^{m}(\overline{x})\cup t_{B}^{c}(%
\overline{x})$ is approximately satisfied in $B$.

Two C*-algebras are\emph{\ commutant positively weakly equivalent} if they
are each commutant positively weakly contained in the other. For unital
nuclear $G$-C*-algebras, the following characterization of commutant weak $%
\mathcal{L}_{G}^{1\text{,C*}}$-containment follows from the Choi-Effros
lifting theorem.

\begin{proposition}
\label{Proposition:nuclear}Suppose that $A$ is a unital nuclear $G$%
-C*-algebra, and $B$ is a $G$-C*-algebra. Then $A$ is commutant positively
weakly contained in $B$ if and only if for any separable nuclear $G$%
-invariant unital C*-subalgebra $A_{0}\subset A$ and separable $G$%
-C*-subalgebra $B_{0}\subset B$, there exists a sequence $(\phi _{n})_{n\in 
\mathbb{N}}$ of completely positive contractive maps $\phi _{n}\colon
A_{0}\rightarrow B$ such that, for every compact subset $K\subset G$, for
every $x,y\in A_{0}$, and for every $b\in B_{0}$, we have that 
\begin{equation*}
\lim_{n\rightarrow +\infty }\left\Vert b(\phi_n (x)\phi_n (y)- \phi_n
(xy))\right\Vert =0,  \ \ \lim_{n\rightarrow +\infty }\left\Vert \phi_n (x)b-b\phi_n (x)\right\Vert =0
\end{equation*}%
\begin{equation*}
\lim_{n\rightarrow +\infty }\left\Vert b\phi_n (1)-b\right\Vert =0, \ \ \mbox{ and } \ \ 
\lim_{n\rightarrow +\infty }\max\limits_{g\in K}\left\Vert b(\phi_n
(g^{A}x)-g^{B}\phi_n (x))\right\Vert =0\text{.}
\end{equation*}
\end{proposition}

The notions of \emph{commutant weak containment} and commutant existential
theory are defined analogously, considering arbitrary (not necessarily
positive primitive) quantifier-free $\mathcal{L}_{G}^{1\text{,C*}}$-types.

\subsection{Space of separable nuclear \texorpdfstring{$G$}{G}-C*-algebras
and smooth classification}

We now observe that there exists a natural standard Borel space of separable 
\emph{nuclear }$G$-C*-algebras. Indeed, by Kirchberg's nuclear embedding
theorem \cite[Theorem 6.3.12]{rordam_classification_2002}, any separable
nuclear C*-algebra is *-isomorphic to the range of a conditional expectation
on $\mathcal{O}_{2}$. Given such a conditional expectation $E$, set $A=E(%
\mathcal{O}_{2})$. Write $\mathrm{CPC}(\mathcal{O}_{2})$ for the semigroup
of all completely positive contractive maps of $\mathcal{O}_{2}$ into
itself, endowed with the topology of pointwise convergence in norm. Given an
action $\alpha \colon G\rightarrow \mathrm{Aut}(A)$, one can define a
continuous function $\rho _{\alpha }\colon G\rightarrow \mathrm{CPC}(%
\mathcal{O}_{2})$ by $\rho _{\alpha }(g):=\alpha _{g}\circ E$. Then

\begin{enumerate}
\item $\rho _{\alpha}( gh) =\rho _{\alpha}( g) \circ \rho _{\alpha}( h) $
for every $g,h\in G$, and

\item $\rho _{\alpha}( 1) =E$.
\end{enumerate}

Conversely, any pair $( E,\rho) $, consisting of a conditional expectation $%
E\colon \mathcal{O}_2\to\mathcal{O}_2$ and a continuous function $\rho\colon
G\to \mathrm{CPC}(\mathcal{O}_2)$ satisfying (1) and (2), arises from a
continuous action of $G$ on the range of $E$, as described above.

Observe that $\mathrm{CPC}( \mathcal{O}_{2}) $ is a Polish space when
endowed with the topology of pointwise convergence. Similarly, the space $%
\mathrm{Exp}( \mathcal{O}_{2}) $ of conditional expectations defined on $%
\mathcal{O}_{2}$ is a Polish space when endowed with the topology of
pointwise convergence. (This can be seen for instance by observing that
conditional expectations onto a given C*-subalgebra are precisely the
idempotent maps of norm $1$ mapping onto that C*-subalgebra \cite[Theorem
II.6.10.2]{blackadar_operator_2006}.)

The space $G$\textrm{-}$\mathrm{C}^{\ast }\mathrm{ALG}$ of pairs $(E,\rho )$
arising from a continuous action of $G$ on the image of $E$, is a $G_{\delta
}$ subset of the space $\mathrm{Exp}(\mathcal{O}_{2})\times \mathrm{CPC}(%
\mathcal{O}_{2})$, hence a Polish space with the induced topology; see \cite[%
Theorem 3.11]{kechris_classical_1995}. We will regard $G$\textrm{-}$\mathrm{C%
}^{\ast }\mathrm{ALG}$ as the Polish space of separable nuclear $G$%
-C*-algebras. For an element $(E,\rho)\in G$\textrm{-}$\mathrm{C}^{\ast }%
\mathrm{ALG}$, we write $C^{\ast }(E,\rho)$ for the associated $G$%
-C*-algebra.

It is easy to see, by induction on the complexity, that any $\mathcal{L}%
_{G}^{\text{C*}}$-formula $\varphi (x_{1},\ldots ,x_{n},\gamma _{1},\ldots
,\gamma _{m})$ induces a Borel map $\hat{\varphi}\colon G\text{\emph{-}}%
\mathrm{C}^{\ast }\mathrm{ALG}\times \mathcal{O}_{2}^{n}\times
G^{m}\rightarrow \mathbb{R}$ given by 
\begin{equation*}
((E,\rho ),(a_{1},\ldots ,a_{n}),(g_{1},\ldots ,g_{m}))\mapsto \varphi
^{C^{\ast }(E,\rho )}(E(a_{1}),\ldots ,E(a_{n}),g_{1},\ldots ,g_{m}).
\end{equation*}%
In other words, the $\mathcal{L}_{G}^{\text{C*}}$-theory of a separable
nuclear $G$-C*-algebra can be computed in a Borel fashion in the
parameterization $G$\textrm{-}$\mathrm{C}^{\ast }\mathrm{ALG}$ of $G$%
-C*-algebras. This allows one to conclude the following.

\begin{theorem}
\label{Theorem:smooth-classification} Separable nuclear $G$-C*-algebras are
smoothly classifiable, in the sense of Borel complexity, up to weak $%
\mathcal{L}_{G}^{\text{C*}}$-equivalence and positive weak $\mathcal{L}_{G}^{%
\text{C*}}$-equivalence.
\end{theorem}

An introduction to the theory of Borel complexity of equivalence relations
can be found in \cite{gao_invariant_2009}. Similar conclusions hold for
unital C*-algebras and (positive) weak $\mathcal{L}_{G}^{1\text{,C*}}$%
-equivalence.

\subsection{Strongly self-absorbing \texorpdfstring{$G$}{G}-C*-algebras}
We continue to fix a compact group $G$.
Let $A$ and $B$ be $G$-C*-algebras and let $\eta _{1},\eta _{2}\colon
A\rightarrow B$ be unital $G$-*-homomorphisms. By $\mathcal{M}(B)$ we
denote the multiplier algebra of $B$, which is endowed with a canonical 
\emph{strictly continuous }$G$-action. Then $\eta _{1}$ and $\eta _{2}$ are
said to be $G$-\emph{unitarily equivalent }if
there exists a unitary element $u$ in $\mathcal{M}(B)^G$ such that $\mathrm{Ad}%
(u)\circ \eta _{1}=\eta _{2}$. Similarly, we
say that $\eta _{1}$ and $\eta _{2}$ are \emph{approximately }$G$-\emph{unitarily equivalent }if there
exists a net $(u_{i})_{i\in I}$ of unitaries in $\mathcal{M}(B)^G$ such that $(%
\mathrm{Ad}(u_{i})\circ \eta _{1})_{i\in I}$ converges pointwise to $\eta
_{2}$

\begin{definition}
The $G$-C*-algebras $(A,\alpha)$ and $(B,\beta)$ are said to be:
\begin{enumerate}\item \emph{conjugate (or $G$-isomorphic)}, if there exists an isomorphism $\eta\colon A\to B$ satisfying
$\eta\circ\alpha_g=\beta_g\circ \eta$ for every $g\in G$;
\emph{cocycle conjugate }if
there exists a strictly
continuous map $v\colon G\rightarrow U(\mathcal{M}(A))$ satisfying $%
v_{gh}=v_{g}g^{B}v_{h}$ such that the action $g\mapsto \mathrm{Ad}(v_g)\circ\beta_g$ is 
conjugate to $\alpha$. 
\end{enumerate}

We say that $A$ is $G$-\emph{equivariantly} $B$-\emph{absorbing }if $A\otimes B$ is 
cocycle conjugate to $B$.
\end{definition}

\begin{definition}
A\emph{\ }$G$-C*-algebra $D$ is said to have\emph{\ approximately }$G$-\emph{%
inner half-flip,} if it is unital and the canonical $G$-equivariant
injective unital *-homomorphisms $\mathrm{id}_{D}\otimes 1_{D},1_{D}\otimes 
\mathrm{id}_{D}\colon D\rightarrow D\otimes D$ are approximately $G$%
-unitarily equivalent. A\emph{\ }$G$-C*-algebra $D$ is said to be a \emph{%
strongly} \emph{self absorbing} $G$-C*-algebra if it is unital and $\mathrm{%
id}_{D}\otimes 1_{D}$ is approximately $G$-unitarily equivalent to a $G$%
-equivariant *-isomorphism.
\end{definition}

Observe that if $D$ has approximately $G$-inner half-flip, then it has
approximately inner half-flip as a C*-algebra. Similarly, if $D$ is a
strongly self-absorbing $G$-C*-algebra, then $D$ is strongly self-absorbing
as a C*-algebra. Recall that any unital C*-algebra $D$ with approximately
inner half-flip is automatically simple, nuclear, and has at most one trace;
see \cite{effros_c*-algebras_1978}. 

The following remark will be used repeatedly and without further reference.

\begin{remark}\label{rem:CocConjSSA}
Suppose $D$ is a strongly self-absorbing $G$-C*-algebra, and let $A$ be a 
separable $G$-C*-algebra. Then $A$ is $G$-equivariantly $D$-absorbing if 
and only if $A\otimes D$ is \emph{conjugate} to $A$; see 
\cite[Theorem~4.7 and Proposition 4.8]{szabo_strongly_2015}. Thus, when
working with strongly self-absorbing actions, we will mostly use conjugacy
as the relevant equivalence relation, keeping in mind that it is 
equivalent to cocycle conjugacy.
\end{remark}

The proof of the following theorem
follows closely arguments from \cite{farah_relative_2015}.

\begin{theorem}
\label{Theorem:commutant}Suppose that $D$ is a separable $G$-C*-algebra with
approximately $G$-inner half-flip, and that $C$ is a countably positively
quantifier-free\ $\mathcal{L}_{G}^{1\text{,}\mathrm{C}^{\mathrm{\ast }}}$%
-saturated unital $G$-C*-algebra. Suppose that $D$ is commutant weakly
contained in $C$. Fix a $G$-equivariant unital *-homomorphism $\theta \colon
D\rightarrow C$. The following statements hold:

\begin{enumerate}
\item Any two $G$-equivariant unital *-homomorphisms $D\to C$ are $G$%
-unitarily equivalent;

\item The inclusion $\theta (D) ^{\prime }\cap C\hookrightarrow C$ is an $%
\mathcal{L}_{G}^{1\text{,}\mathrm{C}^{\mathrm{\ast }}}$-existential $G$%
-equivariant unital *-homomorphism;

\item $\theta (D) ^{\prime }\cap C$ is an elementary $\mathcal{L}_{G}^{1%
\text{,}\mathrm{C}^{\mathrm{\ast }}}$-substructure of $C$;

\item If $C$ has density character $\aleph _{1}$, then the inclusion $\theta
(D) ^{\prime }\cap C\hookrightarrow C$ is approximately $G$-unitarily
equivalent to a $G$-isomorphism.
\end{enumerate}
\end{theorem}

\begin{proof}
Let $\theta _{1}\colon D\rightarrow C$ be a $G$-equivariant unital
*-homomorphism. Assume first that the ranges of $\theta $ and $\theta _{1}$
commute. Choose a sequence $(u_{n})_{n\in \mathbb{N}}$ of unitaries in $%
D\otimes D$ witnessing the fact that $\mathrm{id}_{D}\otimes
1_{D},1_{D}\otimes \mathrm{id}_{D}\colon D\rightarrow D\otimes D$ are
approximately $G$-unitarily equivalent. Let $\Theta \colon D\otimes
D\rightarrow C$ be the $G$-equivariant unital *-homomorphism given by $%
d_{1}\otimes d_{2}\mapsto \theta (d_{1})\theta _{1}(d_{2})$. Considering the
unitaries $\Theta (u_{n})$, for $n\in \mathbb{N}$, and applying the fact
that $A$ is countably positively quantifier-free\ $\mathcal{L}_{G}^{1\text{%
,C*}}$-saturated, we obtain a unitary $u\in C^{G}$ satisfying $\mathrm{Ad}%
(u)\circ \theta =\theta _{1}$. In the general case, when the ranges of $%
\theta $ and $\theta _{1}$ do not necessarily commute, we may find a unital $%
G$-equivariant *-homomorphism $\theta _{2}\colon D\rightarrow C$ whose range
commutes with those of $\theta $ and $\theta _{1}$. By the argument above,
it follows that $\theta _{2}$ is $G$-unitarily equivalent to both $\theta $
and $\theta _{1}$, so (1) follows.

We prove (2) and (3) simultaneously. Let us identify $D$ with its image
under $\theta $. Suppose that $\overline{a}$ is a tuple in $D^{\prime }\cap
C $, $\overline{b}$ is a tuple in $C$, and $\varphi (\overline{x},\overline{y%
}) $ is an $\mathcal{L}_{G}^{1\text{,C*}}$-formula. Let $B$ be the $G$%
-C*-algebra generated by $D\cup \left\{ \overline{a},\overline{b}\right\} $
inside $C$. Observe that $B^{\prime }\cap C$ satisfies the same assumptions
as $C$. Particularly, by (1) there exists a unitary $u\in B^{\prime }\cap C$
such that $g^{C}u=u$ for every $g\in G$ and $u^{\ast }\overline{b}u\in
D^{\prime }\cap C$. Hence we have $\varphi ^{C}(\overline{a},\overline{b}%
)=\varphi ^{D^{\prime }\cap C}(\overline{a},u^{\ast }\overline{b}u)$, as
desired.

The argument above shows that for any separable $G$-C*-subalgebra $B$ of $C$
and finite tuple $\overline{b}$ in $C$ there exists a unitary $u$ in the
fixed point algebra of $C$ such that $u\in B^{\prime }\cap C$ and $u^{\ast }%
\overline{b}u\in D^{\prime }\cap C$. One can then apply the intertwining
argument of~\cite[Theorem~2.11]{farah_relative_2015} to get (4).
\end{proof}

\begin{corollary}
\label{Corollary:commutant}Suppose that $D$ is a separable $G$-C*-algebra
with approximately $G$-inner half-flip, and $\mathcal{F}$ is a countably
incomplete filter. Let $A$ be a separable unital $G$-C*-algebra, and let $%
\theta \colon D\rightarrow \prod_{\mathcal{F}}^{G}A$ be a $G$-equivariant
unital *-homomorphism. Then:

\begin{enumerate}
\item Any $G$-equivariant unital *-homomorphism $D\rightarrow \prod_{%
\mathcal{F}}^{G}A$ is $G$-unitarily equivalent to $\theta $;

\item The inclusion $\theta (D)^{\prime }\cap \prod_{\mathcal{F}%
}^{G}A\hookrightarrow \prod_{\mathcal{F}}^{G}A$ is an $\mathcal{L}_{G}^{1%
\text{,}\mathrm{C}^{\mathrm{\ast }}}$-existential $G$-equivariant unital
*-homomorphism;

\item $\theta (D)^{\prime }\cap \prod_{\mathcal{F}}^{G}A$ is an elementary $%
\mathcal{L}_{G}^{1\text{,}\mathrm{C}^{\mathrm{\ast }}}$-substructure of $%
\prod_{\mathcal{F}}^{G}A$;

\item If $\mathcal{F}$ is a filter over $\mathbb{N}$ and the Continuum
Hypothesis holds, then $\theta (D)^{\prime }\cap \prod_{\mathcal{F}}^{G}A$
is $G$-equivariantly *-isomorphic to $\prod_{\mathcal{F}}^{G}A$.
\end{enumerate}
\end{corollary}

\begin{remark}
Theorem \ref{Theorem:commutant} and Corollary \ref{Corollary:commutant}
generalize \cite[Theorem 1, Theorem 2, Corollary 2.12]{farah_relative_2015}
in two ways: they extend them to the $G$-equivariant setting, and they
remove the unnecessary assumption on the filter $\mathcal{F}$ that the
corresponding reduced product be countably saturated. An example of a
countable incomplete filter over $\mathbb{N}$ that does not satisfy such an
assumption is provided in \cite[Example 3.2]{farah_relative_2015}.
\end{remark}

Suppose now that $D$ is a strongly self-absorbing $G$-C*-algebra. Observe
that for any separable $G$-C*-algebra $A$ and any countably incomplete
filter $\mathcal{F}$, the following assertions are equivalent:

\begin{enumerate}
\item $D$ is commutant weakly contained in $A$

\item $D$ is positively commutant weakly contained in $A$,

\item $D$ embeds equivariantly into $F_{\mathcal{F}}^{G}(A) $.

\item[(4)] $A$ and $A\otimes D$ are (cocycle) conjugate;

\item[(5)] $A$ is $G$-equivariantly $D$-absorbing;

\item[(6)] $D$ is weakly $\mathcal{L}_{G}^{1\text{,C*}}$-contained in $A$;

\item[(7)] $D$ is positively weakly $\mathcal{L}_{G}^{1\text{,C*}}$%
-contained in $A$.
\end{enumerate}

We deduce the following rigidity result for strongly self-absorbing $G$%
-C*-algebras.

\begin{proposition}
\label{Proposition:smooth-ssa} Let $D$ and $E$ be strongly self-absorbing $G$%
-C*-algebras. The following assertions are equivalent:

\begin{enumerate}
\item $D$ and $E$ are (cocycle) conjugate;

\item $D$ and $E$ are weakly $\mathcal{L}_{G}^{1\text{,}\mathrm{C}^{\mathrm{%
\ast }}}$-equivalent;

\item $D$ and $E$ are isomorphic as C*-algebras to the same strongly
self-absorbing C*-algebra $B$, and the $\mathrm{Aut}(B)$-orbits of $D$ and $%
E $ inside the Polish space $\mathrm{Act}_{G}(B)$ of continuous actions of $%
G $ on $B$ have the same closure.
\end{enumerate}

In particular, the classification of strongly self-absorbing $G$%
-C*-algebras up to (cocycle) conjugacy is smooth.
\end{proposition}

The equivalence of (2) and (3) in Proposition \ref{Proposition:smooth-ssa}
is due to the fact that if $D$ is a strongly self-absorbing C*-algebra, then
any injective *-homomorphism $\eta \colon D\rightarrow \prod_{\mathcal{U}}D$%
, where $\mathcal{U}$ is an ultrafilter over $\mathbb{N}$, admits a lift $%
(\eta _{n})_{n\in \mathbb{N}}$ consisting of automorphisms of $D$.

Proposition \ref{Proposition:smooth-ssa} can be seen as the equivariant
analogue of \cite[Theorem 2.16, Corollary 2.17]{farah_relative_2015}. We
would like to remark, however, that Proposition \ref{Proposition:smooth-ssa}
is in principle somewhat more surprising than its nonequivariant
counterpart. Indeed, while there are only very few known strongly
self-absorbing C*-algebras (and it is indeed currently known to be complete
under additional regularity assumptions on the algebra), there seem to exist
a greater variety of strongly self-absorbing actions on C*-algebras. For
instance, for a fixed compact group $G$ and a fixed strongly self-absorbing
C*-algebra $D$, there may exist multiple (non cocycle equivalent) strongly
self-absorbing actions on $D$. In fact, a complete list of all strongly
self-absorbing actions is at the moment far out of reach.

The following consequence of Corollary \ref{Corollary:commutant} seems worth
isolating.

\begin{corollary}
\label{Corollary:ssa} Let $D$ be a strongly self-absorbing $G$-C*-algebra,
let $A$ be a separable unital $G$-equivariantly $D$-absorbing $G$%
-C*-algebra, let $\mathcal{F}$ be a countably incomplete filter, let $\theta
\colon D\rightarrow \prod_{\mathcal{F}}^{G}A$ be a $\mathcal{L}_{G}^{1\text{,%
}\mathrm{C}^{\mathrm{\ast }}}$-embedding. Then:

\begin{enumerate}
\item Any two $G$-equivariant unital *-homomorphisms of $D$ into $\prod_{%
\mathcal{F}}^{G}A$ are $G$-unitarily equivalent;

\item The inclusion $\theta (D)^{\prime }\cap \prod_{\mathcal{F}%
}^{G}A\hookrightarrow \prod_{\mathcal{F}}^{G}A$ is an $\mathcal{L}_{G}^{1%
\text{,}\mathrm{C}^{\mathrm{\ast }}}$-existential $G$-equivariant unital
*-homomorphism;

\item $\theta (D)^{\prime }\cap \prod_{\mathcal{F}}^{G}A$ is a $G$%
-elementary substructure of $\prod_{\mathcal{F}}^{G}A$;

\item If $\mathcal{F}$ is a filter over $\mathbb{N}$ and the Continuum
Hypothesis holds, then $\theta (D)^{\prime }\cap \prod_{\mathcal{F}}^{G}A$
is $G$-equivariantly *-isomorphic to $\prod_{\mathcal{F}}^{G}A$.
\end{enumerate}
\end{corollary}

Using the results above, one can provide the following model-theoretic
characterization of strongly self-absorbing $G$-C*-algebras, which in the
nonequivariant setting is \cite[Theorem 2.14]{farah_relative_2015}. (Recall
that when $G$ is compact, the notions of strongly self-absorbing $G$%
-C*-algebra and strongly self-absorbing $G$-C*-algebra coincide.)

\begin{theorem}
\label{Theorem:characterization-ssa} Let $D$ be a separable unital $G$%
-C*-algebra, and let $\mathcal{F}$ be a countably incomplete filter. Then $D$
is a strongly self absorbing $G$-C*-algebra if and only if $D$ is weakly $%
\mathcal{L}_{G}^{1\text{,}\mathrm{C}^{\mathrm{\ast }}}$-equivalent to $%
D\otimes D$, and all the $G$-equivariant unital *-homomorphisms $%
D\rightarrow \prod_{\mathcal{F}}^{G}D$ are $G$-unitarily equivalent.
\end{theorem}

\begin{proof}
The \textquotedblleft only if\textquotedblright\ implication is a
consequence of the fact that $D$ is $G$-strongly cocycle conjugate to $%
D\otimes D$, and part~(1) of Theorem~\ref{Theorem:commutant}. We prove the
converse. Since $D$ is weakly $\mathcal{L}_{G}^{1\text{,}\mathrm{C}^{\mathrm{%
\ast }}}$-equivalent to $D\otimes D$, we deduce that $D\otimes D$ is a $G$%
-elementary substructure of $\prod_{\mathcal{F}}^{G}D$, say via an embedding 
$\rho $. In particular, the $G$-equivariant unital *-homomorphisms $\rho
_{1},\rho _{2}\colon D\rightarrow \prod_{\mathcal{F}}^{G}D$, given by $\rho
_{1}(d)=\rho (d\otimes 1_{D})$ and $\rho _{2}(d)=\rho (1_{D}\otimes d)$, for 
$d\in D$, are $G$-unitarily equivalent. It follows that $D$ has
approximately $G$-inner half-flip. The conclusion now follows from the
implication (ii)$\Rightarrow $(i) in \cite[Theorem 4.6]{szabo_strongly_2015}.
\end{proof}

\subsection{Limiting examples}

We have shown in Proposition \ref{Proposition:smooth-ssa} that, for any
second countable locally compact group $G$, the classification problem for
strongly self-absorbing $G$-actions on C*-algebras is smooth in the sense of
Borel complexity theory. In this subsection, we observe that the same is not
true for the broader class of $G$-actions with approximately $G$-inner
half-flip, even if one only considers actions of $\mathbb{Z}_{2}:=\mathbb{Z}%
/2\mathbb{Z}$ on the C*-algebra $\mathcal{O}_{2}$. The notion of complete
analytic set can be found in \cite[Section 22.9]{kechris_classical_1995}.

\begin{proposition}
\label{Proposition:nonclassification}The relations of conjugacy and cocycle
conjugacy for approximately representable $\mathbb{Z}_{2}$-actions on $%
\mathcal{O}_{2}$ with Rokhlin dimension $1$ and approximately $\mathbb{Z}%
_{2} $-inner half-flip are complete analytic sets. Furthermore, the
classification problem for such actions, up to conjugacy or cocycle
conjugacy, is strictly more complicated than the classification problem for
any class of countable structures with Borel isomorphism relation.
\end{proposition}

\begin{proof}
Recall that in \cite{izumi_finite_2004} Izumi constructed an action $\nu $
of $\mathbb{Z}_{2}$ on $\mathcal{O}_{2}$ whose crossed product $D=\mathcal{O}%
_{2}\rtimes _{\nu }\mathbb{Z}_{2}$ is a Kichrberg algebra satisfying the
Universal Coefficient Theorem, with trivial $K_{1}$-group, $K_{0}$-group
isomorphic to $\mathbb{Z}\left[ \frac{1}{2}\right] $, and zero element of $%
K_{0}( D) $ corresponding to the unit of $D$; see \cite[Lemma 4.7]%
{izumi_finite_2004}.

Such an action was used in \cite{gardella_conjugacy_2016} to prove that the
relations of conjugacy and cocycle conjugacy of $\mathbb{Z}_{2}$-actions on $%
\mathcal{O}_{2}$ are complete analytic sets, when regarded as subsets of $%
\mathrm{Act}_{\mathbb{Z}_{2}}(\mathcal{O}_{2})\times \mathrm{Act}_{\mathbb{Z}%
_{2}}(\mathcal{O}_{2})$. Precisely, it is proved in \cite%
{gardella_conjugacy_2016}, relying on a construction of R\o rdam from \cite%
{rordam_classification_1995}, that there exists a Borel map assigning to
each uniquely $2$-divisible torsion-free countable abelian group $\Gamma $ a
Kirchberg algebra $A_{\Gamma }$ satisfying the Universal Coefficient
Theorem, with trivial $K_{1}$-group, $K_{0}$-group isomorphic to $\Gamma $,
and zero element of $K_{0}(A_{\Gamma })$ corresponding to the unit of $%
A_{\Gamma }$. Denote by $\iota _{A_{\Gamma }}$ the trivial $\mathbb{Z}_{2}$%
-action on $A_{\Gamma }$. Then the function $\Gamma \mapsto \alpha _{\Gamma
}:=\nu \otimes \iota _{A_{\Gamma }}$ provides a Borel reduction from the
relation $E$ of isomorphism of uniquely $2$-divisible torsion-free countable
abelian groups to the relations of conjugacy and cocycle conjugacy of $%
\mathbb{Z}_{2}$-actions on $\mathcal{O}_{2}$. It is furthermore shown in 
\cite{gardella_conjugacy_2016}, modifying an argument of Hjorth from \cite%
{hjorth_isomorphism_2002}, that $E$ is a complete analytic set. Furthermore,
if $F$ is the relation of isomorphism within a class of countable
structures, and if $F$ is Borel, then $F$ is Borel reducible to $E$ (but not
vice versa). It was furthermore observed in \cite{gardella_conjugacy_2016}
that, for any uniquely $2$-divisible torsion-free countable abelian group $%
\Gamma $, the action $\alpha _{\Gamma }$ has Rokhlin dimension $1$, and is
approximately representable.

We claim that $\alpha _{\Gamma }$ has approximately $\mathbb{Z}_{2}$-inner
half-flip. To see this, it is enough to observe that the $\mathbb{Z}_{2}$%
-action $\nu $ on $\mathcal{O}_{2}$ (corresponding to the case when $\Gamma $
is trivial) is strongly self-absorbing. This follows from the fact that $\nu 
$ is constructed as the infinite tensor product $\bigotimes_{n\in \mathbb{N}}%
\mathrm{Ad}( u) $, where $u$ is a unitary element of $\mathcal{O}_{\infty }^{%
\mathrm{st}}$, using the identification $\mathcal{O}_{2}\cong
\bigotimes_{n\in \mathbb{N}}\mathcal{O}_{\infty }^{\mathrm{st}}$; see \cite[%
Section 4]{izumi_finite_2004}. Since $\mathcal{O}_{\infty }^{\mathrm{st}}$
is a C*-algebra with approximately inner half-flip, one can deduce from \cite%
[Proposition 5.3]{szabo_strongly_2016} that $\nu $ is strongly
self-absorbing. Since $\alpha_\Gamma$ is the tensor product of a strongly
self-absorbing action (namely, $\nu$) with an action with approximately $%
\mathbb{Z}_2$-inner half-flip (namely, $\iota_{A_\Gamma}$), it follows that $%
\alpha_\Gamma$ has approximately $\mathbb{Z}_2$-inner half-flip. This proves
the claim.

Using these observations, and considering the fact that the set of $\mathbb{Z%
}_{2}$-actions on $\mathcal{O}_{2}$ with approximately $\mathbb{Z}_{2}$%
-inner half-flip is analytic, the result follows.
\end{proof}

Clearly, similar conclusions hold for $G$-actions on $\mathcal{O}_{2}$ for
any countable discrete group $G$ with a quotient of order $2$, such as the
group of integers. This can be seen by regarding a $\mathbb{Z}_{2}$-action
as a $G$-action in the canonical way.

\section{Order zero dimension and Rokhlin dimension\label{Section:oz}}

\subsection{Order zero dimension\label{Subsection:oz}}

The notion of positive weak $\mathcal{L}$-containment between $\mathcal{L}$%
-morphisms can be defined in the general setting of logic for metric
structures; see Subsection \ref{Subsection:existential-embedding}. For $G$%
-C*-algebras, one has the following: a $G$-equivariant *-homomorphism $%
\theta \colon A\rightarrow B$ is positively weakly $\mathcal{L}_{G}^{\text{C*%
}}$-contained in another $G$-equivariant *-homomorphism $f\colon
A\rightarrow C$ if for any separable subalgebras $A_{0}\subset A$ and $%
B_{0}\subset B$ such that $\theta (A_{0})\subset B_{0}$, and for some
(equivalently, any) countably incomplete filter $\mathcal{F}$, there exists
a $G$-equivariant *-homomorphism $\gamma \colon B_{0}\rightarrow \prod_{%
\mathcal{F}}^{G}C$ such that $(\gamma \circ \theta )|_{A_{0}}=(\Delta
_{C}\circ f)|_{A_{0}}$, where $\Delta _{C}\colon C\rightarrow \prod_{%
\mathcal{F}}^{G}C$ is the diagonal *-homomorphism. Various equivalent
formulations of this notion can be found in Subsection \ref%
{Subsection:existential-embedding}.

We now present natural generalizations of positive weak $\mathcal{L}_{G}^{%
\text{C*}}$-containment where instead of a single *-ho\-momorphism one
considers a tuple of completely positive contractive order zero maps.
Whenever $f\colon A\rightarrow B$ is a $G$-equivariant *-homomorphism, we
will regard $B$ as a $G$-equivariant $A$-bimodule, as defined in Subsection %
\ref{Subsection:bimodules}.

\begin{definition}
\label{Definition:d-containment} Let $A$, $B$, and $C$ be $G$-C*-algebras,
and let $\theta \colon A\rightarrow B$ and $f\colon A\rightarrow C$ be $G$%
-equivariant *-homomorphisms. We say that $\theta $ is \emph{$G$%
-equivariantly $d$-contained} in $f$ if for any separable $G$-C*-subalgebras 
$A_{0}\subset A$ and $B_{0}\subset B$ such that $\theta (A_{0})\subset B_{0}$%
, and for some (equivalently, any) countably incomplete filter $\mathcal{F}$%
, there exist $G$-equivariant completely positive contractive order zero $A$%
-bimodule maps $\psi _{0},\ldots ,\psi _{d}\colon B_{0}\rightarrow \prod_{%
\mathcal{F}}^{G}C$ whose sum $\psi =\psi _{0}+\cdots +\psi _{d}$ is
contractive and such that $( \psi \circ \theta ) |_{A_{0}}=( \Delta
_{C}\circ f) |_{A_{0}}$.
\end{definition}

The notion of $G$-equivariant $d$-containment from Definition \ref%
{Definition:d-containment} admits a natural syntactic reformulation: $\theta
\colon A\rightarrow B$ is $G$-equivariantly $d$-contained in $f\colon
A\rightarrow C$ if and only if for any tuples $\overline{a}$ in $A$, $%
\overline{b}$ in $B$, and for any tuple $\overline{w}$ of elements of a
finite dimensional C*-algebra, for any positive quantifier-free\ $\mathcal{L}%
_{G}^{\mathrm{oz,}A\text{-}A}$-formulas $\varphi (\overline{z},\overline{y})$%
, for any positive quantifier-free\ $\mathcal{L}_{G}^{\mathrm{osos}\text{,}A%
\text{-}A}$-formulas $\psi (\overline{x},\overline{z},\overline{y})$, where
the variables $\overline{z}$ have finite-dimensional C*-algebras as sorts,
and for any $\varepsilon >0$, there exist tuples $\overline{c}_{0},\ldots ,%
\overline{c}_{d}$ in $C$ such that the following are satisfied for $j=0,\ldots ,d$:%
\begin{equation*}
\psi (f(\overline{a}),\overline{w},\overline{c}_{0}+\ldots +\overline{c}%
_{d})\leq \psi (\theta (\overline{a}),\overline{w},\overline{b})+\varepsilon
\ \ \mbox{ and } \ \
\varphi (\overline{w},\overline{c}_{j})\leq \varphi (\overline{w},\overline{b%
})+\varepsilon.
\end{equation*}%

\begin{remark}
\label{Remark:oz-nuclear}When $B$ is nuclear, in the syntactic
characterization of $d$-containment, one can replace $\mathcal{L}_{G}^{%
\mathrm{osos}\text{,}A\text{-}A}$-formulas with $\mathcal{L}_{G}^{\mathrm{%
osos}\text{-}\mathrm{nuc}\text{,}A\text{-}A}$-formulas, and $\mathcal{L}%
_{G}^{\mathrm{oz}\text{,}A\text{-}A}$-formulas with $\mathcal{L}_{G}^{%
\mathrm{oz}\text{-}\mathrm{nuc}\text{,}A\text{-}A}$-formulas. This follows
from the characterization of $\mathcal{L}_{G}^{\mathrm{osos}\text{-}\mathrm{%
nuc}\text{,}A\text{-}A}$-morphisms from Lemma \ref{Lemma:completely positive
contractive-morphism}.
\end{remark}

\begin{definition}
\label{Definition:oz-dimension}The $G$\emph{-equivariant order zero
dimension }$\dim _{\mathrm{oz}}^{G}(\theta )$ of a $G$-equivariant
*-homomorphism $\theta \colon A\rightarrow B$ is the smallest integer $d\geq
0$ such that $\theta $ is $G$-equivariantly $d$-contained in the identity
map $\mathrm{id}_{A}\colon A\rightarrow A$. If no such $d$ exists, we set $%
\dim _{\mathrm{oz}}^{G}(\theta )=\infty $.
\end{definition}

The proof of the following is an easy consequence of the
syntactic characterization of $G$-equivariant $d$-containment.

\begin{proposition}
\label{prop:propertiesOzDim} Let $\Lambda $ be a directed set.

\begin{enumerate}
\item Let $\theta _{0}\colon A\rightarrow B$ and $\theta _{1}\colon
B\rightarrow C$ be $G$-equivariant *-homomorphisms between $G$-C*-algebras.
Then 
\begin{equation*}
\dim _{\mathrm{oz}}^{G}(\theta _{1}\circ \theta _{0})+1\leq (\dim _{\mathrm{%
oz}}^{G}(\theta _{1})+1)(\dim _{\mathrm{oz}}^{G}(\theta _{0})+1);
\end{equation*}

\item Let $\theta \colon A\rightarrow B$ be a $G$-equivariant
*-homomorphism, and let $C$ be a $G$-C*-algebra. Then 
\begin{equation*}
\dim _{\mathrm{oz}}^{G}(\theta \otimes \mathrm{id}_{C})\leq \dim _{\mathrm{oz%
}}^{G}(\theta );
\end{equation*}

\item Let $(\{A_{\lambda }\}_{\lambda \in \Lambda },\{\theta _{\lambda ,\mu
}\}_{\lambda ,\mu \in \Lambda ,\lambda <\mu })$ be a direct system of $G$%
-C*-algebras (with $G$-equivariant *-homomorphisms). For $\lambda \in
\Lambda $, denote by 
\begin{equation*}
\theta _{\lambda ,\infty }\colon A_{\lambda }\rightarrow \varinjlim A_{\mu }
\end{equation*}
denote the canonical equivariant *-homomorphism. Then 
\begin{equation*}
\dim _{\mathrm{oz}}^{G}(\theta _{\lambda ,\infty })\leq \limsup_{\mu \in
\Lambda }\dim _{\mathrm{oz}}^{G}(\theta _{\lambda ,\mu }).
\end{equation*}

\item For $j=0,1$, let $(\{A_{\lambda }^{(j)}\}_{\lambda \in \Lambda
},\{\theta _{\lambda \mu }^{(j)}\}_{\lambda ,\mu \in \Lambda ,\lambda <\mu
}) $ be a direct system of $G$-C*-algebras. Let $\left\{ \eta _{\lambda
}\colon A_{\lambda }^{(0)}\rightarrow A_{\lambda }^{(1)}\right\} _{\lambda
\in \Lambda }$ be a family of $G$-equivariant *-homomorphisms. Then 
\begin{equation*}
\dim _{\mathrm{oz}}^{G}(\varinjlim_{\lambda \in \Lambda }\eta _{\lambda
})\leq \limsup_{\mu \in \Lambda }\dim _{\mathrm{oz}}^{G}(\eta _{\mu }).
\end{equation*}
\end{enumerate}
\end{proposition}

Let $\theta \colon A\rightarrow B$ is a $G$-equivariant completely positive
contractive order zero map. We recall that by \cite[Proposition~2.3]%
{gardella_regularity_?}, there is a naturally induced completely positive
contractive order zero map $A\rtimes G\rightarrow B\rtimes G$ between the
crossed products, which we will denote in the following by $\widehat{\theta }
$. If $\theta $ is a *-homomorphism, then $\widehat{\theta }$ is a
*-homomorphism as well. If $\theta $ is an $A$-$A$-bimodule morphism, then $%
\widehat{\theta }$ is an $A$-$A$-bimodule morphism as well.

\begin{lemma}
\label{Lemma:crossed} Let $A$ and $B$ be $G$-C*-algebras and let $\theta
\colon A\rightarrow B$ be a $G$-equivariant *-homomorphism. Then $\dim _{%
\mathrm{oz}}(\widehat{\theta })\leq \dim _{\mathrm{oz}}^{G}(\theta )$ and $%
\dim _{\mathrm{oz}}(\theta |_{A^{G}})\leq \dim _{\mathrm{oz}}^{G}(\theta )$.
\end{lemma}

\begin{proof}
Observe that if $A$ is a $G$-C*-algebra and $\mathcal{F}$ is a countably
incomplete filter, then there exists a canonical *-homomorphism $\left(
\prod_{\mathcal{F}}^{G}A\right) \rtimes G\rightarrow \prod_{\mathcal{F}%
}^{G}(A\rtimes G)$ in view of the universal property of the full crossed
product; see \cite{gardella_compact_2015}. This, together with the remarks
above, proves the first assertion. The second assertion can be proved
similarly observing that $A^{G}$ is positively quantifier-free $\mathcal{L}%
_{G}^{\text{C*}}$-definable.
\end{proof}

We also consider the following strengthening of the notion of $d$%
-containment.

\begin{definition}
\label{Definition:d-containment-commuting-towers} Let $A$, $B$, and $C$ be $%
G $-C*-algebras of density character less than $\kappa $, and let $\theta
\colon A\rightarrow B$ and $f\colon A\rightarrow C$ be $G$-equivariant
*-homomorphisms. We say that $\theta $ is $G$-equivariantly $d$-contained in 
$f$ \emph{with commuting towers} if for some (equivalently, any) $\kappa $%
-good filter $\mathcal{F}$, there exist $G$-equivariant completely positive
contractive order zero $A$-bimodule maps $\psi _{0},\ldots ,\psi _{d}\colon
B\rightarrow \prod_{\mathcal{F}}^{G}C$ whose sum $\psi =\psi _{0}+\cdots
+\psi _{d}$ is contractive, such that $\psi \circ \theta =\Delta _{C}\circ f$
and such that, for every $0\leq i<j\leq d$, the images of $\theta ( A)
^{\prime }\cap B$ under $\psi _{i}$ and $\psi _{j}$ commute.
\end{definition}

Observe that, in Definition \ref{Definition:d-containment-commuting-towers},
since the $\psi _{i}$'s are assumed to be $A$-bimodule maps, the image of $%
A^{\prime }\cap B$ under $\psi _{i}$ is contained in $f(A)^{\prime }\cap
\prod_{\mathcal{F}}^{G}C$. Similarly as $d$-containment, the notion of $G$%
-equivariant $d$-containment with commuting towers from Definition \ref%
{Definition:d-containment-commuting-towers} admits a natural syntactic
reformulation: $\theta \colon A\rightarrow B$ is $G$-equivariantly $d$%
-contained in $f\colon A\rightarrow C$ if and only if for any tuples $%
\overline{a}$ in $A$, $\overline{b}$ in $B$, and $\overline{b}^{\prime }\in
\theta (A)^{\prime }\cap B$, and for any tuple $\overline{w}$ of elements of
a finite dimensional C*-algebra, for any positive quantifier-free\ $\mathcal{%
L}_{G}^{\mathrm{oz,}A\text{-}A}$-formulas $\varphi (\overline{z},\overline{y}%
,\overline{y}^{\prime })$, for any positive quantifier-free\ $\mathcal{L}%
_{G}^{\mathrm{osos}\text{,}A\text{-}A}$-formulas $\psi (\overline{x},%
\overline{z},\overline{y})$, where the variables $\overline{z}$ have
finite-dimensional C*-algebras as sorts, and for any $\varepsilon >0$, there
exist tuples $\overline{c}_{0},\ldots ,\overline{c}_{d},\overline{c}%
_{0}^{\prime },\ldots ,\overline{c}_{d}^{\prime }$ in $C$ such that $\left[ 
\overline{c}_{i}^{\prime },\overline{c}_{j}^{\prime }\right] =0$ for $0\leq
i<j\leq d$, 
\begin{equation*}
\psi (f(\overline{a}),\overline{w},\overline{c}_{0}+\ldots +\overline{c}_{d},%
\overline{c}_{0}^{\prime }+\ldots +\overline{c}_{d}^{\prime })\leq \psi
(\theta (\overline{a}),\overline{w},\overline{b},\overline{b}^{\prime
})+\varepsilon
\end{equation*}%
and%
\begin{equation*}
\varphi (\overline{w},\overline{c}_{j})\leq \varphi (\overline{w},\overline{b%
},\overline{b}^{\prime })+\varepsilon \ \mbox { for }j=0,\ldots ,d\text{.}
\end{equation*}

\begin{definition}
\label{Definition:oz-dimension-commuting-towers}The $G$\emph{-}equivariant
order zero dimension\emph{\ with commuting towers }$\dim _{\mathrm{oz}}^{%
\mathrm{c},G}(\theta )$ of a $G$-equivariant *-homomorphism $\theta \colon
A\rightarrow B$ is the smallest integer $d\geq 0$ such that $\theta $ is $G$%
-equivariantly $d$-contained with commuting towers in the identity map $%
\mathrm{id}_{A}\colon A\rightarrow A$. If no such $d$ exists, we set $\dim _{%
\mathrm{oz}}^{G}(\theta )=\infty $.
\end{definition}

\subsection{Commutant \texorpdfstring{$d$}{d}-containment\label%
{Subsection:containment}}

The notion of commutant positive existential $\mathcal{L}_{G}^{\text{C*}}$%
-theory of a $G$-C*-algebra has been introduced in Subsection \ref%
{Subsection:commutant-theory}. In this section, we consider $d$-dimensional
generalizations of such a notion. We will regard (not necessarily unital)
C*-algebras as structures in the Kirchberg language introduced in Subsection %
\ref{Subsection:general}. This will allow us to formulate a definition
applicable in both the unital and the nonunital settings.

\begin{definition}
\label{Definition:commutant-containment} Let $d\in \mathbb{N}$, and let $A$
and $B$ be $G$-C*-algebras. Fix a cardinal $\kappa $ larger than the density
character of $A\ $and $B$. We say that $A$ is \emph{$G$-equivariantly
commutant $d$-contained} in $B$, and write $A\precsim _{d}B$, if for some
(equivalently, any) countably incomplete $\kappa $-good filter $\mathcal{F}$%
, and for any separable unital $G$-C*-subalgebra $C$ of $F_{\mathcal{F}%
}^{G}(A)$, there exist $G$-equivariant completely positive contractive order
zero maps $\eta _{0},\ldots ,\eta _{d}\colon C\rightarrow F_{\mathcal{F}%
}^{G}(B)$ with unital sum.

We say that $A$ is $G$-\emph{equivariantly commutant} $d$-\emph{contained}
in $B$ \emph{with commuting towers}, and write $A\precsim _{d}^{\mathrm{c}}B$%
, if one choose the maps $\eta _{0},\ldots ,\eta _{d}\colon C\rightarrow F_{%
\mathcal{F}}^{G}(B)$ as above to also have pairwise commuting ranges.
\end{definition}

Using Proposition \ref{Proposition:F} one can give a syntactic reformulation
of Definition \ref{Definition:commutant-containment}, which in particular
shows that the choice of the countably incomplete $\kappa $-good filter $%
\mathcal{F}$ is irrelevant. When $A,B$ are separable, one can take any
countably incomplete filter. It is not difficult to see that if $A\precsim
_{d-1}B$ and $B\precsim _{k-1}C$ then $A\precsim _{dk-1}C$.

\begin{remark}
\label{Remark:containment-ssa}Suppose that a separable unital $G$-C*-algebra 
$A$ admits a $G$-equivariant unital *-homomorphism into $A^{\prime }\cap
\prod_{\mathcal{F}}^{G}A$ for some (equivalently, any) countably incomplete
filter. Then $A$ is $G$-equivariantly commutant $d$-contained (with
commuting towers) in $B$ if and only if there exist $G$-equivariant
completely positive contractive order zero maps $\eta _{0},\ldots ,\eta
_{d}\colon A\rightarrow F_{\mathcal{F}}^{G}(B)$ (with commuting ranges) such
that $\eta _{0}+\cdots +\eta _{d}$ is unital. In particular, this applies
when $A$ is commutative, or when $A$ is strongly self-absorbing; see \cite[%
Theorem 4.6]{szabo_strongly_2015}.
\end{remark}

Let $D$ be a strongly self-absorbing $G$-C*-algebra. By \cite[Theorem 4.7
]{szabo_strongly_2015}, a separable $G$-C*-algebra $B$ is 
$G$-equivariantly $D$-absorbing if and only if $D$ is commutant $0$%
-contained in $B$. We will prove in Theorem \ref{Theorem:D-absorption} that
this is in turn equivalent to $D$ being commutant $d$-contained with
commuting towers in $B$ for any $d\in \mathbb{N}$.

\begin{remark}
\label{Remark:bootstrap}Let $A$ and $B$ be separable $G$-C*-algebras with $A$
$G$-equivariantly commutant $d$-contained (with commuting towers) in $B$.
Let $C$ be a separable subalgebra of $F_{\mathcal{F}}^{G}(B)$. It is a
consequence of Proposition~\ref{Proposition:F} that there exist completely
positive contractive order zero maps $\eta _{0},\ldots ,\eta _{d}\colon
A\rightarrow C^{\prime }\cap F_{\mathcal{F}}^{G}(B)$ (with commuting ranges)
such that $\eta _{0}+\cdots +\eta _{d}$ is unital.
\end{remark}

\subsection{Relationship between order zero dimension and 
\texorpdfstring{$d$}{d}-containment}

The notion of\emph{\ Rokhlin dimension }(with commuting towers) for a $G$%
-C*-algebra---see \cite[Definition 1.1]{hirshberg_rokhlin_2015}, \cite[%
Definition 3.2]{gardella_rokhlin_2014}---can be naturally presented in terms
of $d$-containment. Precisely, a $G$-C*-algebra $A$ has Rokhlin dimension
(with commuting towers) at most $d$ if and only if the $G$-C*-algebra $C(G)$
endowed with the canonical left translation action $\mathtt{Lt}$ is $G$%
-equivariantly commutant $d$-contained (with commuting towers) in $A$. We
will denote by $\mathrm{dim}_{\mathrm{Rok}}(A)$ the Rokhlin dimension of a $%
G $-C*-algebra $A$, and by $\mathrm{dim}_{\mathrm{Rok}}^{\mathrm{c}}(A)$ the
Rokhlin dimension with commuting towers of $A$. We point out that Rokhlin
dimension has recently been generalized to $\mathbb{R}$-actions (flows) in 
\cite{hirshberg_rokhlin_2016}.

In Proposition~\ref{Proposition:dimension-implies-containment}, we will
observe that there exists a relationship between the notion of $G$%
-equivariant order zero dimension of a $G$-equivariant *-homomorphism
introduced in Subsection \ref{Subsection:oz}, and the notion of $G$%
-equivariant commutant $d$-containment introduced in Subsection \ref%
{Subsection:containment}. Precisely, if $\theta \colon A\rightarrow B$ has $%
G $-equivariant order zero dimension (with commuting towers) at most $d$,
then $B$ is commutant $d$-contained (with commuting towers) in $A$.

\begin{lemma}
\label{Lemma:oz-dimension-commutant-embedding}Let $C$ be a unital $G$%
-C*-algebra, let $A$ and $B$ be $G$-C*-algebras, let $\kappa $ be a cardinal
larger than the density character of $A$ and $C$, and let $\mathcal{F}$ be a
countably incomplete $\kappa $-good filter. Suppose that $\theta \colon
A\rightarrow B$ is a $G$-equivariant *-homomorphism, and let $1_{C}\otimes
\theta \colon A\rightarrow C\otimes _{\max }B$ be the map $a\mapsto
1_{C}\otimes \theta (a)$. If $\dim _{\mathrm{oz}}^{G}(1_C\otimes \theta
)\leq d<+\infty $, then there exist $G$-equivariant completely positive
contractive order zero maps $\eta _{0},\ldots ,\eta _{d}\colon C\rightarrow
F_{\mathcal{F}}^{G}(A)$ such that $\sum_{j=0}^{d}\eta _{j}$ is unital.
The converse holds if $A=B$ and $\theta \colon A\rightarrow A$ is the
identity.
\end{lemma}

\begin{proof}
Observe that $\theta $ is necessarily injective. We can therefore identify $%
A $ with its image under $\theta $ inside $B$. Let $\psi _{0},\ldots ,\psi
_{d}\colon C\otimes _{\max }B\rightarrow \prod_{\mathcal{F}}^{G}A$ be $G$%
-equivariant completely positive contractive order zero $A$-bimodule maps
witnessing the fact that $\dim _{\mathrm{oz}}^{G}(1_{C}\otimes \theta )\leq
d $. Fix $c_{0}=1,c_{1},\ldots ,c_{n}\in C$. Let $t(\overline{x})$ be a
positive quantifier-free $\mathcal{L}_{G}^{\mathrm{oz}}$-type that is
realized by $(c_{0},\ldots ,c_{n})$ in $C$. Consider the corresponding
multiplier $\mathcal{L}_{G}^{K}(A)$-type $t_{A}^{\mathrm{m}}$ defined as in
Subsection \ref{Subsection:Kirchberg}. Let $t_{A}^{\mathrm{c}}(\overline{x})$
be the commutant type associated with $A$, and consider the $\mathcal{L}%
_{G}^{K}(A)$-type $q_{A}(\overline{y}_{0},\ldots ,\overline{y}_{d})$
consisting of conditions $\varphi (\overline{y}_{j})\leq r$ for any
condition $\varphi (\overline{x})\leq r$ in $t_{A}^{\mathrm{m}}(\overline{x}%
)\cup t_{A}^{\mathrm{c}}(\overline{x})$ and $j=0,1,\ldots ,d$, and $%
\left\Vert a(y_{0,0}+\cdots +y_{0,d})-a\right\Vert =0$ for every $a\in A$.
Fix an approximate unit $(a_{\lambda })_{\lambda \in \Lambda }$ for $A$.
Considering the tuple $\overline{b}:=(c_{0}\otimes a_{\lambda },\ldots
,c_{n}\otimes a_{\lambda })$ in $C\otimes _{\max }B$, for large enough $%
\lambda $, we conclude that the type $t_{A}^{\mathrm{m}}(\overline{x})\cup
t_{A}^{\mathrm{c}}(\overline{x})$ is approximately realized in $C\otimes
_{\max }B$. Recall that, by definition of $G$-equivariant order zero
dimension, $\psi _{0},\ldots ,\psi _{d}$ are completely contractive order
zero $A$-bimodule maps with contractive sums such that $(\psi _{0}+\cdots
+\psi _{d})|_{A}$ is the canonical diagonal $G$-equivariant $\ast $%
-homomorphism $A\rightarrow \prod_{\mathcal{F}}^{G}A$. Therefore, by \L os'
theorem, considering the elements $\psi _{0}(\overline{b}),\ldots ,\psi _{d}(%
\overline{b})$ shows that the type $q_{A}(\overline{x})$ is approximately
realized in $A$. The conclusion follows from quantifier-free positive $%
\mathcal{L}_{G}^{K}(A)$-saturation of $F_{\mathcal{F}}^{G}(A)$; see
Proposition \ref{Proposition:F}.

Conversely, suppose that $A=B$, that $\theta \colon A\rightarrow A$ is the
identity map $\mathrm{id}_{A}$ of $A$, and that there exist $G$-equivariant
completely positive contractive order zero maps $\eta _{0},\ldots ,\eta
_{d}\colon C\rightarrow F_{\mathcal{F}}^{G}(A)$ such that $\eta _{0}+\cdots
+\eta _{d}$ is unital. Then the function $F_{\mathcal{F}}^{G}(A)\times
A\rightarrow \prod_{\mathcal{F}}^{G}A$, given by $([a_{i}]_{i\in
I},b)\mapsto \lbrack a_{i}b]_{i\in I}$, induces a $G$-equivariant
*-homomorphism $\Psi \colon F_{\mathcal{F}}^{G}(A)\otimes _{\max
}A\rightarrow \prod_{\mathcal{F}}^{G}A$, by the universal property of the
maximal tensor product. One can then define $\psi _{j}=\Psi \circ (\eta
_{j}\otimes \mathrm{id}_{A})\colon C\otimes _{\max }A\rightarrow \prod_{%
\mathcal{F}}^{G}A$, for $j=0,\ldots ,d$. These are well defined $G$%
-equivariant completely positive contractive order zero $A$-bimodule maps,
which witness that $\dim _{\mathrm{oz}}^{G}(\theta )\leq d$.
\end{proof}

Recall that, if $A$ is a C*-algebra and $C$ is a unital C*-algebra, then the
relative commutant of $1_{C}\otimes A$ inside $C\otimes _{\max }A$ is equal
to $C\otimes Z(A) $, where $Z(A) $ is the center of $A$; see \cite[Theorem 4]%
{archbold_centre_1975}. Using this fact, the same proof as Lemma \ref%
{Lemma:oz-dimension-commutant-embedding} shows the following.

\begin{lemma}
\label{Lemma:oz-dimension-commutant-embedding-commuting-towers}Let $C$ be a
unital $G$-C*-algebra, let $A$ and $B$ be $G$-C*-algebras, let $\kappa $ be
a cardinal larger than the density character of $A$ and $C$, and let $%
\mathcal{F}$ be a countably incomplete $\kappa $-good filter. Suppose that $%
\theta \colon A\rightarrow B$ is a $G$-equivariant *-homomorphism, and let $%
1_{C}\otimes \theta \colon A\rightarrow C\otimes _{\max }B$ be the map $%
a\mapsto 1_{C}\otimes \theta (a)$. If $\dim _{\mathrm{oz}}^{\text{c}%
,G}(1_C\otimes \theta )\leq d<+\infty $, then there exist $G$-equivariant
completely positive contractive order zero maps $\eta _{0},\ldots ,\eta
_{d}\colon C\rightarrow F_{\mathcal{F}}^{G}(A)$ with commuting ranges such
that $\eta _{0}+\cdots +\eta _{d}$ is unital. The converse holds if $A=B$
and $\theta \colon A\rightarrow A$ is the identity map.
\end{lemma}

When $C=C(G)$, we deduce the following:

\begin{lemma}
\label{Lemma:rokhlin-and-oz} Let $A$ be a $G$-C*-algebra. Denote by $\theta
\colon A\rightarrow C(G)\otimes A$ the second factor embedding. Then $\dim _{%
\mathrm{Rok}}(A)=\dim _{\mathrm{oz}}^{G}(\theta )$ and $\dim _{\mathrm{Rok}%
}^{\mathrm{c}}(A)=\dim _{\mathrm{oz}}^{\mathrm{c},G}(\theta )$.
\end{lemma}

\begin{proposition}
\label{Proposition:dimension-implies-containment}Suppose that $\theta \colon
A\rightarrow B$ is a $G$-equivariant *-homomorphism. If $\dim _{\mathrm{%
\mathrm{oz}}}^{G}(\theta )\leq d$, then $B\precsim _{d}A$. If $\dim _{%
\mathrm{oz}}^{\mathrm{c},G}( \theta ) \leq d$, then $B\precsim _{d}^{\mathrm{%
c}}A$.
\end{proposition}

\begin{proof}
Fix a countably incomplete filter $\mathcal{F}$. Suppose that $\dim _{%
\mathrm{\mathrm{oz}}}^{G}(\theta )\leq d$. Fix a separable unital $G$%
-C*-subalgebra $C$ of $F_{\mathcal{F}}(B)$. Then by Lemma \ref%
{Lemma:oz-dimension-commutant-embedding} the second factor embedding $%
1_{C}\otimes \mathrm{id}_{B}\colon B\rightarrow C\otimes _{\max }B$ has
order zero dimension equal to zero. By Proposition \ref{prop:propertiesOzDim}%
(2), the $G$-equivariant *-homomorphism $(1_{C}\otimes \mathrm{id}_{B})\circ
\theta \colon A\rightarrow C\otimes _{\max }B$ has order zero dimension at
most $d$. Therefore by Lemma \ref{Lemma:oz-dimension-commutant-embedding}
again there exist $G$-equivariant completely positive contractive order zero
maps $\eta _{0},\ldots ,\eta _{d}\colon C\rightarrow F_{\mathcal{F}}^{G}(A)$
such that $\eta _{0}+\cdots +\eta _{d}$ is unital. By Lemma~\ref%
{Lemma:oz-dimension-commutant-embedding}, this concludes the proof.

The second assertion can be proved in the same way, by replacing Lemma \ref%
{Lemma:oz-dimension-commutant-embedding} with Lemma \ref%
{Lemma:oz-dimension-commutant-embedding-commuting-towers}.
\end{proof}

\subsection{Dimension functions\label{Subsection:dimension}}

By a \emph{dimension function} for (nuclear) $G$-C*-algebras we mean a
function from the class of (nuclear) C*-algebras to $\left\{ 0,1,2,\ldots
,\infty \right\} $.

\begin{definition}
\label{Definition:AE}A dimension function $\dim $ for $G$-C*-algebras is
said to be \emph{positively $\forall \exists $-axiomatizable} if there
exists a collection $\mathcal{F}$ of formulas $\xi ( \overline{x},\overline{z%
}_{0},\ldots ,\overline{z}_{d},\overline{y}_{0},\ldots ,\overline{y}_{d}) $
of the form%
\begin{equation*}
\max \left\{ \eta ( \overline{x},\overline{z}_{0},\ldots ,\overline{z}_{d})
,\varphi _{0}( \overline{z}_{0},\overline{y}_{0}) ,\ldots ,\varphi _{d}(%
\overline{z}_{d},\overline{y}_{d}),\psi ( \overline{x},\overline{z}%
_{0},\ldots ,\overline{z}_{d},\overline{y}_{0}+\cdots +\overline{y}_{d})
\right\},
\end{equation*}%
where

\begin{enumerate}
\item $\overline{z}_{0},\ldots ,\overline{z}_{d}$ have finite-dimensional
C*-algebras as sorts,

\item $\eta $ is a positive quantifier-free \ $\mathcal{L}_{G}^{\text{C*}}$%
-formula,

\item $\varphi $ is a positive quantifier-free \ $\mathcal{L}_{G}^{\mathrm{oz%
}}$-formula,

\item $\psi $ is a positive quantifier-free \ $\mathcal{L}_{G}^{\mathrm{osos}%
}$-formula,
\end{enumerate}
such that the following holds: for a $G$-C*-algebra $A$, $\dim (A)\leq d$ if
and only if%
\begin{equation*}
A\models \sup\limits_{\overline{x}}\inf\limits_{\overline{z}_{0}}\cdots
\inf\limits_{\overline{z}_{d}}\inf\limits_{\overline{y}_{0}}\cdots
\inf\limits_{\overline{y}_{d}}\xi (\overline{x},\overline{z}_{0},\ldots ,%
\overline{z}_{d},\overline{y}_{0},\ldots ,\overline{y}_{d})=0\text{.}
\end{equation*}
\end{definition}

\begin{definition}
A dimension function for \emph{nuclear }$G$-C*-algebras is said to be \emph{%
nuclearly }positively $\forall \exists $-axiomatizable if in Definition \ref%
{Definition:AE} we can simultaneously choose $\varphi $ and $\psi $ to be
positive quantifier-free formulas in $\mathcal{L}_{G}^{\mathrm{oz}\text{-}%
\mathrm{nuc}}$ and $\mathcal{L}_{G}^{\text{C*-}\mathrm{nuc}}$, respectively.
\end{definition}

\begin{example}
\label{eg:NDimDr} The following are positively $\forall \exists $%
-axiomatizable dimension functions for nuclear C*-algebras:

\begin{enumerate}
\item Nuclear dimension. Indeed, one can consider variables $( \overline{z}%
_{0},\ldots ,\overline{z}_{d}) $ with sorts finite-dimensional C*-algebras $%
F_{0},\ldots ,F_{d}$ and then the formulas

\begin{itemize}
\item $\eta (\overline{x},\overline{z}_{0},\ldots ,\overline{z}_{d})\equiv
\max\limits_{j=0,\ldots ,d}\inf\limits_{s\in \mathrm{CPC}(A,F_{j})}\max%
\limits_{k}\left\Vert s(x_{k})-z_{j,k}\right\Vert $;

\item $\varphi _{j}(\overline{y}_{j},\overline{z}_{j})\equiv
\inf\limits_{t\in \mathrm{CPC}(F_{j},A)}\max\limits_{k}\left\Vert
t(z_{j,k})-y_{j,k}\right\Vert $, for a fixed $j=0,\ldots ,d$;

\item $\psi (\overline{x},\overline{z}_{0},\ldots ,\overline{z}_{d},%
\overline{y})\equiv \max\nolimits_{k}\left\Vert x_{k}-y_{k}\right\Vert $.
\end{itemize}

\item Decomposition rank. In fact, one may just consider the same formulas $%
\eta $ and $\psi $ as in (1), together with 
\begin{equation*}
\varphi _{j}(\overline{z}_{j},\overline{y}_{j})\equiv \inf\limits_{t\in 
\mathrm{CPC}(F_{j},A)}\max\limits_{k}\left\Vert
t(z_{j,k})-y_{j,k}\right\Vert ,\ \mbox{ for }j=0,\ldots ,d.
\end{equation*}
\end{enumerate}
\end{example}

If $\overline{x},\overline{y}$ are $n$-tuples of variables, we write $\delta
(\overline{x},\overline{y})$ for the formula 
\begin{equation*}
\max\limits_{1\leq j,k\leq n}\left\Vert x_{j}y_{k}-y_{k}x_{j}\right\Vert 
\text{.}
\end{equation*}

\begin{definition}
\label{Definition:existentially-axiomatizable} A dimension function dim for
(separable) $G$-C*-algebras is said to be \emph{\ commutant positively
existentially axiomatizable} if there exists a collection $\mathcal{F}$ of
formulas $\xi (\overline{x},\overline{y}_{0},\ldots ,\overline{y}_{d})$ of
the form%
\begin{equation*}
\max_{\substack{ 0\leq j<k\leq d  \\ 1\leq \ell \leq n}}\{\delta (\overline{x%
},\overline{y}_{j}),\varphi (\overline{y}_{j}),\left\Vert x_{\ell
}(y_{0,j}+\cdots +y_{d,j})-x_{\ell }\right\Vert \},
\end{equation*}%
where $\varphi $ is a quantifier-free\ positive $\mathcal{L}_{G}^{\mathrm{oz}%
}$-formula with parameters from finite-dimensional C*-algebras, such that
the following hods: for a (separable) $G$-C*-algebra $A$, one has $\dim
(A)\leq d$ if and only if%
\begin{equation*}
A\models \sup\limits_{\overline{x}}\inf\limits_{\overline{y}_{0}}\cdots
\inf\limits_{\overline{y}_{d}}\xi (\overline{x},\overline{y}_{0},\ldots ,%
\overline{y}_{d})=0\text{.}
\end{equation*}
\end{definition}

\begin{definition}
\label{Definition:commutant-commuting}Suppose that $\dim $ is a dimension
function for (separable) $G$-C*-algebras. We say that $\dim $ is\emph{\
commutant positively existentially axiomatizable with commuting towers} if
there exists a collection $\mathcal{F}$ of formulas $\xi (\overline{x},%
\overline{y}_{0},\ldots ,\overline{y}_{d})$ of the form%
\begin{equation*}
\max_{\substack{ 0\leq j<k\leq d  \\ 1\leq \ell \leq n}}\{\delta (\overline{x%
},\overline{y}_{j}),\delta (\overline{y}_{j},\overline{y}_{k}),\varphi (%
\overline{z}_{j},\overline{y}_{j}),\left\Vert x_{\ell }(y_{0,j}+\cdots
+y_{d,j})-x_{\ell }\right\Vert \}
\end{equation*}%
where $\varphi $ is a positive quantifier-free \ $\mathcal{L}_{G}^{\mathrm{oz%
}}$-formula with parameters from finite-dimensional C*-algebras, such that
the following hods: for a (separable) $G$-C*-algebra $A$, one has $\dim
(A)\leq d$ if and only if%
\begin{equation*}
A\models \sup\limits_{\overline{x}}\inf\limits_{\overline{y}_{0}}\cdots
\inf\limits_{\overline{y}_{d}}\xi (\overline{x},\overline{z}_{0},\ldots ,%
\overline{z}_{d},\overline{y}_{0},\ldots ,\overline{y}_{d})=0\text{.}
\end{equation*}
\end{definition}

\begin{example}
\label{Example:dim-Rok}Suppose that $C$ is a fixed $G$-C*-algebra. Set $\dim
( A) \leq d$ if and only if $C$ is commutant $d$-contained in $C$ (with
commuting towers). Then $\dim $ is a dimension function for $G$-C*-algebras
that is commutant positively existentially axiomatizable (with commuting
towers).

In the particular case when $C$ is the $G$-C*-algebra $C(G)$ endowed with
the canonical shift action of $G$, this says that Rokhlin dimension (with
commuting towers) is a commutant positively existentially axiomatizable
(with commuting towers) dimension function.
\end{example}

The following is a consequence of Definition \ref%
{Definition:commutant-commuting} and the syntactic characterization of
commutant $d$-containment.

\begin{proposition}
\label{Proposition:commuting} Let $\dim $ be a dimension function for
separable $G$-C*-algebras that is positively existentially axiomatizable
(with commuting towers). Let $A$ and $B$ be separable $G$-C*-algebras such
that $A$ is commutant $d$-contained (with commuting towers) in $B$. Then 
\begin{equation*}
\dim (B)+1\leq (d+1)(\dim (A)+1)\text{.}
\end{equation*}
\end{proposition}

Similarly, the following fact is a consequence of the syntactic
characterization of $G$-equivariant $d$-containment, Remark \ref%
{Remark:oz-nuclear}, and Proposition \ref%
{Proposition:dimension-implies-containment}.

\begin{proposition}
\label{Proposition:estimate-dimension} Let $A$ and $B$ be $G$-C*-algebras,
and let $\theta \colon A\rightarrow B$ be a $G$-equivariant *-homomorphism.
Suppose that $\dim $ is a dimension function for C*-algebras. If $\dim $ is
positively $\forall \exists $-axiomatizable dimension or commutant
positively existentially axiomatizable, then 
\begin{equation*}
\dim (A)+1\leq (\dim _{\mathrm{oz}}^{G}(\theta )+1)(\dim (B)+1).
\end{equation*}

Moreover, if $B$ is nuclear and $\dim $ is nuclearly $\forall \exists $%
-axiomatizable, then again 
\begin{equation*}
\dim (A)+1\leq (\dim _{\mathrm{oz}}^{G}(\theta )+1)(\dim (B)+1).
\end{equation*}
\end{proposition}

In particular, Proposition \ref{Proposition:estimate-dimension} applies when 
$\dim $ is either nuclear dimension $\dim _{\mathrm{nuc}}$, decomposition
rank $\mathrm{dr}$, or Rokhlin dimension $\dim _{\mathrm{Rok}}$; see Example~%
\ref{eg:NDimDr} and Example \ref{Example:dim-Rok}. More generally, one can
define the nuclear dimension and decomposition rank of a *-homomorphism $%
f\colon A\rightarrow B$, and then show that if $\theta \colon A\rightarrow B$
is $d$-contained in $f$, then 
\begin{equation*}
\dim _{\mathrm{nuc}}(\theta )+1\leq (d+1)(\dim _{\mathrm{nuc}}(f)+1)\ 
\mbox{
and }\ \mathrm{dr}(\theta )+1\leq (d+1)(\mathrm{dr}(f)+1)\text{.}
\end{equation*}

The following result relates the order zero dimension of the canonical
inclusions $A^{G}\rightarrow A$ and $A\rtimes G\rightarrow A\otimes \mathcal{%
K}(L^{2}(G))$ to the Rokhlin dimension of a $G$-C*-algebra $A$.


\begin{proposition}
\label{Proposition:crossed-product} Let $A$ be a $G$-C*-algebra $A$, and
denote by $\iota \colon A^{G}\hookrightarrow A$ and $\sigma \colon A\rtimes
G\rightarrow A\otimes \mathcal{K}(L^{2}(G))$ the canonical inclusion maps.
Then $\dim _{\mathrm{oz}}(\iota )\leq \dim _{\mathrm{Rok}}(A)$ and $\dim _{%
\mathrm{oz}}(\sigma )\leq \dim _{\mathrm{Rok}}(A)$.
\end{proposition}

\begin{proof}
Denote by $\theta \colon A\rightarrow C(G)\otimes A$ the second factor
embedding. Let $\mathtt{Lt}$ denote the action of $G$ on $C(G)$ by left
translation, and denote by $\alpha $ the given action on $A$. Endow $%
C(G)\otimes A$ with the tensor product action $\gamma =\mathtt{Lt}\otimes
\alpha $. Then $\theta $ is $G$-equivariant, and hence it induces a
*-homomorphism $A\rtimes G\rightarrow (C(G)\otimes A)\rtimes G$. Observe
that $(C(G)\otimes A,\gamma )$ is canonically $G$-equivariantly isomorphic
to $(C(G)\otimes A,\mathtt{Lt}\otimes \iota _{A})$ by \cite[Proposition 2.3]%
{gardella_crossed_2014}. Then the crossed product $(C(G)\otimes A)\rtimes
_{\gamma }G$ is canonically isomorphic to $A\otimes \mathcal{K}(L^{2}(G))$,
and the fixed point algebra $(C(G)\otimes A)^{\gamma }$ is canonically
isomorphic to $A$. It follows that the map $\widehat{\theta }$---defined
right before Lemma \ref{Lemma:crossed}---is canonically conjugate to $\sigma 
$, and $\theta |_{A^{G}}$ is canonically conjugate to $\iota $.

Using Lemma~\ref{Lemma:rokhlin-and-oz} at the first step, Lemma \ref%
{Lemma:crossed} at the second, and the above observations at the third, we
get 
\begin{equation*}
\dim _{\mathrm{Rok}}(A)=\dim _{\mathrm{oz}}^{G}(\theta )\geq \dim _{\mathrm{%
oz}}(\widehat{\theta })=\dim _{\mathrm{oz}}(\widehat{\sigma }).
\end{equation*}%
Similarly, we have $\dim _{\mathrm{Rok}}(A)\geq \dim _{\mathrm{oz}}(\iota )$%
, as desired.
\end{proof}

\begin{corollary}
\label{Corollary:existential-crossed-product} Let $A$ be a $G$-C*-algebra $A$%
, and let $\dim $ be a positively $\forall \exists $-axiomatizable dimension
function for C*-algebras. Then 
\begin{equation*}
\dim (A^{G})+1\leq (\dim _{\mathrm{Rok}}(A)+1)(\dim (A)+1)\text{,}
\end{equation*}%
and 
\begin{equation*}
\dim (A\rtimes G)\leq (\dim _{\mathrm{Rok}}(A)+1)(\dim (A)+1)\text{.}
\end{equation*}
\end{corollary}

For separable unital $A$, the following first appeared as \cite[Theorem 3.3]%
{gardella_regularity_?}. The particular case of commuting towers has also
been independently obtained in \cite{gardella_rokhlin_2016}, using
completely different methods.

\begin{corollary}
\label{Corollary:nuclear-esitmate-crossed} Let $A$ be a $G$-C*-algebra $A$.
Then%
\begin{equation*}
\dim _{\mathrm{nuc}}(A^{G})+1\leq \dim _{\mathrm{nuc}}(A\rtimes G)+1\leq
(\dim _{\mathrm{Rok}}(A)+1)(\dim _{\mathrm{nuc}}(A)+1)
\end{equation*}%
and%
\begin{equation*}
\mathrm{dr}(A^{G})+1\leq \mathrm{dr}(A\rtimes G)+1\leq (\dim _{\mathrm{Rok}%
}(A)+1)(\mathrm{dr}(A)+1).
\end{equation*}
\end{corollary}

\begin{proof}
The first inequalities in Corollary \ref{Corollary:nuclear-esitmate-crossed}
are due to the fact that the fixed point algebra $A^{G}$ of a $G$-C*-algebra
is a corner of the crossed product $A\rtimes G$---see \cite[Theorem]%
{rosenberg_appendix_1979}---and the fact that decomposition rank and nuclear
dimension are nonincreasing when passing to hereditary subalgebras; see \cite%
[Proposition 3.8]{kirchberg_covering_2004} and \cite[Proposition 2.5]%
{winter_nuclear_2010}. The second inequalities are an immediate consequence
of Corollary~\ref{Corollary:existential-crossed-product} and Example~\ref%
{eg:NDimDr}.
\end{proof}

\subsection{Bundles}

In this subsection, we generalize the main result of \cite%
{dadarlat_trivialization_2008} to equivariant bundles; see Theorem~\ref%
{thm:trivialBundle}. This result will be crucial in our applications to
actions with finite Rokhlin dimension in Subsections~6.4 and~6.5.

We will need the following equivariant version of the Choi-Effros lifting
theorem for compact groups. 

\begin{proposition}
\label{prop:CEequiv} Let $(A,\alpha )$ and $(B,\beta )$ be $G$-C*-algebras,
and let $\varphi \colon A\rightarrow B$ be a surjective, $G$-equivariant,
nuclear *-homomorphism. Then there exists a $G$-equivariant completely
positive contractive lift $\sigma \colon B\rightarrow A$. If $\varphi $ is
unital, then we can also choose $\sigma $ to be unital.
\end{proposition}

\begin{proof}
Use Choi-Effros to find a completely positive contractive lift $\rho\colon
B\to A$ (which may be chosen to be unital if $\varphi$ is). If $\mu$ denotes
the normalized Haar measure on $G$, then it is easy to check that the map $%
\sigma\colon B\to A$ given by $\sigma(b)=\int_G
\alpha_g(\rho(\beta_{g^{-1}}(b))) \ \mathrm{d}\mu$, for all $b\in B$, is as
in the statement.
\end{proof}

Suppose that $X$ is a compact metrizable space. The definition of $C( X) $%
-algebra can be found in \cite[Definition 2.1]{dadarlat_trivialization_2008}%
. We consider here the natural equivariant analog of a $C( X) $-algebra:

\begin{definition}
Let $A$ be a $C(X)$-algebra. For $x\in X$, denote by $U_{x}$ the open subset 
$X\setminus \{x\}$ of $X$, and denote by $A(U_{x})$ the corresponding ideal
of $A$.
We say that $A$ is a \emph{$G$-$C(X)$-algebra}, if $A$ is endowed with an
action $\alpha \colon G\rightarrow \mathrm{Aut}(A)$ satisfying $\alpha
_{g}(A(U_{x}))\subset A(U_{x})$ for all $x\in X$ and all $g\in G$.
\end{definition}

In the context of the above definition, given $x\in X$, denote by $A_x$ the
quotient $A/A(U_x)$ and by $\pi_x\colon A\to A_x$ the canonical quotient
map. Then $\alpha$ induces actions $\alpha^{(x)}\colon G\to\mathrm{Aut}(A_x)$%
, that make each $\pi_x$ equivariant.

The definition of \emph{unitarily regular action} is given in~\cite[%
Definition~1.18]{szabo_strongly_2016}. Observe that the trivial action on a
strongly self-adsorbing C*-algebra is unitarily regular. More generally,
this applies to any strongly self-absorbing $G$-C*-algebra that $G$%
-equivariantly absorbs the trivial action on the Jiang-Su algebra; see \cite[%
Proposition 1.20]{szabo_strongly_2016}. The main theorem of this subsection
is the following:

\begin{theorem}
\label{thm:trivialBundle} Let $X$ be a compact metrizable space of finite
covering dimension. Let $(D,\delta )$ be a strongly self-absorbing,
unitarily regular $G$-C*-algebra, and let $(A,\alpha )$ be a separable,
unital $G$-$C(X)$-algebra such that $A_{x}$ is $G$-equivariantly isomorphic
to $D$, for all $x\in X$. Then there is a $G$-equivariant $C(X)$-linear
isomorphism 
\[(A,\alpha )\cong (C(X)\otimes D,\iota _{C(X)}\otimes \delta ).\]
\end{theorem}

Our proof follows the lines of Dadarlat-Winter's proof of the nonequivariant
version of Theorem \ref{thm:trivialBundle} from \cite[ Section~4]%
{dadarlat_trivialization_2008}. In fact, for the sake of succinctness, we
only mention what changes are needed in said proof, and leave the smaller
details to the reader. Similar results for general locally compact groups
are explored in \cite{forough_preparation_2018}.

Throughout the rest of the subsection, we fix a compact metrizable space $X$%
, a strongly self-absorbing $G$-C*-algebra $(D,\delta)$, and a separable
unital $G$-$C(X)$-algebra $A$.

\begin{definition}
Let $(B,\beta)$ and $(C,\gamma)$ be $G$-C*-algebras, let $\varepsilon>0$ and
let $F\subset B$ be a compact set. We say that a linear map $\varphi\colon
B\to C$ is \emph{$\varepsilon$-multiplicative} (respectively, \emph{$%
\varepsilon$-equivariant}) on $F$, if $\|\varphi(b_1b_2)-\varphi(b_1)%
\varphi(b_2)\|<\varepsilon$ for all $b_1,b_2\in F$ (respectively, $%
\|\gamma_g(\varphi(b))-\varphi(\beta_g(b))\|<\varepsilon$ for all $g\in G$
and for all $b\in F$).
\end{definition}

The following is the analog of Proposition~4.1 of~\cite%
{dadarlat_trivialization_2008}.

\begin{proposition}
\label{prop:Elliott} Denote by $\mu\colon C(X)\to A$ the structure map.
Suppose that for any $\varepsilon>0$ and for any compact subsets $F\subset A$%
, $H_1\subset C(X)$ and $H\subset D$, there are completely positive
contractive maps $\psi\colon A\to C(X)\otimes D$ and $\varphi\colon
C(X)\otimes D\to A$ satisfying

\begin{enumerate}
\item $\|(\varphi\circ\psi)(a)-a\|<\varepsilon$ for all $a\in F$;

\item $\|\varphi(f\otimes 1_D)-\mu(f)\|<\varepsilon$ for all $f\in H_1$;

\item $\|(\psi\circ\mu)(f)-f\otimes 1_D\|<\varepsilon$ for all $f\in H_1$;

\item $\varphi$ is $\varepsilon$-multiplicative and $\varepsilon$%
-equivariant on $(1_{C(X)}\otimes \mathrm{id}_D)(H_2)$;

\item $\psi$ is $\varepsilon$-multiplicative and $\varepsilon$-equivariant
on $F$.
\end{enumerate}

Then there is a $G$-equivariant $C(X)$-linear isomorphism $(A,\alpha)\cong
(C(X)\otimes D,\iota_{C(X)}\otimes\delta)$.
\end{proposition}

\begin{proof}
The only thing that needs to be checked is that the isomorphisms $\overline{%
\varphi }$ and $\overline{\psi }$ constructed in~\cite%
{dadarlat_trivialization_2008}, are equivariant, which is a routine
computation.
\end{proof}

We need an equivariant version of \cite[Proposition~3.5]%
{dadarlat_trivialization_2008}, in order to prove the analog of \cite[%
Lemma~4.5]{dadarlat_trivialization_2008}. We note here that when $%
\varepsilon >0$ is small enough, then any unitary in $A_{\varepsilon }^{G}$
can be perturbed to a nearby unitary in $A^{G}$. Moreover, if the original
unitary can be connected to the unit within $A_{\varepsilon }^{G}$, then its
perturbation can be connected to the unit by a path in $A^{G}$.

\begin{proposition}
\label{prop:3.5} Suppose that $D$ is unitarily regular. Then for any finite
set $F\subset D$ and every $\varepsilon >0$, there exist a finite set $%
H\subset D$ and $\delta >0$ with the following property: for any unital $D$%
-absorbing $G$-C*-algebra $A$, and any unital completely positive maps $%
\varphi ,\psi \colon D\rightarrow A$ that are $\delta $-multiplicative and $%
\delta $-equivariant on $H$, there is a unitary $u\in \mathcal{U}_{0}(A^{G})$
such that $\Vert \varphi (d)-u\psi (d)u^{\ast }\Vert <\varepsilon $ for all $%
d\in F$.
\end{proposition}

\begin{proof}
The proof in~\cite{dadarlat_trivialization_2008} applies almost verbatim,
with the following changes: the maps $\Phi $ and $\Psi $ are also
equivariant. Instead of \cite[Corollary~1.12]{toms_strongly_2007}, use \cite[%
Proposition~3.4(iii)]{szabo_strongly_2015}; the obtained unitary $V$ can be
chosen to belongs to $(B\otimes D)^{G}$, and similarly with $V_{n}$. The
equivariant analog of \cite[Proposition~1.9]{toms_strongly_2007} is
straightforward to show for compact groups (choosing unitaries in the fixed
point algebra). The unital homomorphisms $\theta _{n}$ can then be chosen to
be equivariant, and the maps $\gamma _{n}$ are also equivariant. 
Finally, one must use \cite[Theorem~2.15]%
{szabo_strongly_2016} instead of \cite[Theorem~3.1]%
{dadarlat_trivialization_2008} (this is where unitary regularity of the
action on $D$ is used).
\end{proof}

Lemma~4.2 in~\cite{dadarlat_trivialization_2008} goes through with only
minor changes:

\begin{lemma}
\label{lemma:4.2} Adopt the notation from \cite[Lemma~4.2]%
{dadarlat_trivialization_2008}. Assume furthermore that $D$ is unitarily
regular and that the maps $\sigma _{1}$ and $\sigma _{2}$ are $\delta
(F,\gamma )$-equivariant on $\mathcal{E}(F,\gamma )$. Then there is a
continuous path $(u_{t})_{t\in \lbrack 0,1]}$ of unitaries in $(C(K)\otimes
D)^{G}$ satisfying $u_{0}=1_{C(K)}\otimes 1_{D}$ and $\Vert u_{1}\sigma
_{1}(d)u_{1}^{\ast }-\sigma _{2}(d)\Vert <\gamma $ for all $d\in F\cdot F$.
\end{lemma}

\begin{proof}
Replace every application of \cite[Proposition~3.5]%
{dadarlat_trivialization_2008} with an application of Proposition~\ref%
{prop:3.5}.
\end{proof}

We need an equivariant analog of a local approximate trivialization; see 
\cite[Definition~4.3]{dadarlat_trivialization_2008}. Since our notation
differs slightly from the one used in said paper, we reproduce the
definition entirely.

\begin{definition}
\label{df:trivialization} For $n\in\mathbb{N}$, we write $p\colon [0,1]^n\to
[0,1]$ for the first coordinate projection. Given a compact subset $X\subset
[0,1]^n$, set $Y=p(X)$. If $C\subset Y$ is a closed subset, we write $%
X_C=p^{-1}(C)$. Let $A$ be a unital $G$-$C(X)$-algebra $A$. We abbreviate $%
A_{X_C}$ to $A_C$, and $A_{X_{\{s\}}}$ to $A_s$, for $s\in Y$, while the
fiber maps are denoted $\pi_C$ and $\pi_s$, respectively. (We will not
distinguish, as far as the notation is concerned, between fiber maps of
different $C(X)$-algebras associated to the same closed subset of $X$.)

Suppose that $D$ is a strongly self-absorbing $G$-C*-algebra, that each
fiber of $A$ is $G$-equivariantly isomorphic to $D$, and that for each $s\in
Y$, there is a $G$-$C(X_s)$-algebra isomorphism $A_s\cong C(X_s)\otimes D$.
Let $\eta>0$, let $t\in Y$, and let $\theta\colon A_t\to C(X_t)\otimes D$ be
a $G$-$C(X_t)$-algebra isomorphism. Fix compact subsets $F\subset A$
containing $1_A$, $H\subset C(X)\otimes D$, and $\widehat{H}\subset
C(X_t)\otimes D$.

Let $Y^{(t)}$ be a closed neighborhood of $t$ in $Y$. A \emph{$G$%
-equivariant $(\theta ,F,H,\widehat{H},\eta )$-trivialization of $A$ over $%
Y^{(t)}$} is a family of diagrams, indexed over $s\in Y^{(t)}$, as follows: 
\begin{equation*}
\xymatrix{ & & A\ar[d]^-{\pi_{Y^{(t)}}} & \\ & & A_{Y^{(t)}}\ar[r]^-{\pi_s}
\ar[d]^-{\theta^{(t)}} & A_s\ar[d]^-{\theta_s^{(t)}}\\ C(X)\otimes
D\ar[r]^-{\pi_{Y^{(t)}}} & C(Y^{(t)})\otimes
D\ar[r]^-{\iota^{(t)}}\ar[ur]^-{\sigma^{(t)}} & \prod\limits_{r\in Y^{(t)}}
C(X_r)\otimes D\ar[r]^-{\pi_s}\ar[d]^-{\pi_t} & C(X_s)\otimes D)\\ & &
C(X_t)\otimes D\ar[ul]^-{\zeta^{(t)}}, & }
\end{equation*}%
where all C*-algebras are $G$-$C(X)$-algebras in the obvious way; each map
is $G$-equivariant, unital and completely positive; and conditions (i)
through (xii) in \cite[Definition~4.3]{dadarlat_trivialization_2008} are
satisfied. 
\end{definition}

Existence of equivariant local approximate trivializations, in the sense of
the definition above, is established similarly as in the nonequivariant case:

\begin{lemma}
\label{lemma:4.5} Adopt the notation and assumptions of the first two
paragraphs of Definition~\ref{df:trivialization}, and assume moreover that $%
D $ is unitarily regular. Then there exist a closed neighborhood $%
Y^{(t)}\subset Y$ of $t$ and a $G$-equivariant $(\theta, F, H, \widehat{H}%
,\eta)$-trivialization of $A$ over $Y^{(t)}$.
\end{lemma}

\begin{proof}
Again, the proof given in \cite{dadarlat_trivialization_2008} requires only
minor changes: the isomorphisms $\widetilde{\theta }_{s}^{(t)}\colon
A_{s}\rightarrow C(X_{s})\otimes D$ are chosen to be $G$-equivariant. Also,
the $G$-equivariant, unital completely positive lifts $\overline{\zeta }%
^{(t)}\colon C(X_{t})\rightarrow C(X_{\widetilde{Y}^{(t)}})$ and $\overline{%
\sigma }^{(t)}\colon D\rightarrow A_{\widetilde{Y}^{(t)}}$ are obtained
using Proposition~\ref{prop:CEequiv}. The applications of \cite[Lemma~4.2]%
{dadarlat_trivialization_2008} are replaced by applications of Lemma~\ref%
{lemma:4.2}. In particular, $\widetilde{u}^{(s)}$ can be chosen to be $G$%
-invariant. It follows that $\theta _{s}^{(t)}$ is equivariant, since so are 
$\widetilde{\pi }^{(s)}$ and $\widetilde{\theta }_{s}^{(t)}$. The
verification of (xi) and (xii) in Definition~\ref{df:trivialization} is
routine, and we omit it.
\end{proof}

Finally, we come to the proof of the main result of this section:

\begin{proof}[Proof of Theorem~\protect\ref{thm:trivialBundle}]
Use Proposition~\ref{prop:Elliott} instead of \cite[Proposition~4.1]%
{dadarlat_trivialization_2008}. The basis of induction must also assume that 
$\theta _{t}\colon A_{t}\rightarrow C(X_{t})\otimes D$ is $G$-equivariant.
Apply Lemma~\ref{lemma:4.5} in place of \cite[Lemma~4.5]%
{dadarlat_trivialization_2008}. The unital completely positive maps $\lambda
^{(i)},\varrho ^{(i)}\colon C(X_{y_{i}})\otimes D\rightarrow
C(X_{t_{i}})\otimes D$ are $G$-equivariant because so are $\zeta ^{(y_{i})}$%
, $\sigma ^{(y_{i})}$, $\pi _{t_{i}}$, $\theta _{t_{i}}^{(y_{i})}$, and $%
\theta _{t_{i}}^{(y_{i+1})}$. The unitaries $u_{t}^{(i)}$, for $t\in \lbrack
0,1]$ and $i\in I$, can be chosen to be $G$-invariant by Lemma~\ref%
{lemma:4.2}; in other words, the path $t\mapsto u_{t}^{(i)}$ determines a $G$%
-invariant unitary in $C([0,1])\otimes C(X_{t_{1}})\otimes D$, where $%
C([0,1])$ carries the trivial $G$-action. The unitaries defined in (31) and
(32) are automatically $G$-invariant. Finally, the maps $\psi \colon
A\rightarrow C(X)\otimes D$ and $\varphi \colon C(X)\otimes D\rightarrow A$
are readily checked to be equivariant (observe that the structure map of a $%
G $-$C(X)$-algebra is equivariant when $C(X)$ is endowed with the trivial $G$%
-action). This finishes the proof.
\end{proof}

\subsection{\texorpdfstring{$G$}{G}-equivariant \texorpdfstring{$D$}{D}%
-absorption}

We start by providing a new characterization of $G$-equivariant $D$%
-absorption. The nonequivariant case (when the group is trivial), has
recently been observed in \cite{hirshberg_rokhlin_2016}.

\begin{theorem}
\label{Theorem:D-absorption} Let $A$ be a separable $G$-C*-algebra, let $%
\mathcal{F}$ be a countably incomplete filter, and let $D$ be a strongly
self-absorbing, unitarily regular $G$-C*-algebra. Fix $d\in \mathbb{N}$.
Then $A$ is $G$-equivariantly $D$-absorbing if and only if there exist $G$%
-equivariant completely positive contractive order zero maps $\psi
_{0},\ldots ,\psi _{d}\colon D\rightarrow F_{\mathcal{F}}^{G}(A)$ with
commuting ranges such that $\psi _{0}+\cdots +\psi _{d}$ is unital.
\end{theorem}

\begin{proof}
By \cite[Theorem~3.7]{szabo_strongly_2015}, being $D$-absorbing is
equivalent to the condition in Theorem \ref{Theorem:D-absorption} with $d=0$%
. We now prove the converse implication. We let $C(D)=C_{0}((0,1])\otimes D$
denote the cone of $D$, and $C(D)^{\dagger }$ denote its minimal
unitization, endowed with the canonical $G$-action. The tensor product $%
C(D)^{\dagger }\otimes \cdots \otimes C(D)^{\dagger }$ of $d+1$ copies of $%
C(D)^{\dagger }$ has a canonical $G$-equivariant character. We let $E$ be
its the kernel, which is a $G$-invariant ideal. Observe that if $B$ is a
unital C*-algebra, then $(d+1)$-tuples of $G$-equivariant completely
positive contractive order zero maps $D\rightarrow B$ with commuting ranges
and unital sum, are into one-to-one correspondence with unital $G$%
-equivariant *-homomorphisms $E\rightarrow B$. This follows form the
structure theorem for completely positive contractive order zero maps \cite[%
Corollary 4.1]{winter_completely_2009}---or, more precisely, its equivariant
counterpart \cite[Corollary 2.10]{gardella_rokhlin_2014}---and the universal
properties of unitization and tensor products. Therefore, in order to
conclude the proof, it is enough to show that $E$ is $G$-equivariantly $D$%
-absorbing.

Denote by $X$ the spectrum of the center of $E$, which is a subspace of the $%
(d+1)$-dimensional cube $[0,1]^{d+1}$. Thus, $X$ is a compact metrizable
space. Moreover, the $G$-C*-algebra\ $E$ is easily seen to be a $G$-$C(X)$%
-algebra with fibers isomorphic to $D$. By Theorem~\ref{thm:trivialBundle},
we conclude that $E$ is $G$-equivariantly isomorphic to $C(X)\otimes D$, and
in particular is $G$-equivariantly $D$-absorbing. This finishes the proof.
\end{proof}

In view of Remark \ref{Remark:containment-ssa}, one can reformulate Theorem %
\ref{Theorem:D-absorption} by asserting that $A$ is $G$-equivariantly $D$%
-absorbing if and only if it is commutant $d$-contained in $D$ with
commuting towers for some $d\in \mathbb{N}$.

\begin{corollary}
\label{Corollary:D-absorption} Let $A$ be a separable $G$-C*-algebra, let $%
\mathcal{F}$ be a countably incomplete filter, and let $D$ be a strongly
self-absorbing, unitarily regular $G$-C*-algebra. Then $A$ is $D$-absorbing
if and only if there exist $d\in \mathbb{N}$ and completely positive
contractive order zero maps $\psi _{0},\ldots ,\psi _{d}\colon D\rightarrow
F_{\mathcal{F}}^{G}(A)$ with commuting ranges such that $\psi _{0}+\cdots
+\psi _{d}$ is unital.
\end{corollary}

Suppose that $D$ is a strongly self-absorbing $G$-C*-algebra. Consider the $%
\left\{ 0,\infty \right\} $-valued dimension function for separable $G$%
-C*-algebras obtained by setting $\dim _{D}(A)=0$ if and only if $A$ is $G$%
-equivariantly $D$-absorbing. The following proposition is an immediate
consequence of Theorem \ref{Theorem:D-absorption}; see also Example \ref%
{Example:dim-Rok}.

\begin{proposition}
\label{Proposition:dimDCommPosExisAxiom} Let $D$ be a strongly
self-absorbing, unitarily regular $G$-C*-algebra. Then $\dim _{D}$, as
defined above, is commutant positively existentially axiomatizable with
commuting towers.
\end{proposition}

The following is the main result of this subsection. The conclusion is new
even in the nonequivariant setting. Recall that $\mathcal{Z}$-absorbing
strongly self-absorbing actions are automatically unitarily regular.

\begin{corollary}
\label{Corollary:absorbption-containment} Let $A$ and $B$ be separable $G$%
-C*-algebras, and let $D$ be a unitarily regular, strongly self-absorbing $G$%
-C*-algebra. If $A$ is $G$-equivariantly $D$-absorbing and $A\precsim _{d}^{%
\mathrm{c}}B$ for some $d\in \mathbb{N}$, then $B$ is $G$-equivariantly $D$%
-absorbing.
\end{corollary}

\begin{proof}
If $A$ is $G$-equivariantly $D$-absorbing, then $D\precsim _{0}^{\mathrm{c}%
}A $. If furthermore $A\precsim _{d}^{\mathrm{c}}B$, then we have $D\precsim
_{d}^{\mathrm{c}}B$. Therefore $B$ is $G$-equivariantly $D$-absorbing by
Theorem \ref{Theorem:D-absorption} and Remark \ref{Remark:containment-ssa}.
\end{proof}

\subsection{Examples and applications to dimensional inequalities}

In this section, we exhibit some examples of embeddings with finite order
zero dimension, and use them to deduce some dimensional inequalities,
particularly for nuclear dimension and decomposition rank. We need to
extract a technical fact from Section~5 of \cite{matui_decomposition_2014}.
If $a,b$ are elements of a C*-algebra $A$, we write $a\thickapprox
_{\varepsilon }b$ to denote that $\left\Vert a-b\right\Vert <\varepsilon $.

\begin{lemma}
\label{lemma:matui-sato} Let $n\in\mathbb{N}$, and let $\varepsilon>0$. Then
there exist completely positive contractive maps $\lambda_0,\lambda_1\colon
M_n\to \mathcal{Z}$ such that $\lambda_0(1_{M_n})+\lambda_1(1_{M_n})
\approx_{\varepsilon} 1_{\mathcal{Z}}$.
\end{lemma}

\begin{proof}
See the first part of proof of Theorem~1.1 in Section~5 of \cite%
{matui_decomposition_2014}.
\end{proof}

\begin{theorem}
\label{thm:InclZUHF} Let $U$ be a UHF-algebra, and let $\theta\colon 
\mathcal{Z}\to U$ be any unital embedding. Then $\dim _{\mathrm{oz}}(
\theta)= 1$.
\end{theorem}

\begin{proof}
Since any two unital embeddings of $\mathcal{Z}$ into $U$ are approximately
unitarily equivalent, and $\mathcal{Z}\otimes U$ is isomorphic to $U$, we
may assume, without loss of generality, that $\theta $ is the first tensor
factor embedding $\mathcal{Z}\rightarrow \mathcal{Z}\otimes U$. Let $%
\mathcal{F}$ be the filter of cofinite subsets of $\mathbb{N}$. Write $U$ as
an increasing union $U=\overline{\bigcup_{n\in \mathbb{N}}M_{k_{n}}}$ of
matrix algebras $M_{k_{n}}$. Using injectivity of $M_{k_{n}}$, choose a
conditional expectation $E_{n}\colon U\rightarrow M_{k_{n}}$. For $n,m\in 
\mathbb{N}$, let $\lambda _{0}^{(n,m)},\lambda _{1}^{(n,m)}\colon
M_{k_{n}}\rightarrow \mathcal{Z}$ denote the order zero maps obtained from
Lemma~\ref{lemma:matui-sato} for $\varepsilon =1/m$. For $j=0,1$, set 
\begin{equation*}
\lambda _{j}^{(n)}=(\lambda _{j}^{(n,m)})_{m\in \mathbb{N}}\colon
M_{k_{n}}\rightarrow \prod\nolimits_{\mathcal{F}}\mathcal{Z},
\end{equation*}%
which is an order zero map. Note that $\lambda
_{0}^{(n)}(1_{M_{k_{n}}})+\lambda _{1}^{(n)}(1_{M_{k_{n}}})$ is equal to the
identity of $\prod_{\mathcal{F}}\mathcal{Z}$. For $j=0,1$, let $\psi
_{j}\colon U\rightarrow \prod\nolimits_{\mathcal{F}}(\prod\nolimits_{%
\mathcal{F}}\mathcal{Z})$ be given by $\psi _{j}(x)=(\lambda
_{j}(E_{n}(x)))_{n\in \mathbb{N}}$ for all $x\in U$. Then $\psi _{j}$ is
order zero, and $\psi _{0}(1_{U})+\psi _{1}(1_{U})$ is equal to the identity
of $\prod\nolimits_{\mathcal{F}}(\prod\nolimits_{\mathcal{F}}\mathcal{Z})$.
We obtain a commutative diagram 
\begin{equation*}
\xymatrix{\mathbb{C} \ar[r]\ar[d] & U\ar[d]^-{\psi_j}\\
\prod\nolimits_{\mathcal{F}}\mathcal{Z} \ar[r] &
\prod\nolimits_{\mathcal{F}}(\prod\nolimits_{\mathcal{F}}\mathcal{Z}),}
\end{equation*}%
where the maps from $\mathbb{C}$ are the canonical unital homomorphisms, and
the lower horizontal map is the canonical diagonal *-homomorphism $\Delta
_{\prod_{\mathcal{F}}\mathcal{Z}}:\prod_{\mathcal{F}}\mathcal{Z}\rightarrow
\prod_{\mathcal{F}}(\prod_{\mathcal{F}}\mathcal{Z)}$. We claim that there
are completely positive contractive order zero maps $\varphi _{0},\varphi
_{1}\colon U\rightarrow \prod_{\mathcal{F}}\mathcal{Z}$ such that $\psi
_{j}=\Delta _{\prod_{\mathcal{F}}\mathcal{Z}}\circ \varphi _{j}$ for $j=0,1$%
. This (and in fact, a more general statement) can be proved along the lines
of \cite[Lemma~4.18]{gardella_crossed_2014}, replacing condition (2) in its
proof with the following: 
\begin{equation*}
\left\Vert (\psi _{j})_{m}^{(n_{r})}(b^{\ast })-(\psi
_{j})_{m}^{(n_{r})}(b)^{\ast }\right\Vert <\frac{1}{r}\ \ \mbox{ and }\ \
\left\Vert (\psi _{j})_{m}^{(n_{r})}(cc^{\prime })\right\Vert <\frac{1}{r}
\end{equation*}%
whenever $b,c,c^{\prime }\in G_{r}$ satisfy $cc^{\prime \ast }c^{\prime
}=c^{\prime \ast }=c^{\prime }c=0$. We omit the details.

The fact that $\dim _{\mathrm{oz}}(\theta )\leq 1$ now follows from Lemma~%
\ref{Lemma:oz-dimension-commutant-embedding}. It remains to show that $\dim
_{\mathrm{oz}}(\theta )>0$. Since $\dim _{\mathrm{nuc}}(\mathcal{Z})=1$ and $%
\dim _{\mathrm{nuc}}(U)=0$, the claim follows from Proposition~\ref%
{Proposition:estimate-dimension} for $\dim =\dim _{\mathrm{nuc}}$.
\end{proof}

In the proof of the next theorem, given C*-algebras $A$ and $B$, given $%
\varepsilon>0$ and given a finite subset $F\subset A$, we say that a linear
map $\varphi\colon A\to B$ is \emph{$\varepsilon$-order zero} on $F$, if $%
\|\varphi(ab)\|<\varepsilon$ for all $a,b\in F$ satisfying $%
ab=a^*b=ab^*=a^*b^*=0$.

\begin{theorem}
\label{thm:InclKirO2} Let $A$ be a unital Kirchberg algebra, and let $%
\theta\colon A\to \mathcal{O}_2$ be any unital embedding. Then $\dim _{%
\mathrm{oz}}( \theta)\leq 1$. Moreover, $\dim _{\mathrm{oz}}( \theta)=1$
unless $A=\mathcal{O}_2$.
\end{theorem}

\begin{proof}
Assume first that $A=\mathcal{O}_\infty$. As in the proof of Theorem~\ref%
{thm:InclZUHF}, we may assume, without loss of generality, that $\theta$ is
the first tensor factor embedding $\mathcal{O}_\infty\to\mathcal{O}%
_\infty\otimes \mathcal{O}_2$. We will verify the finitary version of order
zero dimension. To that effect, let $\varepsilon>0$, and let $F\subset 
\mathcal{O}_\infty$ and $H\subset \mathcal{O}_2$ be finite subsets
consisting of positive contractions.

Use \cite[Lemma~4.17]{gardella_rokhlin_2014}---see also the first part of
the proof of Theorem~3.3 in~\cite{barlak_rokhlin_2015}---to find
*-homomorphisms $\varphi _{0},\varphi _{1}\colon \mathcal{O}_{2}\rightarrow 
\mathcal{O}_{\infty }$ and positive contractions $k_{0},k_{1}\in \mathcal{O}%
_{2}$ such that $\Vert \varphi _{0}(k_{0})+\varphi _{1}(k_{1})-1_{\mathcal{O}%
_{\infty }}\Vert <\varepsilon /5$. Since $\mathcal{O}_{\infty }$ is
isomorphic to its infinite tensor product, we may choose $\varphi _{0}$ and $%
\varphi _{1}$ to satisfy $\Vert \varphi _{j}(y)a-a\varphi _{j}(y)\Vert
<\Vert y\Vert \varepsilon /5$ for $j=0,1$, for all $y\in \mathcal{O}_{2}$
and for all $a\in F$. (For instance, find $m\in \mathbb{N}$ and a finite
subset $F^{\prime }\subset \otimes _{j=1}^{m}\mathcal{O}_{\infty }\subset
\otimes _{j=1}^{\infty }\mathcal{O}_{\infty }$, such that for every $a\in F$
there exists $a^{\prime }\in F^{\prime }$ with $\Vert a-a^{\prime }\Vert
<\varepsilon /5$. With $\iota _{m+1}\colon \mathcal{O}_{\infty }\rightarrow
\otimes _{j=1}^{\infty }\mathcal{O}_{\infty }$ denoting the $(m+1)$-st
tensor factor embedding, the maps $\varphi _{j}\circ \iota _{m+1}$, for $%
j=0,1$, will satisfy the condition above.) Likewise, since $\mathcal{O}_{2}$
is isomorphic to its infinite tensor product, we may also assume that $\Vert
k_{j}b-bk_{j}\Vert <\varepsilon /5$ for $j=0,1$ and for all $b\in H$. Define
completely positive contractive maps $\gamma _{0},\gamma _{1}\colon \mathcal{%
O}_{\infty }\otimes \mathcal{O}_{2}\rightarrow \mathcal{O}_{\infty }$ on
simple tensors as follows: for $x\in \mathcal{O}_{\infty }$ and for positive 
$y\in \mathcal{O}_{2}$, set 
\begin{equation*}
\gamma _{j}(x\otimes y)=\varphi _{j}(k_{j})^{1/2}\varphi
_{j}(y)^{1/2}x\varphi _{j}(y)^{1/2}\varphi _{j}(k_{j})^{1/2},\ \mbox{ for }%
j=0,1.
\end{equation*}

We claim that $\gamma _{0}$ and $\gamma _{1}$ are $\varepsilon $-order zero
on $F\otimes H$, that $((\gamma _{0}+\gamma _{1})\circ \theta )(a)\approx
_{\varepsilon }a$, and that $\gamma _{j}(a)\approx _{\varepsilon }\gamma
_{j}(1)a$ for all $j=0,1$ and all $a\in F$. To show the first part of the
claim, it is enough to observe that when $x\in F$ and $y\in H$, we have $%
\gamma _{j}(x\otimes y)\approx _{4\varepsilon /5}\varphi _{j}(k_{j}y)x$ for $%
j=0,1$. For the second one, a similar reasoning applies, since for $a\in F$
we have $\gamma _{j}(a\otimes 1_{\mathcal{O}_{2}})\approx _{2\varepsilon
/5}\varphi _{j}(k_{j})a$, and hence 
\begin{equation*}
(\gamma _{0}+\gamma _{1})(a\otimes 1_{\mathcal{O}_{2}})\approx
_{4\varepsilon /5}(\varphi _{0}(k_{0})+\varphi _{1}(k_{1}))a\approx
_{\varepsilon /5}a,
\end{equation*}%
as desired. The third part of the claim also follows, since we have $\gamma
_{j}(a)\approx _{2\varepsilon /5}\varphi _{j}(k_{j})a=\gamma _{j}(1)a$ for $%
j=0,1$ and for $a\in F$. This proves the result for $A=\mathcal{O}_{\infty }$%
.

When $A$ is an arbitrary Kirchberg algebra, the claim follows from the first
part of the proof and part~(2) of Proposition~\ref{prop:propertiesOzDim},
together with Kirchberg's absorption theorems $A\otimes \mathcal{O}_{\infty
}\cong A$ and $A\otimes \mathcal{O}_{\infty }\cong \mathcal{O}_{2}$.

When $A=\mathcal{O}_{2}$, then any inclusion into $\mathcal{O}_{2}$ is
approximately unitarily equivalent to the identity, which clearly has order
zero dimension zero. Since having a positively existential embedding into $%
\mathcal{O}_{2}$ implies absorbing $\mathcal{O}_{2}$, it follows that $\dim
_{\mathrm{oz}}(\theta )=1$ whenever $A$ is not $\mathcal{O}_{2}$.
\end{proof}

In particular, we recover from Theorem \ref{thm:InclKirO2} the following
dimensional estimate from \cite[Theorem~7.1]{matui_decomposition_2014}. The
actual nuclear dimension of Kirchberg algebras has recently been computed in 
\cite[Theorem~G]{bosa_covering_?}: it is always 1. We nevertheless present
this consequence to illustrate the applicability of our techniques.

\begin{corollary}
Let $A$ be a Kirchberg algebra. Then $\mathrm{dim}_\mathrm{nuc}(A)\leq 3$.
\end{corollary}

\begin{proof}
This follows immediately from Theorem~\ref{thm:InclKirO2}, Proposition~\ref%
{Proposition:estimate-dimension}, and the fact that $\dim _{\mathrm{nuc}}(%
\mathcal{O}_{2})=1$.
\end{proof}

In the next result, we endow $\mathcal{Z}, \mathcal{O}_2, \mathcal{O}_\infty$
and the UHF-algebra with the trivial $G$-action, and we endow all tensor
products with the diagonal action.

\begin{theorem}
\label{thm:O2OIUHFDimIneq} Let $A$ be a $G$-C*-algebra, and let $\dim$ be a
positively $\forall \exists $-axiomatizable dimension function for $G$%
-C*-algebra. Let $U$ be a UHF-algebra of infinite type. Then 
\begin{equation*}
\dim(A\otimes \mathcal{Z})\leq 2 \dim(A\otimes U)+1 \ \ \mbox{ and } \ \
\dim(A\otimes \mathcal{O}_\infty)\leq 2 \dim(A\otimes \mathcal{O}_2)+1.
\end{equation*}
\end{theorem}

\begin{proof}
This is a consequence of Theorem~\ref{thm:InclZUHF}, Theorem~\ref%
{thm:InclKirO2}, part~(2) of Proposition~\ref{prop:propertiesOzDim}, and
Proposition~\ref{Proposition:estimate-dimension}.
\end{proof}

We want to highlight two important consequences of Theorem~\ref%
{thm:O2OIUHFDimIneq}. One of them is obtained by letting $\dim $ be the
Rokhlin dimension. In this case, and again endowing $\mathcal{Z},\mathcal{O}%
_{2},\mathcal{O}_{\infty }$ and the UHF-algebra with the trivial $G$-action,
and all tensor products with the diagonal action, we deduce the following
dimensional inequalities (compare with Section~4 of~\cite%
{gardella_rokhlin_2014}).

\begin{corollary}
\label{cor:RokDimIneq} Let $A$ be a $G$-C*-algebra, and let $U$ be a
UHF-algebra of infinite type. Then 
\begin{equation*}
\dim _{\mathrm{Rok}}(A\otimes \mathcal{Z})\leq 2\dim _{\mathrm{Rok}%
}(A\otimes U)+1, \ \ \mbox{ and } \ \ 
\dim _{\mathrm{Rok}}(A\otimes \mathcal{O}_{\infty })\leq 2\dim _{\mathrm{Rok}%
}(A\otimes \mathcal{O}_{2})+1.
\end{equation*}
\end{corollary}

The other consequence is obtained be letting $\dim $ be either the nuclear
dimension or the decomposition rank. The estimates involving nuclear
dimension have previously been observed in \cite[Section~3]%
{barlak_rokhlin_2015}, while the estimates for the decomposition rank are
new.

\begin{corollary}
\label{cor:DimNucDRIneq} Let $A$ be a C*-algebra, and let $U$ be any
UHF-algebra of infinite type. Then 
\begin{equation*}
\dim _{\mathrm{nuc}}(A\otimes \mathcal{Z})\leq 2\dim _{\mathrm{nuc}%
}(A\otimes U)+1, \ \ \mbox{ and } \ \ 
\dim _{\mathrm{nuc}}(A\otimes \mathcal{O}_{\infty })\leq 2\dim _{\mathrm{nuc}%
}(A\otimes \mathcal{O}_{2})+1\text{.}
\end{equation*}%
Furthermore,%
\begin{equation*}
\mathrm{dr}(A\otimes \mathcal{Z})\leq 2\mathrm{dr}(A\otimes U)+1\text{.}
\end{equation*}
\end{corollary}

\subsection{Rokhlin dimension and strongly self-absorbing 
\texorpdfstring{$G$}{G}-C*-algebras}

In this subsection, and since we consider different actions on the same C*-algebra, 
we denote by $\dim _{\mathrm{Rok}}^{\mathrm{c}}(A,\alpha )$ the Rokhlin dimension with
commuting towers of the $G$-C*-algebra $(A,\alpha )$. The following is one
of our main technical results.

\begin{theorem}
\label{Theorem:fixed} Let $\alpha $ be a continuous action of $G$ on a
C*-algebra $A$. If $\mathrm{\dim }_{\mathrm{Rok}}(A,\alpha )\leq d$, then $%
(A,\iota _{A})\precsim _{d}(A,\alpha )$. If $\mathrm{\dim }_{\mathrm{Rok}}^{%
\mathrm{c}}(A,\alpha )\leq d$, then $(A,\iota _{A})\precsim _{d}^{\mathrm{c}%
}(A,\alpha )$.
\end{theorem}

\begin{proof}
We prove the first assertion. The proof of the second assertion is
analogous. Fix a nonprincipal ultrafilter $\mathcal{U}$ over $\mathbb{N}$.
We denote by $F_{\mathcal{U}}(A)$ the Kirchberg invariant of $(A,\iota _{A})$
(endowed with the trivial action), and by $F_{\mathcal{U}}^{G}(A)$ the
Kirchberg invariant of $(A,\alpha )$ (endowed with the canonical $G$-action
obtained from $\alpha $). Since $\mathrm{dim}_{\mathrm{Rok}}^{\mathrm{c}%
}(\alpha )\leq d$, it follows from the reformulation of Rokhlin dimension in
terms of commutant $d$-containment and Proposition \ref{Proposition:F} that
for any separable C*-subalgebra $C$ of $F_{\mathcal{U}}(A) $ there exist $G$%
-equivariant completely positive contractive order zero maps $\psi
_{0},\ldots ,\psi _{d}\colon C(G)\rightarrow C^{\prime }\cap F_{\mathcal{U}%
}^{G}(A)$ such that $\psi _{0}+\cdots +\psi _{d}$ is unital.

Fix a separable C*-subalgebra $C$ of $F_{\mathcal{U}}^{G}(A)$ containing $A$%
. When $G$ is finite, the maps witnessing that $(A,\alpha )\precsim
_{d}(A,\iota _{A})$ can be constructed explicitly, so we outline this first.
For $g\in G$, let $\delta _{g}\in C(G)$ be the characteristic function of $%
\{g\}$. Define maps $\eta _{j}\colon C\rightarrow F_{\mathcal{U}}^{G}(A)$,
for $j=0,\ldots ,d$, by $\eta _{j}(x)=\sum\nolimits_{g\in G}\psi _{j}(\delta
_{g})\alpha _{g}(x)$. Then these maps witness the fact that $(A,\iota _{A})$
is $G$-equivariantly commutant $d$-contained in $(A,\alpha )$.

Suppose now that $G$ is an arbitrary compact second countable group. Below,
if $a$ and $b$ are elements of a C*-algebra and $\varepsilon >0$, we write $%
a\thickapprox _{\varepsilon }b$ to mean that $\left\Vert a-b\right\Vert
<\varepsilon $. Let $\rho$ be a left invariant metric on $G$. Fix a finite
subset $F $ of positive elements of $C$ and $\varepsilon >0$. The argument
in \cite[Proposition 2.11]{gardella_crossed_2014} shows that there exist $%
\delta >0$, a finite subset $K$ of $G$, and a partition of unity $( f_{g})
_{g\in K}$ of $G$ satisfying:

\begin{enumerate}
\item $f_{g}\in C( G) $ is a positive contraction for all $g\in G$;

\item $f_{g}$ and $f_{h}$ are orthogonal whenever $g,h\in G$ satisfy $\rho(
g,h) >\delta $;

\item for every $a\in F$, we have $\alpha_g(a)\thickapprox _{\varepsilon
}\alpha_h(a)$ whenever $g,h\in G$ satisfy $\rho( g,h) <\delta $; and

\item $\alpha _{h}( \sum\nolimits_{g\in K}\psi _{j}(f_{g})a) \thickapprox
_{\varepsilon }\sum_{g\in K}\psi _{j}(f_{g})a$ for all $h\in G$ and all $%
a\in F$.
\end{enumerate}

Define now $\eta _{j}\colon C\rightarrow F_{\mathcal{U}}^{G}(A)$ by $\eta
_{j}(x)=\sum_{g\in K}\psi _{j}(f_{g})\alpha _{g}(x)$ for $j=0,\ldots ,d$.
Observe that $\eta _{0},\ldots ,\eta _{d}$ are completely positive
contractive maps. Furthermore, for every $0\leq j\leq d$, if $a,b\in F$
satisfy $ab\thickapprox _{\varepsilon }0$, then (1) and (2) imply that%
\begin{align*}
\eta _{j}(a)\eta _{j}(b) =&\sum\limits_{g,h\in K}\psi _{j}(f_{g})\psi
_{j}(f_{h})\alpha _{g}(a)\alpha _{h}(b) \thickapprox_{\varepsilon }\sum\limits_{g,h\in K}\psi _{j}(f_{g})\psi
_{j}(f_{h})\alpha _{g}(ab)\thickapprox _{\varepsilon }0\text{.}
\end{align*}

By (3), we have $\alpha _{g}(\eta _{j}(a))\thickapprox _{\varepsilon }\eta
_{j}(a)$ for every $g\in G$, every $j=0,\ldots ,d$, and every $a\in F$.
Since $\varepsilon >0$ and $F\subset C_{+}$ are arbitrary, it follows from
Proposition \ref{Proposition:F} that there exist $G$-equivariant completely
positive contractive order zero maps $\eta _{0},\ldots ,\eta _{d}\colon
C\rightarrow F_{\mathcal{F}}^{G}(A)$ with unital sum. Since this is true for
every separable C*-subalgebra $C$ of $F_{\mathcal{F}}(A)$, we conclude that $%
(A,\iota _{A})\precsim _{d}(A,\alpha )$, as desired.
\end{proof}

The following corollary is then a consequence of Theorem \ref{Theorem:fixed}
and Proposition \ref{Proposition:commuting}.

\begin{corollary}
\label{Corollary:fixed} Let $\dim $ be a dimension function for $G$%
-C*-algebras, $A$ is a C*-algebra, $\alpha $ is a continuous action of a $G$
on $A$, and $\iota _{A}$ is the trivial $G$-action on $A$. If $\dim $ is
positively existentially axiomatizable, then 
\begin{equation*}
\dim (A,\alpha )+1\leq (\dim _{\mathrm{Rok}}(A,\alpha )+1)(\dim (A,\iota
_{A})+1)\text{.}
\end{equation*}%
If $\dim $ is commutant positively existentially axiomatizable, then 
\begin{equation*}
\dim (A,\alpha )+1\leq (\dim _{\mathrm{Rok}}^{\mathrm{c}}(A,\alpha )+1)(\dim
(A,\iota _{A})+1)\text{.}
\end{equation*}
\end{corollary}

We now arrive at one of the main results of this section. It asserts that,
for a strongly self-absorbing C*-algebra, $G$-actions with finite Rokhlin
dimension with commuting towers, on $D$-absorbing C*-algebras, automatically
absorb the trivial action on $D$.

\begin{theorem}
\label{Theorem:absorption} Let $D$ be a strongly self-absorbing C*-algebra,
let $A$ be a separable $D$-absorbing C*-algebra, and let $\alpha \colon
G\rightarrow \mathrm{Aut}(A)$ be an action of $G$ with $\dim _{\mathrm{Rok}%
}^{\mathrm{c}}(A,\alpha )<\infty $. Then $(A,\alpha )$ is $G$-equivariantly $%
(D,\iota _{D})$-absorbing.
\end{theorem}

\begin{proof}
Let $\dim _{D}$ be the $\left\{ 0,\infty \right\} $-valued dimension
function for $G$-C*-algebras which is finite if and only if the given $G$%
-C*-algebra is $(D,\iota _{D})$-absorbing. The action $(D,\iota _{D})$ is
unitarily regular since it absorbs $(\mathcal{Z},\iota _{\mathcal{Z}})$
tensorially; see \cite[Proposition 1.20]{szabo_strongly_2016}. It follows
from this and Proposition~\ref{Proposition:dimDCommPosExisAxiom} that $\dim
_{D}$ is commutant positively existentially axiomatizable with commuting
towers. The result now follows from Corollary \ref{Corollary:fixed}.
\end{proof}

\begin{corollary}
\label{Corollary:crossed-commuting} Let $D$ be a strongly self-absorbing
C*-algebra, and let $(A,\alpha )$ be a separable $G$-C*-algebra with $\dim _{%
\mathrm{Rok}}^{\mathrm{c}}(A,\alpha )<\infty $. If $A$ is (nonequivariantly) 
$D$-absorbing, then so are $A^{G}$ and $A\rtimes_{\alpha} G$.
\end{corollary}

\begin{proof}
By Theorem~\ref{Theorem:absorption}, there is a $G$-equivariant isomorphism
between $(A,\alpha)$ and $(A\otimes D,\alpha\otimes \iota_D)$. Upon taking
crossed products, we deduce that 
\begin{equation*}
A\rtimes_{\alpha} G\cong (A\otimes D)\rtimes_{\alpha\otimes \iota_D}G\cong
(A\rtimes_{\alpha} G)\otimes D,
\end{equation*}
so $A\rtimes_{\alpha} G$ is $D$-absorbing. The same applies to the fixed point
algebra, since we have $A^{\alpha}\cong (A\otimes
D)^{\alpha\otimes\iota_D}=A^{\alpha}\otimes D$.
\end{proof}

Corollary \ref{Corollary:crossed-commuting} is a significant generalization
of previously known results concerning Jiang-Su absorption: for finite
groups this was shown by Hirshberg-Winter-Zacharias in \cite[Theorem 5.9]%
{hirshberg_rokhlin_2015}, and in \cite[Theorem 5.4.4]{gardella_compact_2015}
by the first-named author for compact groups. Similar results have been
independently obtained with different methods in \cite{gardella_rokhlin_2016}%
.

Observe that in the next result we do not require the strongly
self-absorbing action to be unitarily regular, unlike in Theorem~\ref%
{thm:trivialBundle} or \ref{Corollary:absorbption-containment}.

\begin{theorem}
\label{Theorem:absorption+} Let $A$ be a separable C*-algebra, and let $%
\alpha \colon G\rightarrow \mathrm{Aut}(A)$ be an action. Suppose that $A$
absorbs a strongly self-absorbing C*-algebra $D$, and let $\delta \colon
G\rightarrow \mathrm{Aut}(D)$ be any strongly self-absorbing action.

\begin{enumerate}
\item If $\dim _{\mathrm{Rok}}(A,\alpha )=d<\infty $, then $(D,\delta
)\precsim _{d}(A,\alpha )$.

\item If $\dim _{\mathrm{Rok}}^{\mathrm{c}}(A,\alpha )=d<\infty $, then $%
(D,\delta )\precsim _{d}^{\mathrm{c}}(A,\alpha )$. Moreover, $(A,\alpha )$
is $G$-equivariantly $(D,\delta )$-absorbing.
\end{enumerate}
\end{theorem}

\begin{proof}
(1). Suppose that $\dim _{\mathrm{Rok}}( A,\alpha ) =d<+\infty $. By Lemma %
\ref{Lemma:rokhlin-and-oz}, the second-factor embedding $\theta\colon (
A,\alpha ) \rightarrow ( C(G) \otimes A,\mathtt{Lt}\otimes \alpha ) $ has $G$%
-equivariant order zero dimension at most $d$. Since $A\otimes D\cong D$, it
follows from Proposition~2.3 in~\cite{gardella_crossed_2014} that $( C(G)
\otimes A,\mathtt{Lt}\otimes \alpha ) $ is conjugate to $( C(G) \otimes
A\otimes D,\mathtt{Lt}\otimes \alpha \otimes \delta )$. In other words, $(
C(G) \otimes A,\mathtt{Lt}\otimes \alpha ) $ is $G$-equivariantly $(
D,\delta ) $-absorbing. By the implication (ii)$\Rightarrow $(iv) in \cite[%
Theorem 4.29]{barlak_sequentially_2016} the first-factor embedding $\eta
\colon ( C( G) \otimes A,\mathtt{Lt}\otimes \alpha ) \rightarrow ( C(G)
\otimes A\otimes D,\mathtt{Lt}\otimes \alpha \otimes \delta ) $ has $G$%
-equivariant order zero dimension zero. By item (1) in Proposition \ref%
{prop:propertiesOzDim}, the composition 
\begin{equation*}
\eta \circ \theta \colon( A,\alpha ) \rightarrow ( C(G) \otimes A\otimes D,%
\mathtt{Lt}\otimes \alpha \otimes \delta )
\end{equation*}
has dimension at most $d$. Observe that $( \eta \circ \theta ) (A) =1_{C(G)
}\otimes a\otimes 1_{D}$ for all $a\in A$. It follows that the first-factor
embedding $( A,\alpha ) \rightarrow ( A\otimes D,\alpha \otimes \delta ) $
has $G$-equivariant dimension at most $d$ (this embedding is really just $%
\eta\circ\theta$, once its codomain is truncated). We conclude from
Proposition~\ref{Proposition:dimension-implies-containment} that $( A\otimes
D,\alpha \otimes \delta ) \precsim _{d}(A,\alpha )$, and in particular $(
D,\delta ) \precsim _{d}(A,\alpha )$.

(2). Fix a nonprincipal ultrafilter $\mathcal{U}$ over $\mathbb{N}$. Suppose
that $\dim _{\mathrm{Rok}}^{\mathrm{c}}(A,\alpha )\leq d$. We want to show
that $(D,\delta )\precsim _{d}^{\mathrm{c}}(A,\alpha )$. In view of Remark %
\ref{Remark:containment-ssa}, it is enough to show that there exist $G$%
-equivariant unital completely positive contractive maps $\eta _{0},\ldots
,\eta _{d}\colon (D,\delta )\rightarrow (F_{\mathcal{U}}^{G}(A),\alpha )$
with commuting ranges such that $\sum_{j=0}^d\eta _{j}$ is unital. In
order to illustrate the ideas of the proof, we begin by considering the
unital case, since the argument is easier to follow in this case.

Assume that $A$ is unital, so that $F_{\mathcal{U}}^{G}(A)$ is equal to $%
A^{\prime }\cap \prod_{\mathcal{U}}^{G}A$. By assumption, there exist $G$%
-equivariant completely positive contractive order zero $A$-bimodule maps $%
\psi _{0},\ldots ,\psi _{d}\colon (C(G)\otimes A,\mathtt{Lt}\otimes \alpha
)\rightarrow (\prod_{\mathcal{U}}^{G}A,\alpha )$, such that $(\psi
_{0}+\cdots +\psi _{d})\circ (1\otimes \mathrm{id}_{A})$ is the diagonal
inclusion of $A$ into $\prod_{\mathcal{U}}^{G}A$, and $\psi _{j}(C(G)\otimes
1)$ and $\psi _{k}(C(G)\otimes 1)$ commute for $0\leq j<k\leq d$. Define $C$
to be the separable C*-subalgebra of $\prod_{\mathcal{U}}^{G}A$ generated by 
$A$ together with the ranges of $\psi _{0},\ldots ,\psi _{d}$. By Theorem~%
\ref{Theorem:absorption}, $(A,\alpha )$ is $(D,\iota _{D})$-absorbing. Use
Theorem \ref{Theorem:D-absorption} to find a $G$-equivariant *-homomorphism $%
\theta :D\rightarrow C^{\prime }\cap \prod_{\mathcal{U}}^{G}A$. For $%
j=0,\ldots ,d$, define a $G$-equivariant completely positive contractive
order zero $A$-bimodules map 
\begin{equation*}
\phi _{j}\colon (C(G)\otimes A\otimes D,\mathtt{Lt}\otimes \alpha \otimes
\iota _{D})\rightarrow (\prod\nolimits_{\mathcal{U}}^{G}A,\alpha )
\end{equation*}%
by setting $\phi _{j}(f\otimes a\otimes d)=\psi _{j}(f\otimes a)\theta (d)$
for all $f\in C(G)$, all $a\in A$, and all $d\in D$. Observe that $\phi
_{j}(C(G)\otimes 1_{A}\otimes D)$ commutes with $\phi _{k}(C(G)\otimes
1_{A}\otimes D)$ whenever $0\leq j<k\leq d$.

Fix a $G$-equivariant isomorphism $( C( G) \otimes D,\mathtt{Lt}\otimes
\iota _{D}) \to ( C( G) \otimes D,\mathtt{Lt}\otimes \delta ) $, and tensor
it with the identity on $A$ to obtain a $G$-equivariant *-isomorphism 
\begin{equation*}
\pi \colon ( C( G) \otimes A\otimes D,\mathtt{Lt}\otimes \alpha \otimes
\delta ) \rightarrow ( C( G) \otimes A\otimes D,\mathtt{Lt}\otimes \alpha
\otimes \iota _{D})
\end{equation*}%
satisfying $\pi ( 1_{C(G)}\otimes a\otimes 1_D) =1_{C(G)}\otimes a\otimes
1_D $ for every $a\in A$. For $j=0,\ldots,d$, let 
\begin{equation*}
\eta _{j}\colon ( D,\delta ) \rightarrow (A^{\prime }\cap \prod\nolimits_{%
\mathcal{U}}^{G}A,\alpha )
\end{equation*}%
be the $G$-equivariant completely positive order zero map given by $\eta
_{j}( d) =( \phi _{j}\circ \pi ) ( 1_{C(G)}\otimes 1_A\otimes d) $ for all $%
d\in D$. Then $\eta_j(D)$ commutes with $\eta_k(D)$ whenever $0\leq j<k\leq
d $, and $\sum_{j=0}^d\eta_j(1)=1$. We conclude that $( D,\delta ) \precsim
_{d}^{\mathrm{c}}( A,\alpha ) $.

We consider now the general case when $A$ is not necessarily unital. Fix
maps $\psi _{0},\ldots ,\psi _{d}\colon ( C( G) \otimes A,\mathtt{Lt}\otimes
\alpha ) \rightarrow (\prod_{\mathcal{U}}^{G}A,\alpha )$ as before. Let $C$
the separable C*-subalgebra of $\prod_{\mathcal{U}}^{G}A$ generated by the
ranges of $\psi _{0},\ldots ,\psi _{d}$. As above, find a $G$-equivariant
*-homomorphism 
\begin{equation*}
\theta \colon ( D,\iota _{D}) \rightarrow \frac{C^{\prime }\cap \prod_{%
\mathcal{U}}^{G}A}{\mathrm{Ann}( A,C^{\prime }\cap \prod_{\mathcal{U}}^{G}A) 
}\text{.}
\end{equation*}%
Let $\pi $ be as before, and consider the canonical $G$-equivariant
*-homomorphism 
\begin{equation*}
\Psi \colon C\otimes _{\max }\frac{C^{\prime }\cap \prod_{\mathcal{U}}^{G}A}{%
\mathrm{Ann}\left( A,C^{\prime }\cap \prod_{\mathcal{U}}^{G}A\right) }%
\rightarrow \prod\nolimits_{\mathcal{U}}^{G}A
\end{equation*}%
defined as in the proof of Lemma \ref{Lemma:oz-dimension-commutant-embedding}%
. For $j=0,\ldots, d$, set 
\begin{equation*}
\phi_j=\Psi\circ(\psi_j\otimes\theta)\colon C(G)\otimes A\otimes D\to
\prod\nolimits_{\mathcal{U}}^{G}A.
\end{equation*}
Then $\phi_j$ is a $G$-equivariant order zero $A$-bimodule map.

Let $(u_{\lambda })_{\lambda \in \Lambda }$ be an increasing approximate
identity for $A$. For every $\lambda \in \Lambda $ and $j=0,\ldots ,d$,
define 
\begin{equation*}
\eta _{j,\lambda }\colon (D,\delta )\rightarrow (\prod\nolimits_{\mathcal{U}%
}^{G}A,\alpha )
\end{equation*}%
by $\eta _{j,\lambda }(d)=(\phi _{j}\circ \pi )(1_{C(G)}\otimes u_{\lambda
}\otimes d)$. These maps satisfy are completely positive contractive order
zero, have commuting ranges, and 
\begin{equation*}
\left\Vert \lbrack \eta _{j,\lambda }(d),a]\right\Vert \rightarrow 0\ \ %
\mbox{ and }\ \ \left\Vert a{}(\eta _{0,\lambda }+\cdots +\eta _{0,\lambda
})(1)-a\right\Vert \rightarrow 0\text{,}
\end{equation*}%
for every $d\in D$ and $a\in A$. By countable saturation of $\left( \prod_{%
\mathcal{U}}^{G}A,\alpha \right) $, there exist $G$-equivariant completely
positive order zero maps 
\begin{equation*}
\tilde{\eta}_{j}\colon (D,\delta )\rightarrow (A^{\prime }\cap
\prod\nolimits_{\mathcal{U}}^{G}A,\alpha )
\end{equation*}%
satisfying $a\left[ (\tilde{\eta}_{0,\lambda }+\cdots +\tilde{\eta}%
_{0,\lambda })(1)\right] =a$ for every $a\in A$. Composing such maps with
the canonical quotient mapping 
\begin{equation*}
A^{\prime }\cap \prod\nolimits_{\mathcal{U}}^{G}A\rightarrow \frac{A^{\prime
}\cap \prod\nolimits_{\mathcal{U}}^{G}A}{\mathrm{Ann}\left( A,A^{\prime
}\cap \prod\nolimits_{\mathcal{U}}^{G}A\right) }=F_{\mathcal{U}}^{G}(GA)
\end{equation*}
gives $G$-equivariant completely positive order zero maps with commuting
ranges $\eta _{0},\ldots ,\eta _{d}\colon (D,\delta )\rightarrow (F_{%
\mathcal{U}}^{G}(A),\alpha )$, which witness the fact that $(D,\delta
)\precsim _{d}^{\mathrm{c}}(A,\alpha )$.

We now justify the second claim. When $\delta $ is unitarily regular, the
fact that $(D,\delta )\precsim _{d}^{\mathrm{c}}(A,\alpha )$ implies that $%
(A,\alpha )$ is $G$-equivariantly $(D,\delta )$-absorbing is a consequence
of Corollary \ref{Corollary:absorbption-containment}. Now suppose that $%
\delta $ is an arbitrary strongly self-absorbing action, and consider $%
\delta ^{\mathcal{Z}}=\delta \otimes \mathrm{id}_{\mathcal{Z}}$, regarded as
an action of $G$ on $D\otimes \mathcal{Z}\cong D$. Then $\delta ^{\mathcal{Z}%
}$ is unitarily regular, and hence $\alpha $ absorbs $\delta ^{\mathcal{Z}}$
by the above paragraph. Since $\mathrm{id}_{\mathcal{Z}}$ is unitarily
regular and $A$ is $\mathcal{Z}$-absorbing (because it is $D$-absorbing), we
also deduce that $\alpha $ absorbs $\mathrm{id}_{\mathcal{Z}}$. Putting
these things together, we deduce that 
\begin{equation*}
\alpha \cong \alpha \otimes \delta ^{\mathcal{Z}}=\alpha \otimes \mathrm{id}%
_{\mathcal{Z}}\otimes \delta \cong \alpha \otimes \delta .
\end{equation*}%
In other words, $\alpha $ absorbs $\delta $, and the proof is finished.
\end{proof}

Theorem \ref{Theorem:absorption+} has a number of new and strong
consequences, as already the case of the trivial action on $D$ is new. For
example, we derive now some dimension reduction type results. Roughly
speaking, these statements say that, in some contexts, the Rokhlin property
is \emph{equivalent} to finite Rokhlin dimension with commuting towers. (By
definition, a compact group action has the Rokhlin property if it has
Rokhlin dimension zero.)

These are useful results, since proving directly that an action has the
Rokhlin property is often challenging, and there are not many tools
available. On the other hand, Rokhlin dimensional estimates are much easier
to come by, particularly for finite groups. It follows from our results
that, in some cases, knowing that the Rokhlin dimension is finite is enough
to deduce the Rokhlin property. Having access to the Rokhlin property is
extremely valuable, since it entails classifiability (of the action), and
the structure of the crossed product is extremely well-understood; see \cite%
{gardella_crossed_2014}.

\begin{corollary}
\label{cor:DimReductionO2} Let $A$ be an $\mathcal{O}_{2}$-absorbing
C*-algebra, and let $\alpha \colon G\rightarrow \mathrm{Aut}(A)$ be an
action with $\mathrm{dim}_{\mathrm{Rok}}^{\mathrm{c}}(\alpha )<\infty $.
Then $\alpha $ has the Rokhlin property.
\end{corollary}

\begin{proof}
Let $\delta\colon G\to\mathrm{Aut}(\mathcal{O}_2)$ be any action with the
Rokhlin property; one such action is constructed in \cite%
{gardella_compact_2015}. By the classification theorem in \cite%
{gardella_classification_2016}, the action $\delta$ is strongly
self-absorbing. It follows from Theorem~\ref{Theorem:absorption+} that $%
\alpha$ absorbs $\delta$, so $\alpha$ has the Rokhlin property.
\end{proof}



\begin{corollary}
\label{cor:DimReductionUHF} Let $G$ be a finite group, let $A$ be an $%
M_{|G|^\infty}$-absorbing C*-algebra, and let $\alpha\colon G\to\mathrm{Aut}%
(A)$ be an action with $\mathrm{dim}_{\mathrm{Rok}}^{\mathrm{c}%
}(\alpha)<\infty$. Then $\alpha$ has the Rokhlin property. This in
particular applies to Cuntz algebras of the form $\mathcal{O}_{n|G|}$.
\end{corollary}

\begin{proof}
Let $\delta\colon G\to\mathrm{Aut}(M_{|G|^\infty})$ be the infinite tensor
product of conjugation by the left regular representation of $G$ on $%
\ell^2(G)$. This action is well-known to have the Rokhlin property, since $%
C(G)$ embeds equivariantly into $B(\ell ^{2}(G))\cong M_{|G|}$ as
multiplication operators. It is elementary to show that such an action is
strongly self-absorbing. It follows from Theorem~\ref{Theorem:absorption+}
that $\alpha$ absorbs $\delta$, so $\alpha$ has the Rokhlin property.
\end{proof}

Finally, we obtain some new Rokhlin-dimensional estimates. The following one
represents a satisfactory parallel with the $\{0,1,\infty\}$-type behaviour
that nuclear dimension and decomposition rank tend to have in the
noncommutative setting. It is also particularly satisfactory, since proving
finiteness of the Rokhlin dimension is a far easier task than proving that
it is (at most) 1.

\begin{corollary}
\label{cor:estimate} Let $G$ be a finite group, let $A$ be a C*-algebra, and
let $\alpha \colon G\rightarrow \mathrm{Aut}(A)$ satisfy $\mathrm{dim}_{%
\mathrm{Rok}}^{\mathrm{c}}(\alpha )<\infty $. Then 
\begin{equation*}
\mathrm{dim}_{\mathrm{Rok}}(\alpha \otimes \mathrm{id}_{\mathcal{Z}})\leq 1.
\end{equation*}%
If $A$ is $\mathcal{Z}$-absorbing, then%
\begin{equation*}
\dim _{\mathrm{Rok}}\left( \alpha \right) \leq 1\text{.}
\end{equation*}
\end{corollary}

\begin{proof}
We apply Corollary~\ref{cor:RokDimIneq} with $U=M_{|G|^{\infty }}$, so that $%
\mathrm{dim}_{\mathrm{Rok}}(\alpha \otimes \mathrm{id}_{U})=0$ by Corollary~%
\ref{cor:DimReductionUHF}. The second assertion follows from the first one
and Theorem \ref{Theorem:absorption+}.
\end{proof}

The next result is a dynamical version of the main result of~\cite%
{tikuisis_decomposition_2014}, which states that dr$(C(X)\otimes\mathcal{Z}%
)\leq 2$.

\begin{corollary}
\label{cor:estimateCommutative} Let $G$ be a finite group, let $X$ be a
compact Hausdorff space, and let $\alpha\colon G\to\mathrm{Aut}(C(X))$ be
induced by a free action of $G$ on $X$. Then 
\begin{equation*}
\mathrm{dim}_{\mathrm{Rok}}(\alpha\otimes\mathrm{id}_{\mathcal{Z}})\leq 1.
\end{equation*}
\end{corollary}

\begin{proof}
This is an immediate consequence of Corollary~\ref{cor:estimate}, since $%
\alpha $ has finite Rokhlin dimension (with commuting towers) by Theorem~4.2
in~\cite{gardella_rokhlin_2014}.
\end{proof}

\appendix

\section*{Appendix}

\renewcommand{\thesection}{A} \setcounter{theorem}{0}

Here we recall some basic notions from first order logic for metric
structures. We will consider here and in the following a \emph{multi-sorted
language $\mathcal{L}$}. This is endowed with a collection of \emph{sorts}.
A collection of \emph{domains}, as well as a collection of \emph{%
pseudometric symbols}, are associated with each given sort. We refer the
reader to \cite[Appendix]{gardella_model_2017} for details concerning the
semantic and syntax in this setting, which is slightly more general than the
usual setting of logic for metric structures as considered in \cite%
{ben_yaacov_model_2008} or \cite{farah_model_2016, farah_model_2014}. In
particular, the notions of (positive/positive primitive) $\mathcal{L}$%
-formula and $\mathcal{L}$-type, (positive quantifier-free) $\kappa $%
-saturated structure, $\mathcal{L}$-definable set, reduced product, and
ultraproduct, can be found in \cite[Appendix]{gardella_model_2017}.

An $\mathcal{L}$-\emph{embedding} $\theta :M\rightarrow N$ between $\mathcal{%
L}$-structures is a collection of functions $\theta _{\mathcal{S}}\colon 
\mathcal{S}^{M}\rightarrow \mathcal{S}^{N}$ such that $\theta _{\mathcal{S}%
}(D^{M})\subseteq D^{N}$ for every domain $D$ associated with $\mathcal{S}$,
and $\varphi (\theta (\overline{a}))=\varphi (\overline{a})$ for any
quantifier-free\ $\mathcal{L}$-formula $\theta (\overline{x})$ and tuple $%
\overline{a}$ in $M$. Similarly, an $\mathcal{L}$-\emph{morphism }$\theta $
between $\mathcal{L}$-structures is a a collection of functions $\theta _{%
\mathcal{S}}\colon \mathcal{S}^{M}\rightarrow \mathcal{S}^{N}$ such that $%
\theta _{\mathcal{S}}(D^{M})\subseteq D^{N}$ for every domain associated
with $\mathcal{S}$, such that $\varphi (\theta (\overline{a}))\leq \varphi (%
\overline{a})$ for any atomic\ $\mathcal{L}$-formula $\theta (\overline{x})$
and any tuple $\overline{a}$ in $M$. If $\mathcal{F}$ is a filter over a set 
$I$ and $\left( M_{i}\right) _{i\in I}$ is an $I$-sequence of $\mathcal{L}$%
-structures, we let $\prod_{\mathcal{F}}M_{i}$ be the corresponding reduced
product. In the case when $\left( M_{i}\right) _{i\in I}$ is constantly
equal to a fixed $\mathcal{L}$-structure $M$, one obtaines the reduced $%
\mathcal{L}$-power $\prod_{\mathcal{F}}M$ of $M$ with respect to $\mathcal{F}
$.

\subsection{Existential theories\label{Subsection:existential}}

Suppose that $M$ is an $\mathcal{L}$-structure. The \emph{existential $%
\mathcal{L}$-theory }$\mathrm{Th}_{\exists }^{\mathcal{L}}(M)$ of $M$ is the
function $\varphi \mapsto \varphi ^{M}$ that assigns to an existential $%
\mathcal{L}$-sentence $\varphi $ its value $\varphi ^{M}$ in $M$. We say
that $M$ is \emph{weakly $\mathcal{L}$-contained }in $N$ if $\mathrm{Th}%
_{\exists }^{\mathcal{L}}(M)\leq \mathrm{Th}_{\exists }^{\mathcal{L}}(N)$,
and \emph{weakly $\mathcal{L}$-equivalent }to $N$ if $M$ and $N$ have the
same existential $\mathcal{L}$-theory. We will identify the existential $%
\mathcal{L}$-theory of an $\mathcal{L}$-structure with its weak $\mathcal{L}$%
-equivalence class. It follows from saturation of ultrapowers and \L os'
theorem that $M$ is weakly $\mathcal{L}$-contained in $N$ if and only if for
some (equivalently, any) countably incomplete ultrafilter $\mathcal{U}$,
every separable substructure of $M$ admits an $\mathcal{L}$-embedding into $%
\prod_{\mathcal{U}}N$. This is equivalent to the assertion that if a
quantifier-free\ $\mathcal{L}$-type is approximately realized in $M$, then
it is approximately realized in $N$.

A class $\mathfrak{C}$ of structures is said to be \emph{existentially $%
\mathcal{L}$-axiomatizable} if there is a collection $\left( \varphi
_{i}\right) $ of existential $\mathcal{L}$-sentences such that, for every $%
\mathcal{L}$-structure $M$, $M\in \mathfrak{C}$ if and only if $\varphi
_{i}^{M}\leq 0$ for every $i\in I$. More generally, we consider the
following notion, which has been introduced in \cite[Definition 5.7.1]%
{farah_model_2016}.

\begin{definition}
\label{Definition:uniform-family}A class $\mathfrak{C}$ of (separable)
structures is said to be \emph{definable by a uniform family of existential $%
\mathcal{L}$-formulas} if, for every $k\in \mathbb{N}$, there exist $%
n_{k}\in \mathbb{N}$ and an uniformly equicontinuous collections $\mathcal{F}%
_{k}(x_{1},\ldots ,x_{n_{k}})$ of existential $\mathcal{L}$-formulas, such
that a (separable) $\mathcal{L}$-structure $M$ belongs to $\mathfrak{C}$ if
and only if for every $k\in \mathbb{N}$ and every $\overline{a}\in M^{n_{k}}$
there exists $\varphi \in \mathcal{F}_{k}$, such that $M\models \varphi (%
\overline{a})\leq 1/k$.
\end{definition}

Observe that if $\mathfrak{C}$ is a class of (separable) structures
definable by a uniform family of existential $\mathcal{L}$-formulas, then $%
\mathfrak{C}$ is closed under (countable) direct limits. We say that a
property is definable by a\emph{\ }uniform family of existential $\mathcal{L}
$-formulas if the class of $\mathcal{L}$-structures satisfying that property
is.

The notions of existential positive $\mathcal{L}$-theory, positive weak $%
\mathcal{L}$-containment, positive weak $\mathcal{L}$-equivalence,
positively existentially $\mathcal{L}$-axiomatizable class, and class
definable by a uniform family of existential positive primitive $\mathcal{L}$%
-formulas are defined as above, by only considering existential positive $%
\mathcal{L}$-formulas. It follows from \L os' Theorem 
and Proposition \ref{Proposition:good} below that 
$M$ is positively weakly $\mathcal{L}$-contained in $N$ if and only if for
some (equivalently, any) countably incomplete filter $\mathcal{F}$, every
separable substructure of $M $ admits an $\mathcal{L}$-morphism to $\prod_{%
\mathcal{F}}N$.

\subsection{Existential theories of embeddings\label%
{Subsection:existential-embedding}}

Let $A,M$ be $\mathcal{L}$-structures and $\theta \colon
A\rightarrow M$ an $\mathcal{L}$-embedding. We can regard $(M,\theta
_{M}) $ as a structure in the language $\mathcal{L}(A)$ obtained by adding a
constant symbol $c_{a}$ for any element $a\in A$. The interpretation of $%
c_{a}$ in $(M,\theta )$ is the image $\theta (a)$ of $a$ under $\theta $.
One can then define the notions of quantifier-free\ $\mathcal{L}(A)$-formula
and quantifier-free\ $\mathcal{L}(A)$-type. The same definition as in
Subsubsection \ref{Subsection:existential} gives the notion of weak $%
\mathcal{L}$-containment, weak $\mathcal{L}$-equivalence, and existential $%
\mathcal{L}$-theory for embeddings $\theta _{M}\colon A\rightarrow M$ and $%
\theta _{N}\colon A\rightarrow N$. As in the case of $\mathcal{L}$%
-structures, one can say that $\theta _{M}$ is weakly $\mathcal{L}$%
-contained in $\theta _{N}$ if and only if for any separable substructures $%
A_{0}\subset A$ and $M_{0}\subset M$ such that $\theta _{M}(A_{0})\subset
M_{0}$, and for some (equivalently, any) countably incomplete ultrafilter $%
\mathcal{U}$, there exists an $\mathcal{L}$-embedding $\eta \colon
M_{0}\rightarrow \prod_{\mathcal{U}}N$ such that $\Delta _{N}\circ \theta
_{N}|_{A_{0}}=\eta \circ \theta _{M}|_{A_{0}}$.

\begin{definition}
\label{Definition:existential-embedding}An $\mathcal{L}$-embedding $\theta
_{M}\colon A\rightarrow M$ is said to be \emph{$\mathcal{L}$-existential }if
for any quantifier-free\ $\mathcal{L}$-formula $\varphi (\overline{x},%
\overline{y})$ and any tuple $\overline{a}\in A$, the value of $%
\inf\nolimits_{\overline{y}}\varphi (\overline{a},\overline{y})$ in $A$ is
the same as the value of $\inf\nolimits_{\overline{y}}\varphi (\theta _{M}(%
\overline{a}),\overline{y})$ in $M$.
\end{definition}

It is easy to see that $\theta _{M}\colon A\rightarrow M$ is $\mathcal{L}$%
-existential if and only if $\theta _{M}$ is weakly $\mathcal{L}$-contained
in the identity embedding $\mathrm{id}_{A}\colon A\rightarrow A$.

Similarly, one can define the notion of positively $\mathcal{L}$-existential 
$\mathcal{L}$-embedding $\theta _{M}\colon A\rightarrow M$, by only
considering existential positive $\mathcal{L}$-formulas. The following fact
follows easily from the definitions.

\begin{proposition}
\label{Proposition:definable-class}Suppose that $\mathfrak{C}$ is a class of
structures that is definable by a uniform family of existential positive $%
\mathcal{L}$-formulas.\ If $\theta _{M}\colon A\rightarrow M$ is a
positively $\mathcal{L}$-existential $\mathcal{L}$-embedding and $M\in 
\mathfrak{C}$, then $A\in \mathfrak{C}$.
\end{proposition}

\begin{proposition}
\label{Proposition:limit} Let $\Lambda $ be a directed set. The following
properties follow easily from the definition.

\begin{enumerate}
\item The composition of positive $\mathcal{L}$-existential $\mathcal{L}$%
-embeddings is a positively $\mathcal{L}$-existential $\mathcal{L}$%
-embedding.

\item Let $(\{M_{\lambda }\}_{\lambda \in \Lambda },\{\theta _{\lambda ,\mu
}\}_{\lambda ,\mu \in \Lambda ,\lambda <\mu })$ be a direct system of $%
\mathcal{L}$-structures with positively $\mathcal{L}$-existential $\mathcal{L%
}$-embeddings $\theta _{\lambda \mu }\colon M_{\lambda }\rightarrow M_{\mu }$
for $\lambda <\mu $. If $M$ is the corresponding direct limit, then the
canonical $\mathcal{L}$-embedding of $M_{\lambda }$ into $M$, for $\lambda
\in \Lambda $, is positively $\mathcal{L}$-existential.

\item For $j=0,1$, let $(\{M_{\lambda }^{(j)}\}_{\lambda \in \Lambda
},\{\theta _{\lambda \mu }^{(j)}\}_{\lambda ,\mu \in \Lambda ,\lambda <\mu
}) $ be a direct system of $\mathcal{L}$-structures. Let $\{\eta _{\lambda
}\colon M_{\lambda }^{(0)}\rightarrow M_{\lambda }^{(1)}\}_{\lambda \in
\Lambda }$ be a family of intertwining positively $\mathcal{L}$-existential $%
\mathcal{L}$-embeddings. Then 
\begin{equation*}
\varinjlim \eta _{\lambda }\colon \varinjlim_{\lambda }M_{\lambda
}^{(0)}\rightarrow \varinjlim_{\lambda }M_{\lambda }^{(1)}
\end{equation*}%
is a positively $\mathcal{L}$-existential $\mathcal{L}$-embedding.
\end{enumerate}
\end{proposition}

The analogue of Remark \ref{Proposition:limit} holds for $\mathcal{L}$%
-existential $\mathcal{L}$-embeddings as well.

\subsection{Saturation of ultrapowers}

The notion of $\kappa $-good ultrafilter has been introduced in the case of
model theory for discrete structures in the classical monograph \cite[%
Section 6.1]{chang_model_1977}. The notion of $\kappa $-good filter can be
defined exactly as for ultrafilters. Theorem 6.1.4 of \cite{chang_model_1977}
shows that countably incomplete $\kappa $-good ultrafilters exist for any
cardinal $\kappa $. Every countably incomplete ultrafilter is $\aleph _{1}$%
-good; see \cite[Exercise 6.1.2]{chang_model_1977}. In particular, every
nonprincipal ultrafilter over a countable set is $\aleph _{1}$-good. The
same proof as \cite[Theorem 6.1.8]{chang_model_1977} shows the following.

\begin{proposition}
\label{Proposition:good}Suppose that $\kappa $ is a cardinal larger than the
density character of $\mathcal{L}$. Suppose that $M$ is an $\mathcal{L}$%
-structure and $\mathcal{U}$ is a countably incomplete $\kappa $-good
ultrafilter. Then $\prod_{\mathcal{U}}M$ is $\mathcal{L}$-$\kappa $%
-saturated.

If $\mathcal{F}$ is a countably incomplete $\kappa $-good filter, then $%
\prod_{\mathcal{F}}M$ is positively quantifier-free $\mathcal{L}$-$\kappa $%
-saturated.
\end{proposition}

Using Proposition \ref{Proposition:good} one can easily deduce the following
characterization of $\mathcal{L}$-existential $\mathcal{L}$-embeddings.

\begin{theorem}
\label{Theorem:existential-embedding} Let $A$ and $M$ be $\mathcal{L}$%
-structures, and let $\theta \colon A\rightarrow M$ be an $\mathcal{L}$%
-embedding. Let $\kappa $ be a cardinal greater than the density character
of $M$ and the density character of $\mathcal{L}$. The following assertions
are equivalent:

\begin{enumerate}
\item $\theta $ is an $\mathcal{L}$-existential $\mathcal{L}$-embedding;

\item there exist an $\mathcal{L}$-structure $N$ and an $\mathcal{L}$%
-embedding $\eta \colon M\rightarrow N$ such that $\eta \circ \theta \colon
A\rightarrow N$ is an $\mathcal{L}$-existential $\mathcal{L}$-embedding;

\item if $N$ is a quantifier-free\ $\mathcal{L}$-$\kappa $-saturated $%
\mathcal{L}$-structure, and $\theta _{N}\colon A\rightarrow N$ is an $%
\mathcal{L}$-embedding, then there exists an $\mathcal{L}$-embedding $\eta
\colon M\rightarrow N$ such that $\eta \circ \theta =\theta _{N}$;

\item for some (equivalently, any) countably incomplete ultrafilter $%
\mathcal{U}$, and for every separable $A_{0}\subset A$ and $M_{0}\subset M$
such that $\theta _{M}(A_{0})\subset M_{0}$, there exists an $\mathcal{L}$%
-embedding $\eta \colon M_{0}\rightarrow \prod_{\mathcal{U}}A$ such that $%
\eta \circ \theta _{M}|_{A_{0}}=\Delta _{A}|_{A_{0}}$.
\end{enumerate}
\end{theorem}

A similar characterization can be given for positively $\mathcal{L}$%
-existential $\mathcal{L}$-embeddings.

\begin{theorem}
\label{Theorem:positive-existential-embedding} Let $A$ and $M$ be $\mathcal{L%
}$-structures, and let $\theta \colon A\rightarrow M$ be an $\mathcal{L}$%
-embedding. Let $\kappa $ be a cardinal larger than the density character of 
$M$ and the density character of $\mathcal{L}$. The following assertions are
equivalent:

\begin{enumerate}
\item $\theta $ is a positively $\mathcal{L}$-existential $\mathcal{L}$%
-embedding;

\item there exist an $\mathcal{L}$-structure $N$ and an $\mathcal{L}$%
-morphism $\eta \colon M\rightarrow N$ such that $\eta \circ \theta \colon
A\rightarrow N$ is an $\mathcal{L}$-existential $\mathcal{L}$-embedding;

\item if $N$ is a quantifier-free\ positively quantifier-free $\mathcal{L}$-$%
\kappa $-saturated $\mathcal{L}$-structure, and $\theta _{N}\colon
A\rightarrow N$ is an $\mathcal{L}$-embedding, then there exists an $%
\mathcal{L}$-morphism $\eta \colon M\rightarrow N$ such that $\eta \circ
\theta =\theta _{N}$;

\item for some (equivalently, any) countably incomplete ultrafilter $%
\mathcal{F}$, and for every separable $A_{0}\subset A$ and $M_{0}\subset M$
such that $\theta _{M}(A_{0})\subset M_{0}$, there exists an $\mathcal{L}$%
-morphism $\eta \colon M_{0}\rightarrow \prod_{\mathcal{F}}A$ such that $%
\eta \circ \theta _{M}|_{A_{0}}=\Delta _{A}|_{A_{0}}$.
\end{enumerate}
\end{theorem}

We isolate the following fact, which is an immediate consequence of the
semantic characterization of positive $\mathcal{L}$-existential $\mathcal{L}$%
-embedding. If $F$ is a functor between two categories, we denote by $%
F(\theta )$ the image of a morphism $\theta $ under $F$. We regard the class
of $\mathcal{L}$-structures as a category with $\mathcal{L}$-morphisms as
morphisms.

\begin{proposition}
\label{Proposition:functor}Suppose that $\mathcal{L}^{(0)}$ and $\mathcal{L}%
^{(1)}$ are languages.\ Let $F$ be a functor from the category of $\mathcal{L%
}^{(0)}$-structures to the category of $\mathcal{L}^{(1)}$-structures.
Assume that $F$ preserves direct limits and that for any separable $\mathcal{%
L}^{(0)}$-structure $M$ and nonprincipal ultrafilter $\mathcal{U}$ over $%
\mathbb{N}$, there exists an $\mathcal{L}^{(1)}$-morphism $\rho _{M}\colon
F(\prod_{\mathcal{U}}M)\rightarrow \prod_{\mathcal{U}}F(M)$ such that $\rho
_{M}\circ F(\Delta _{M})=\Delta _{F(M)}$. If $A$ and $M$ are $\mathcal{L}%
^{(0)}$-structures in $\mathfrak{C}$ and $\theta _{M}\colon A\rightarrow M$
is a positive $\mathcal{L}^{(0)}$-existential $\mathcal{L}^{(0)}$-embedding,
then $F(\theta _{M})$ is a positive $\mathcal{L}^{(1)}$-existential $%
\mathcal{L}^{(1)}$--embedding.
\end{proposition}

\begin{proof}
Since $F$ preserves direct limits, it is enough to consider the case when $M$
is separable. In this case, the conclusion is the consequence of the first
assumption on the functor $F$ and Condition (3) of Theorem \ref%
{Theorem:existential-embedding}.
\end{proof}

In particular, Proposition \ref{Proposition:functor} applies when the
functor $F$ preserves both direct limits and ultraproducts.

\subsection{A more general framework\label{Subsection:general}}

In this subsection, we consider a more general framework, that will later
allow us to deal with not necessarily unital C*-algebras. For simplicity, we
restrict to the single-sorted case. Consider the language $\mathcal{L}$ that
contains

\begin{itemize}
\item a collection of function symbols,

\item a collection of relation symbols,

\item a directed collection $\mathcal{D}$ of pseudometric symbols, and

\item a distinguished collection $p(x)$ of quantifier-free positive
primitive $\mathcal{L}$-conditions.
\end{itemize}

The pseudometric symbols are to be interpreted as pseudometrics in a given $%
\mathcal{L}$-structure. We denote by $p^{+}(x)$ denotes the collection of
conditions $\varphi (x)\leq r+\varepsilon $ whenever $\varphi (x)\leq r$ is
a condition in $p$. We let $\wp _{\mathrm{fin}}(p^{+})$ be the collection of
finite subsets of $p^{+}$. We will assume that if $d_{0}(t_{0}(\overline{x}%
),t_{1}(\overline{x}))\leq r$ is a condition in $p$ for some $d_{0}\in 
\mathcal{D}$ and $\mathcal{L}$-terms $t_{0},t_{1}$, then the condition $%
d(t_{0}(\overline{x}),t_{1}(\overline{x}))\leq r$ also belongs to $p$ for
every $d\in \mathcal{D}$. Furthermore, we assume that for any relation
symbol $B$ in $\mathcal{L}$ and function symbol $f$ in $\mathcal{L}$, the
language $\mathcal{L}$ contains functions $\varpi _{B}\colon \mathbb{R}%
\rightarrow \mathbb{R}\times \mathcal{D}\times \wp _{\mathrm{fin}}(p^{+})$
and $\varpi _{f}\colon \mathbb{R}\times \mathcal{D}\rightarrow \mathbb{R}%
\times \mathcal{D}\times \wp _{\mathrm{fin}}(p^{+})$.

\begin{definition}
An $\mathcal{L}$-structure is a set $M$ endowed with an interpretation $%
B^{M} $ of any function or relation symbol in $\mathcal{L}$ such that:

\begin{enumerate}
\item the pseudometric symbols in $\mathcal{D}$ are interpreted as
pseudometrics on $M$;

\item for any $n$-ary relation symbol $B$ and any $\varepsilon _{0}>0$, if $%
\varpi _{B}( \varepsilon _{0}) =( \varepsilon _{1},d_{1},q_{1}) $, then for
any realizations $\overline{a},\overline{b}$ of $q_{1}$ in $M$ with $%
\max_{i}d_{1}^{M}( a_{i},b_{i}) \leq \varepsilon _{1}$, one has $\left\vert
B( \overline{a}) -B( \overline{b}) \right\vert \leq \varepsilon _{0}$;

\item for any $n$-ary function symbol $f$, any $\varepsilon >0$, and any $%
d_{0}\in \mathcal{D}$, if $\varpi _{f}(\varepsilon
_{0},d_{0},q_{0})=(\varepsilon _{1},d_{1},q_{1})$, then for any realizations 
$\overline{a},\overline{b}$ of $q_{1}$ in $M$ with $%
\max_{i}d_{1}^{M}(a_{i},b_{i})\leq \varepsilon _{1}$, then $f(\overline{a})$
is a realization of $q_{0}$ and $d_{0}^{M}(f(\overline{a}),f(\overline{b}%
))\leq \varepsilon _{0}$.
\end{enumerate}

The notions of $\mathcal{L}$-formulas and $\mathcal{L}$-types in this
setting are defined in the usual way.
\end{definition}

Suppose that $(M_{i})_{i\in I}$ is a collection of $\mathcal{L}$-structures,
and $\mathcal{F}$ is a filter over $I$. We let $M=\prod_{i\in I}M_{i}$ be
the cartesian product. For every $i\in I$ and $d\in \mathcal{D}$, define a
pseudometric $d^{M}$ on $M$ by 
\begin{equation*}
d^{M}((a_{i})_{i\in I},(b_{i})_{i\in I})=\limsup_{i\rightarrow \mathcal{F}%
}d^{M_{i}}(a_{i},b_{i})\text{.}
\end{equation*}%
Let now $M_{\mathcal{F}}$ be the quotient of $M$ by the equivalence relation 
$(a_{i})_{i\in I}\sim (b_{i})_{i\in I}$ if and only if $d^{M}((a_{i})_{i\in
I},(b_{i})_{i\in I})=0$ for every $d\in \mathcal{D}$. As before, we denote
by $\boldsymbol{a}$ the equivalence class of the collection $(a_{i})_{i\in
I} $. Set $\prod\nolimits_{\mathcal{F}}M_{i}$ to be the set of $\boldsymbol{a%
}\in M_{\mathcal{F}}$ such that for every $q\in \wp _{\mathrm{fin}}(p^{+})$
the set $\{i\in I\colon a_{i}\text{ is a realization of }q\}$
belongs to $\mathcal{F}$.

The interpretation in $\prod_{\mathcal{F}}M_{i}$ of function and relation
symbols from $\mathcal{L}$ is also defined in the usual way. For instance,
if $B$ is an $n$-ary relation symbol from $\mathcal{L}$ and $\boldsymbol{a}%
_{1},\ldots ,\boldsymbol{a}_{n}\in \prod_{\mathcal{F}}M_{i}$, then we let 
\begin{equation*}
B^{\prod_{\mathcal{F}}M_{i}}(\boldsymbol{a}_{1},\ldots ,\boldsymbol{a}%
_{n})=\limsup_{i\rightarrow \mathcal{F}}B(a_{1,i},\ldots ,a_{n,i})\text{.}
\end{equation*}%
Similarly, if $f$ is an $n$-ary function symbol from $\mathcal{L}$ and $%
\boldsymbol{a}_{1},\ldots ,\boldsymbol{a}_{n}\in \prod_{\mathcal{F}}M_{i}$,
we let $f^{\prod_{\mathcal{F}}M_{i}}(\boldsymbol{a}_{1},\ldots ,\boldsymbol{a%
}_{n})$ be the element with representative sequence 
\begin{equation*}
(f(a_{1,i},\ldots ,a_{n,i}))_{i\in I}\text{.}
\end{equation*}
The definition of $\mathcal{L}$-structure guarantee that these definitions
do not depend on the representatives, and define an $\mathcal{L}$-structure $%
\prod_{\mathcal{F}}M_{i}$, which we call the \emph{reduced product}.

\begin{remark}
\label{rmk:charactPosQtfFreeTyp} Let $M$ be an $\mathcal{L}$-structure, let $%
t(\overline{x})$ be a positive primitive quantifier free type. Let $\kappa $
be a cardinal larger than the density character of $M$ and the density
character of $\mathcal{L}$, and let $\mathcal{F}$ be a countably incomplete $%
\kappa $-good filter. It follows from Proposition \ref{Proposition:good}
that the following statements are equivalent:

\begin{enumerate}
\item $t(\overline{x})$ is realized in $\prod_{\mathcal{F}}M$

\item $t(\overline{x})$ is approximately realized in $\prod_{\mathcal{F}}M$

\item $t(\overline{x})\cup p(x_{1})\cup \cdots \cup p(x_{n})$ is
approximately realized in $M$.
\end{enumerate}
\end{remark}

By expanding the language to include constants to name elements of $\prod_{%
\mathcal{F}}M$, one can deduce that $\prod_{\mathcal{F}}M$ is positively
quantifier-free $\mathcal{L}$-$\kappa $-saturated. One may also consider
countably incomplete $\kappa $-good \emph{ultrafilters} (rather than
filters) and arbitrary quantifier-free types.


\providecommand{\MR}[1]{}
\providecommand{\bysame}{\leavevmode\hbox to3em{\hrulefill}\thinspace}
\providecommand{\MR}{\relax\ifhmode\unskip\space\fi MR }
\providecommand{\MRhref}[2]{%
  \href{http://www.ams.org/mathscinet-getitem?mr=#1}{#2}
}
\providecommand{\href}[2]{#2}

\end{document}